\documentclass[11pt]{article}

\usepackage[utf8]{inputenc}
\usepackage[english]{babel}
\usepackage{dutchcal}
\usepackage[T1]{fontenc}

\usepackage{amsmath}
\usepackage{amssymb}
\usepackage{amsthm}
\usepackage{tikz}
\usepackage{tikz-cd}
\usepackage{todonotes}
\usepackage{hyperref}

\usepackage{zref-clever}
\newcommand{\cref}[1]{\zcref{#1}}
\newcommand{\Cref}[1]{\zcref[S]{#1}}

\usepackage{quiver}

\usepackage{footmisc}
\usepackage{mathrsfs}
\usepackage{sseq}
\usepackage{spectralsequences}
\usepackage{extarrows}
\usepackage{enumitem}
\usepackage{stmaryrd}

\hypersetup{
  colorlinks   = true, 
  urlcolor     = blue, 
  linkcolor    = blue, 
  citecolor   = blue 
}
\newcommand{\llpar}{\llparenthesis}
\newcommand{\rrpar}{\rrparenthesis}

\newcommand{\ab}{\mathrm{ab}}
\newcommand{\abfgq}{\textup{ab-fgq}}

\newcommand{\cn}{\mathrm{cn}}
\newcommand{\cpl}{\mathrm{cpl}}

\newcommand{\Ell}{\mathrm{Ell}}
\newcommand{\Ellst}{\mathrm{Ell}^{\mathrm{st}}}
\newcommand{\et}{\mathrm{\acute{e}t}}
\newcommand{\fgq}{\mathrm{fgq}}
\newcommand{\fin}{\mathrm{fin}}
\newcommand{\fpqc}{\mathrm{fpqc}}

\newcommand{\gl}{\mathrm{gl}}
\newcommand{\loc}{\mathrm{loc}}
\newcommand{\LS}{\mathrm{LS}}
\newcommand{\ori}{\mathrm{or}}
\newcommand{\nc}{{\mathrm{nc}}}
\newcommand{\relSpec}{{\underline{\mathrm{Spec}}}}
\newcommand{\st}{\mathrm{st}}
\newcommand{\temp}{\mathrm{temp}}
\DeclareMathOperator{\Spec}{Spec}
\DeclareMathOperator{\Spet}{Sp\et}

\newcommand{\twoabgl}{\textup{(2,ab)-gl}}

\newcommand{\1}{\mathbf{1}}

\newcommand{\ELL}{{{E}{\ell\ell}}}

\DeclareMathOperator{\KU}{KU}

\newcommand{\gTMF}{{\mathbf{TMF}}}
\DeclareMathOperator{\ku}{ku}
\DeclareMathOperator{\KO}{KO}
\DeclareMathOperator{\ko}{ko}

\DeclareMathOperator{\res}{res}

\newcommand{\Sph}{\mathbf{S}}

\newcommand{\TMF}{\mathrm{TMF}}

\DeclareMathOperator{\CAlg}{CAlg}

\DeclareMathOperator{\Cat}{Cat}
\newcommand{\PrL}{{\mathrm{Pr}^L}}
\newcommand{\PrLst}{{\mathrm{Pr}^L_\st}}

\DeclareMathOperator{\DM}{DM}
\newcommand{\FG}{{\mathrm{FG}}}
\DeclareMathOperator{\Fun}{Fun}

\DeclareMathOperator{\Glo}{Glo}
\DeclareMathOperator{\Orb}{Orb}

\DeclareMathOperator{\Ind}{Ind}
\DeclareMathOperator{\Mod}{Mod}

\DeclareMathOperator{\QCoh}{QCoh}

\DeclareMathOperator{\Shv}{Shv}
\DeclareMathOperator{\Stk}{Stk}

\DeclareMathOperator{\SpSch}{SpSch}
\DeclareMathOperator{\Sp}{Sp}
\DeclareMathOperator{\SpDM}{SpDM}
\newcommand{\Spc}{\mathcal{S}}

\DeclareMathOperator{\Top}{Top}

\newcommand{\colim}{\mathrm{colim}\,}

\DeclareMathOperator{\fib}{fib}

\newcommand{\id}{\mathrm{id}}
\DeclareMathOperator{\Map}{Map}
\DeclareMathOperator{\map}{map}

\DeclareMathOperator{\Hom}{Hom}
\newcommand{\HOM}{\underline{\mathrm{Hom}}}
\newcommand{\op}{\mathrm{op}}
\newcommand{\dual}[1]{{{\widehat{#1}}}}

\newcommand{\B}{\mathbf{B}}

\newcommand{\calC}{\mathcal{C}}

\newcommand{\calD}{\mathcal{D}}

\newcommand{\E}{\mathbf{E}}
\newcommand{\sfE}{\mathsf{E}}
\newcommand{\EE}{\mathcal{E}}
\newcommand{\calE}{\mathcal{E}}

\newcommand{\F}{\mathbf{F}}
\newcommand{\calF}{\mathcal{F}}
\newcommand{\G}{\mathbf{G}}
\newcommand{\sfG}{\mathsf{G}}

\newcommand{\M}{\mathsf{M}}
\newcommand{\sfM}{\mathsf{M}}
\newcommand{\calM}{\mathcal{M}}
\newcommand{\calN}{\mathcal{N}}

\renewcommand{\O}{\mathcal{O}}

\renewcommand{\P}{\mathbf{P}}
\newcommand{\calP}{\mathcal{P}}
\newcommand{\calQ}{\mathcal{Q}}

\newcommand{\Q}{\mathbf{Q}}

\newcommand{\sfS}{\mathsf{S}}

\newcommand{\sfT}{\mathsf{T}}

\newcommand{\sfU}{\mathsf{U}}
\newcommand{\sfV}{\mathsf{V}}
\newcommand{\U}{\mathsf{U}}
\newcommand{\sfW}{\mathsf{W}}

\newcommand{\X}{\mathsf{X}}
\newcommand{\sfX}{\mathsf{X}}
\newcommand{\calX}{\mathcal{X}}

\newcommand{\Y}{\mathsf{Y}}
\newcommand{\calY}{\mathcal{Y}}
\newcommand{\sfY}{\mathsf{Y}}

\newcommand{\Z}{\mathbf{Z}}
\newcommand{\sfZ}{\mathsf{Z}}

\newcommand{\Ga}{\Gamma}

\newcommand{\gGa}{\underline{\Ga}}

\NewDocumentCommand{\newzctheorem}{mO{#1}m}{
  \newtheorem{#1}[sharedtheoremcounter]{#3}
    \AddToHook{env/#1/begin}{
      \zcsetup{countertype={sharedtheoremcounter=#2}}}
}

\numberwithin{equation}{sharedtheoremcounter}

\newzctheorem{theorem}{Theorem}
\zcRefTypeSetup{theorem}{
  Name-sg = Th.\hspace{-0.4em} ,
  name-sg = th.\hspace{-0.4em} ,
  Name-pl = Ths.\hspace{-0.4em} ,
  name-pl = ths.\hspace{-0.4em} ,
}

\newzctheorem{cor}{Corollary}
\zcRefTypeSetup{cor}{
  Name-sg = Cor.\hspace{-0.4em} ,
  name-sg = cor.\hspace{-0.4em} ,
  Name-pl = Cors.\hspace{-0.4em} ,
  name-pl = cors.\hspace{-0.4em} ,
}

\newtheorem{theoremalph}{Theorem}

\zcRefTypeSetup{theoremalph}{
  Name-sg = Th.\hspace{-0.3em} ,
  name-sg = th.\hspace{-0.3em} ,
  Name-pl = Ths.\hspace{-0.3em} ,
  name-pl = ths.\hspace{-0.3em} ,
}

\zcRefTypeSetup{coralph}{
  Name-sg = Cor.\hspace{-0.4em} ,
  name-sg = cor.\hspace{-0.4em} ,
  Name-pl = Cors.\hspace{-0.4em} ,
  name-pl = cors.\hspace{-0.4em} ,
}

\newzctheorem{lemma}{Lemma}
\zcRefTypeSetup{lemma}{
  Name-sg = Lm.\hspace{-0.4em},
  name-sg = lm.\hspace{-0.4em} ,
  Name-pl = Lms.\hspace{-0.4em} ,
  name-pl = lms.\hspace{-0.4em} ,
}

\newzctheorem{prop}{Proposition}
\zcRefTypeSetup{prop}{
  Name-sg = Pr.\hspace{-0.4em} ,
  name-sg = pr.\hspace{-0.4em} ,
  Name-pl = Prs.\hspace{-0.4em} ,
  name-pl = prs.\hspace{-0.4em} ,
}

\theoremstyle{definition}
\newzctheorem{mydef}{Definition}
\zcRefTypeSetup{mydef}{
  Name-sg = Df.\hspace{-0.2em} ,
  name-sg = df.\hspace{-0.2em} ,
  Name-pl = Dfs.\hspace{-0.2em} ,
  name-pl = dfs.\hspace{-0.2em} ,
}

\newzctheorem{example}{Example}
\zcRefTypeSetup{example}{
  Name-sg = Ex.\hspace{-0.4em} ,
  name-sg = ex.\hspace{-0.4em} ,
  Name-pl = Exs.\hspace{-0.4em} ,
  name-pl = exs.\hspace{-0.4em} ,
}

\newzctheorem{remark}{Remark}
\zcRefTypeSetup{remark}{
  Name-sg = Rmk.\hspace{-0.4em} ,
  name-sg = rmk.\hspace{-0.4em} ,
  Name-pl = Rmks.\hspace{-0.4em} ,
  name-pl = rmks.\hspace{-0.4em} ,
}

\newzctheorem{warning}{Warning}
\zcRefTypeSetup{warning}{
  Name-sg = Wrn.\hspace{-0.4em} ,
  name-sg = wrn.\hspace{-0.4em} ,
  Name-pl = Wrns.\hspace{-0.4em} ,
  name-pl = wrns.\hspace{-0.4em} ,
}

\newzctheorem{notation}{Notation}
\zcRefTypeSetup{notation}{
  Name-sg = Nt.\hspace{-0.4em} ,
  name-sg = nt.\hspace{-0.4em} ,
  Name-pl = Nts.\hspace{-0.4em} ,
  name-pl = nts.\hspace{-0.4em} ,
}

\newzctheorem{construction}{Construction}
\zcRefTypeSetup{construction}{
  Name-sg = Con.\hspace{-0.4em} ,
  name-sg = con.\hspace{-0.4em} ,
  Name-pl = Cons.\hspace{-0.4em} ,
  name-pl = cons.\hspace{-0.4em} ,
}

\newzctheorem{variant}{Variant}
\zcRefTypeSetup{variant}{
  Name-sg = Var.\hspace{-0.4em} ,
  name-sg = var.\hspace{-0.4em} ,
  Name-pl = Vars.\hspace{-0.4em} ,
  name-pl = vars.\hspace{-0.4em} ,
}

\zcRefTypeSetup{section}{
  Name-sg = \textsection\negthickspace,
  name-sg = \textsection\negthickspace,
  Name-pl = \textsection\negthickspace,
  name-pl = \textsection\negthickspace,
}
\zcRefTypeSetup{subsection}{
  Name-sg = \textsection\negthickspace,
  name-sg = \textsection\negthickspace,
  Name-pl = \textsection\negthickspace,
  name-pl = \textsection\negthickspace,
}
\zcRefTypeSetup{subsubsection}{
  Name-sg = \textsection\negthickspace,
  name-sg = \textsection\negthickspace,
  Name-pl = \textsection\negthickspace,
  name-pl = \textsection\negthickspace,
}
\zcRefTypeSetup{appendix}{
  Name-sg = \textsection\negthickspace,
  name-sg = \textsection\negthickspace,
  Name-pl = \textsection\negthickspace,
  name-pl = \textsection\negthickspace,
}

\newcommand{\triangleleftend}[1]{%
  \AddToHook{env/#1/begin}{%
    \pushQED{\qed}%
    \renewcommand{\qedsymbol}{\ensuremath{\triangleleft}}%
  }%
  \AddToHook{env/#1/end}{\popQED}%
}
\triangleleftend{mydef}
\triangleleftend{remark}
\triangleleftend{example}
\triangleleftend{construction}
\triangleleftend{warn}
\triangleleftend{warning}
\triangleleftend{notation}
\triangleleftend{conjecture}
\triangleleftend{question}

\usepackage{parskip}

\makeatletter
\def\thm@space@setup{%
  \thm@preskip=\parskip \thm@postskip=0pt
}
\makeatother
\begingroup
    \makeatletter
    \@for\theoremstyle:=mydef,remark,plain,theorem,defn\do{%
        \expandafter\g@addto@macro\csname th@\theoremstyle\endcsname{%
            \addtolength\thm@preskip\parskip
            }%
        }
\endgroup
\usetikzlibrary{spath3}
\tikzset{between/.style n args={2}{/tikz/spath/at end path construction={
    \tikzset{spath/split at keep middle={current}{#1}{#2}}
}}}

\begin{document}
\title{
On equivariant affineness in elliptic cohomology
}

\author{Jack Morgan Davies\footnote{\url{jack.davies@ruhr-uni-bochum.de}}, Sil Linskens\footnote{\url{sil.linskens@mathematiks.uni-regensburg.de}}, and Leonard Tokic\footnote{\url{leonard.tokic@ruhr-uni-bochum.de}}}
\date{\today}

\maketitle

\begin{abstract}
    We investigate {\it equivariant affineness} in spectral algebraic geometry: the phenomenon where $\infty$-categories of quasi-coherent sheaves of a spectral stack are equivalent to $\infty$-categories of modules over an equivariant ring spectrum. In the nonequivariant case, this is known as 0-affineness, and has been studied by Gaitsgory, Mathew--Meier, amongst others. Allowing for equivariant phenomena opens up a much wider class of spectral stacks to be studied through the lens of spectra.    
    
    The spectral stacks at the heart of our analysis are oriented elliptic curves, their associated Hilbert schemes, and related constructions. When working over the universal base, so the moduli stack of oriented elliptic curves, we show that these spectral stacks are $U(1)$-, $U(n)$-, and $SU(n)$-equivariantly affine, with associated $G$-ring spectrum Gepner--Meier's {\it equivariant topological modular forms} $\mathrm{TMF}_G$. Our results also encompass the groups $Sp(n)$ as well as all finite products of such groups. Our most general result is a relative statement over an arbitrary base stack. These results can be viewed as an integral and global refinement of Greenlees' algebraic model for modules over rational $U(1)$-equivariant elliptic cohomology.
\end{abstract}  

\setcounter{tocdepth}{2}
\tableofcontents

\section{Introduction}

A presentably symmetric monoidal stable $\infty$-category $(\calC, \otimes, \1)$ is \emph{affine} if the functor
\begin{equation}\label{eq:maps_out_of_unit}
\calC \to \Mod_R(\Sp); \qquad X \mapsto \map_\calC(\1, X)
\end{equation}
is an equivalence, where $R = \map_\calC(\1, \1)$ is the endomorphism $\E_\infty$-ring of the unit $\1$. The name is inspired from the situation where $\calC = \Mod_A(\Sp) = \QCoh(\Spec A)$ is the $\infty$-category of modules over an $\E_\infty$-ring $A$, in which case $\map_{\Mod_A(\Sp)}(A,A) = A$ and (\ref{eq:maps_out_of_unit}) is an equivalence. Establishing affineness for interesting $\calC$ can be very powerful, allowing one to pass information back and forth through the equivalence. This can benefit the study of both $\calC$ and $R$.

\begin{example}\label{ex:MM_mellor_affineness_intro}
On the geometric side, consider the $\infty$-category of quasi-coherent sheaves $\QCoh(\M_\Ell^\ori)$ on the \emph{spectral moduli stack of oriented elliptic curves}. This $\infty$-category is affine by a famous result of Mathew--Meier \cite{akhilandlennart}, and the functor (\ref{eq:maps_out_of_unit}) is given by the global sections functor
\[\Ga \colon \QCoh(\M_\Ell^\ori) \xrightarrow{\simeq} \Mod_{\Ga(\M_\Ell^\ori)}(\Sp).\]
Stacks $\X$ such that $\QCoh(\X)$ is affine are called \emph{0-affine}. The $\E_\infty$-ring of global sections $\Ga(\O_{\M_\Ell^\ori}) = \Ga(\M_\Ell^\ori)$ is Hopkins' ring $\TMF$ of \emph{topological modular forms} \cite{hopkinsfirsttmficm,hopkinssecondtmficm}. The affineness of $\QCoh(\M_\Ell^\ori)$ can be used to reduce questions concerning $\TMF$ to local questions on $\M_\Ell^\ori$, such as Galois theory \cite[\textsection 7]{akhilandlennart} and duality \cite{tmfwls}, and is also a key step in computations \cite[Th.B]{smfcomputation}. Conversely, it also shows that $\QCoh(\M_\Ell^\ori)$ is categorically extremely simple: it is compactly generated by its unit $\O_{\M_\Ell^\ori}$. This is perhaps a surprising fact, as $\infty$-categories of quasi-coherent sheaves on stacks often do \textbf{not} have a compact unit 
and can also easily fail to have a single generator.
\end{example}

\begin{example}\label{ex:MNN_ktheory_affineness_intro}
On the equivariant side, consider the $\infty$-category of modules $\Mod_{\KU_G}(\Sp_G)$ of Atiyah--Segal's \emph{$G$-equivariant topological $K$-theory spectrum} $\KU_G$ in $G$-spectra, where $G$ is a compact Lie group. In this case, the functor (\ref{eq:maps_out_of_unit}) sends a $\KU_G$-module in $G$-spectra $M$ to its $G$-fixed points $M^G$. For connected $G$ with torsion-free $\pi_1 G$, this $G$-fixed points functor is an equivalence and consequently $\Mod_{\KU_G}(\Sp_G)$ is affine. For tori $G=U(1)^r$, this is a result of Greenlees--Shipley \cite{greenlees_shipley_unipotence}, and the case for more general $G$ is due to Mathew--Naumann--Noel \cite{nilpotenceanddescentinequivariant}.
It is also remarkable in this case that $\Mod_{\KU_G}(\Sp_G)$ is affine, as although the $\infty$-category of $G$-spectrum $\Sp_G$ is compactly generated, it certainly requires more than one generator, often indexed by the orbits $\Sigma^\infty_+ G/H$ ranging over all conjugacy classes of closed subgroups $H \leq G$.
\end{example}

The goal of this article is to simultaneously unite and generalise the previous two examples under the banner of \emph{equivariant affineness}. From an algebro-geometric perspective, this shows that many $\infty$-categories of quasi-coherent sheaves are well-behaved, i.e.~have compact unit and generators parametrised by the orbit spaces $G/H$. For example, this will apply to $\calE$, the \emph{universal oriented elliptic curve}, where $\QCoh(\calE)$ is definitely \textbf{not} affine; see \Cref{rmk:universal_case_not_affine}. From an equivariant perspective, this shows the certain $G$-equivariant ring spectra can be studied using the local geometry of the associated stacks. Building from the previous example, the $U(1)$-equivariant ring spectrum associated with $\calE$ is \emph{equivariant topological modular forms} $\TMF_{U(1)}$ of Gepner--Meier \cite{davidandlennart}, and so equivariant affineness in this case yields an equivalence
\begin{equation}\label{eq:opening_intro}\Ga_{U(1)} \colon \QCoh(\calE) \xrightarrow{\simeq} \Mod_{\TMF_{U(1)}}(\Sp_{U(1)})\end{equation}
through a $U(1)$-equivariant global sections functor; see \Cref{maintheorem_universal} below. This shows that the algebra of $\TMF_{U(1)}$ is determined by the geometry of $\calE$. The use of $U(1)$-equivariant spectra compensates for the fact that the elliptic curve $\calE$ is not affine. This particular equivalence can also be viewed as an integral incarnation of Greenlees' algebraic model for rational $U(1)$-equivariant elliptic cohomology \cite{greenlees_rational_s1_elliptic}. 

Our most general result, see \Cref{maintheorem_arbitrary_base}, applies to arbitrary oriented elliptic curves as well as many related stacks such as their associated Hilbert schemes of $n$-points, quotients of symmetric products, and their finite products, all over a general base stack. This produces examples of $G$-equivariant affine $\infty$-categories of quasi-coherent sheaves for many more connected compact Lie groups $G$ other than $U(1)$. In order to state these results, we need to recall some foundations of equivariant elliptic cohomology.

\subsection*{Equivariant elliptic cohomology}
As pioneered by Grojnowski \cite{grojnowski_equivariant_elliptic}, Greenlees \cite{greenlees_rational_s1_elliptic}, Lurie \cite{lurieecsurvey,ec3}, and Gepner--Meier \cite{davidandlennart,gm2}, \emph{equivariant elliptic cohomology} uses the arithmetic and geometry of elliptic curves to produce powerful and subtle $G$-equivariant cohomology theories for compact Lie groups $G$. The construction of equivariant elliptic cohomology is similar to that of rational Bredon cohomology or equivariant topological $K$-theory, in the sense that all are born out of abelian schemes: one uses the additive group $\G_a$ for Bredon cohomology and the multiplicative group $\G_m$ for topological $K$-theory. 

To construct equivariant elliptic cohomology following Gepner--Meier, one starts with the data of an (oriented) elliptic curve $\sfE$ over a spectral stack $\M$. For each abelian compact Lie group $G$, one then defines a spectral scheme $\sfE_G$ as the homomorphism scheme
\[\sfE_G = \Hom_\M(\dual{G}, \sfE),\]
where $\dual{G} = \Hom(G,U(1))$ is the Pontryagin dual of $G$; in \cite{davidandlennart}, these $\sfE_G$ were written as $\Ell(\mathbf{B}G)$. In particular, $\sfE_{\Z/n} = \sfE[n]$ is the $n$-torsion of $\sfE$, and $\sfE_{U(1)} = \sfE$ recovers the elliptic curve itself. The global sections of $\sfE_G$ should be the $G$-fixed points of the equivariant cohomology theory associated with $\sfE$. Upgrading this construction to a genuine $G$-spectrum is challenging. Gepner--Meier \cite{davidandlennart} achieve this for $G=U(1)$ by relying on the the geometry of $\sfE_{U(1)} = \sfE$. The assignment $G\mapsto \sfE_G$ for abelian compact Lie groups is then generalised in \cite{gm2} to nonabelian Lie groups by way of a subtle Kan extension; this construction is hinted at in Lurie's survey \cite[Rmk.5.1]{lurieecsurvey} as well as Meier's Oberwolfach report \cite[p.1972-4]{integral_elliptic_oberwolfach}, and we give an overview in \Cref{rmk:construction_of_Ellst}.

For general compact Lie groups $G$, the stacks $\sfE_G$ are hard to control. Gepner--Meier identify the class of \emph{applicable} compact Lie groups (\Cref{df:applicableCLG}), including unitary groups, special unitary groups, symplectic groups, and finite products thereof, for which the assignment $G \mapsto \sfE_G$ has particularly nice properties:
\begin{itemize}
    \item The functor from the orbit category of $G$ to $\QCoh(\sfE_G)$ sending $G/H$ to $p_\ast \O_{\sfE_H}$ induced by $p\colon \sfE_H \to \sfE_G$, extends (and dualises) to a symmetric monoidal colimit-preserving \emph{equivariant elliptic homology functor}
    \begin{equation}\label{eq:equiv_ell_hom_intro} \O_{\sfE_G} \otimes - \colon \Sp_G \to \QCoh(\sfE_G);\end{equation}
    see \Cref{ssec:EEcohomology} for more details.
    \item For $G=U(n)$, $SU(n)$, or $Sp(n)$, the stacks $\sfE_G$ are relative schemes over their base, with underlying classical scheme given by an elliptic curve, Hilbert scheme, or a related construction; see \Cref{thm:extra-applicable}.
\end{itemize}

The value of the right adjoint to $\O_{\sfE_G}\otimes -$ on the unit $\O_{\sfE_G}$ of $\QCoh(\sfE_G)$ is written as $\Ga_G \sfE$; this is the $G$-equivariant ring spectrum, or \emph{$G$-equivariant elliptic cohomology theory}, associated with $\sfE$. As this right adjoint is lax monoidal, it naturally factors as the \emph{equivariant global sections functor}
\[\underline{\Ga}(\sfE_G, -) \colon \QCoh(\sfE_G) \to \Mod_{\Ga_G \sfE}(\Sp_G).\]
Over an affine base $\M = \Spec A$, we write $\Ga_G \sfE = A_G$, suppressing the elliptic curve, and over the universal base $\M_\Ell^\ori$ we write $\Ga_G(\calE) = \TMF_G$ for \emph{equivariant topological modular forms}, also suppressing the universal oriented elliptic curve $\calE$.

We can only define the $G$-equivariant global sections functor when $G$ is applicable, however, one can still define $\calE_G$, $\O_{\sfE_G} \otimes -$, and $\TMF_G$ for general compact Lie groups $G$. There are a few ways to achieve this, which we review in \Cref{thm:Ell_geo_nice_prop}, \Cref{rmk:general_cohomology_definition,rmk:TMFG_for_general_G}.

The main results of this article identify various $G$ and $\sfE$ such that $\QCoh(\sfE_G)$ is equivariantly affine, meaning the functor $\underline{\Ga}(\sfE_G,-)$ above is an equivalence.

\subsection*{Main results}
The first result applies to the universal oriented elliptic curve $\calE$ over $\M_\Ell^\ori$.

\begin{theoremalph}\label{maintheorem_universal}
    Let $G = \prod_i U(n_i) \times \prod_j SU(n_j) \times \prod_k Sp(n_k)$ be a finite product of unitary, special unitary, and symplectic groups. Then the $G$-global sections functor
    \[\underline{\Ga}(\calE_G,-) \colon \QCoh(\calE_G) \xrightarrow{\simeq} \Mod_{\TMF_G}(\Sp_G)\]
    is an equivalence of symmetric monoidal $\infty$-categories.
\end{theoremalph}

In particular, if $G=U(1)$ is the torus of rank 1, then this is exactly the equivalence (\ref{eq:opening_intro}). This can be viewed as both a computation of the $\infty$-category of quasi-coherent sheaves on the universal elliptic curve, and a tool for studying $\TMF_{U(1)}$-modules locally on $\calE$. By construction, one can match the generators $(e_n)_\ast \O_{\calE[n]}$ of $\QCoh(\calE)$, where $e_n \colon \calE[n] \to \calE$ is the inclusion of the $n$-torsion, with the generators $\TMF_{U(1)} \otimes (U(1)/C_n)_+$ of $\Mod_{\TMF_{U(1)}}(\Sp_{U(1)})$. There are also similar translations for other groups; see \Cref{rmk:identification-of-divisors} for some more examples.

\begin{remark}
Let us point out immediately that equivariant affineness is \textbf{false} for \emph{finite abelian} groups. Indeed, by construction one has $\calE_{\Z/n} = \calE[n] \subseteq \calE$ as the $n$-torsion of the universal oriented elliptic curve. Grojnowski notes that the equivariant elliptic homology functor $\Sp_{\Z/2} \to \QCoh(\calE[2])$ is not symmetric monoidal in this case, which obstructs equivariant affineness. One can also note that $\calE[n]$ is finite flat, and in particular affine, over $\M_\Ell^\ori$, hence \cite[Th.A]{reconstruction} shows that the \emph{nonequivariant} global sections functor
\[\QCoh(\calE_{\Z/n}) = \QCoh(\calE[n]) \xrightarrow{\simeq} \Mod_{\Ga(\calE[n])}(\Sp) = \Mod_{\TMF^{\Z/n}}(\Sp)\]
is an equivalence, where $\TMF^{\Z/n}$ indicates the genuine $\Z/n$-fixed points. The $\infty$-category $\Mod_{\TMF^{\Z/n}}(\Sp)$ is not equivalent to $\Mod_{\TMF_{\Z/n}}(\Sp_{\Z/n})$ defined by \cite{gepner2024global2ringsgenuinerefinements} or \cite{temperedglobal} for $n\geq 2$. These references instead highlight the need to use Lurie's $\infty$-categories of \emph{tempered local systems} as a replacement for quasi-coherent sheaves when dealing with finite groups to best model modules over \(\TMF_{\Z/n}\) geometrically. Indeed, we expect that \Cref{maintheorem_universal} is a phenomenon that can only occur for \emph{connected} compact Lie groups.
\end{remark}

\begin{remark}\label{rmk:universal_case_not_affine}
Let us emphasise that $\QCoh(\calE)$ is $U(1)$-equivariantly affine, but \textbf{not} affine. Indeed, if it were, then the unit map $\O_\calE \to [n]_\ast \O_\calE$ induced by the $n$-fold multiplication map $[n] \colon \calE \to \calE$ would be an equivalence after inverting $n$. Indeed, applying global sections induces the endomorphism of $\TMF\oplus \TMF[1]$ given by the identity on the first factor and multiplication by $n$ on the second; this can be seen with a descent spectral sequence computation. On the other hand, pulling this unit map back against the identity section $e\colon \M_\Ell^\ori \to \calE$ induces the inclusion of a summand $\O_{\M_\Ell^\ori} \to \O_{\M_\Ell^\ori}^{\oplus n^2}$, which is never an equivalence for $n\geq 2$.
\end{remark}

\begin{remark}
The work of Konovalov--Prikhodko \cite{temperedlocalsystems_Konovalov_Prikhodko} generalises Lurie's tempered local systems of \cite{ec3} to compact Lie groups, among other things. In particular, their categories $\LS^\temp(BG,\mathcal{TMF})$ specialise to Lurie's tempered local systems $\mathrm{LocSys}_{\calE[\P^\infty]}(BG)$ when $G$ is finite, and to $\QCoh(\calE_G)$ when $G=U(1)^r$ is a torus. Combining \cite[Th.15.27]{gepner2024global2ringsgenuinerefinements} with the derived defect base of $\TMF$ from \cite[Ex.1.4.2]{temperedglobal}, one can identify $\mathrm{LocSys}_{\calE[\P^\infty]}(BG)$ with $\Mod_{\TMF_G}(\Sp_G)$ for all finite groups $G$. Together with \Cref{maintheorem_universal} for $G=U(1)^r$, this shows that Konovalov--Prikhodko's $\LS^\temp(BG,\mathcal{TMF})$ agrees with $\Mod_{\TMF_G}(\Sp_G)$ for tori and finite groups. This suggests that this generalised category of tempered local systems might be the best model of $\Mod_{\TMF_G}(\Sp_G)$ for more general $G$.
\end{remark}



We also generalise \Cref{maintheorem_universal} to work over an arbitrary base $\M$. To formulate this, we write
\[\QCoh_G(\M) = \QCoh(\M) \otimes \Sp_G 
\]
for any compact Lie group $G$. By $\QCoh(\M)$-linearising the equivariant elliptic homology functor (\ref{eq:equiv_ell_hom_intro}) and taking right adjoints, one obtains an \emph{equivariant pushforward} functor
\[p_{\ast,G} \colon \QCoh(\sfE_G) \to \QCoh_G(\M),\]
where $p_G \colon \sfE_G \to \M$ is the structure map; see \Cref{df:g-pushforward} for more details. The following is a formal consequence of \Cref{maintheorem_universal}.

\begin{theoremalph}\label{maintheorem_arbitrary_base}
    Let $p\colon \sfE \to \M$ be an oriented elliptic curve over a spectral stack $\M$. Then for $G$ as in \Cref{maintheorem_universal}, the $G$-equivariant pushforward functor
    \[p_{\ast,G} \colon \QCoh(\sfE_G) \xrightarrow{\simeq} \Mod_{p_{\ast,G} \O_{\sfE_G}}(\QCoh_G(\M))\]
    is an equivalence.
\end{theoremalph}

When $\M$ is $0$-affine, the right-hand side is again simply modules in $G$-spectra over $\Gamma_G\sfE$. For example, if $\sfT$ is the \emph{oriented Tate curve} over $\Spec \KU\llpar q \rrpar / C_2$ of \cite[Th.A]{globaltate}, we have an equivalence
\[\Ga_G \colon \QCoh(\sfT_G) \xrightarrow{\simeq} \Mod_{\KO\llpar q \rrpar_G}(\Sp_G)\]
for those $G$ of \Cref{maintheorem_universal}, where $\KO\llpar q \rrpar_G$ is the $G$-equivariant $\E_\infty$-ring of \emph{real smooth Tate $K$-theory} from \cite[Df.7.4]{globaltate}.

There are many other examples and applications of \Cref{maintheorem_universal,maintheorem_arbitrary_base}. For instance, given an oriented elliptic curve $\sfE$ over a nonconnective qcqs spectral scheme $\sfS$ and $G$ as in \Cref{maintheorem_universal}, the $\infty$-category $\Mod_{\Ga_G \sfE}(\Sp_G)$ is generated by a single compact object; this is a consequence of \Cref{maintheorem_arbitrary_base} together with \Cref{ex:schemes_are_passable}. As another example, we see that the Galois groupoid $\pi_{\leq1}(\Mod_{\TMF_{U(1)}}(\Sp_{U(1)}))$ is equivalent to the Galois groupoid of the classical universal elliptic curve $\pi_{\leq 1} \calE^\heartsuit$; this follows from \Cref{maintheorem_universal} combined with (the proof of) \cite[Th.10.4]{mathew_Galois} and the fact that elliptic curves are regular. As yet another example, if $G$ and $K$ are two compact Lie groups as in \Cref{maintheorem_universal}, then we have a Künneth theorem for $\infty$-categories of modules over $G\times K$-equivariant topological modular forms:
\[\Mod_{\TMF_G}(\Sp_G) \otimes_{\Mod_\TMF} \Mod_{\TMF_K}(\Sp_K) \xrightarrow{\simeq} \Mod_{\TMF_{G\times K}}(\Sp_{G\times K})\]
This is a combination of \Cref{maintheorem_universal} and \Cref{ex:MM_mellor_affineness_intro} together with \Cref{pr:general_kuenneth}.

\subsection*{Proof outline}
Let us give an outline of the proof strategy for the above results; so roughly \Cref{sec:affine_case} in reverse.

Reducing \Cref{maintheorem_arbitrary_base} over an arbitrary base $\M$ to the universal case of \Cref{maintheorem_universal} over $\M_\Ell^\ori$ is formal. For simplicity, suppose that the base $\M$ is 0-affine, so $\QCoh(\M)$ is affine. One then has equivalences
 \[\QCoh(\sfE_G) \simeq \QCoh(\calE_G) \underset{\QCoh(\M_\Ell^\ori)}{\otimes} \QCoh(\M) \xrightarrow{\simeq} \Mod_{\TMF_G}(\Sp_G) \underset{\Mod_{\TMF}(\Sp)}{\otimes} \Mod_{\Ga(\M)}(\Sp)\]
 \[\simeq \Mod_{\TMF_G \otimes_\TMF \Ga(\M)} (\Sp_G) \simeq \Mod_{\Ga_G \sfE}(\Sp_G),\]
 the first from a Künneth formula for quasi-coherent sheaves on perfect stacks (\Cref{pr:general_kuenneth}) together with the base change equivalence of $G$-spectra $\TMF_G \otimes_\TMF \Ga(\M) \simeq \Ga_G \sfE$ (\Cref{pr:base_change_for_equivariant_theories}) for the last; also note \Cref{maintheorem_universal} and \Cref{ex:MM_mellor_affineness_intro} coming into the second equivalence. This reduces us to proving \Cref{maintheorem_universal}.

To prove the universal case over the moduli stack $\M_\Ell^\ori$, we can choose an étale cover $\Spec A \to \M_\Ell^\ori$, where $A$ is Noetherian and strictly even-periodic, and write $\Spec A^\bullet$ for the associated Čech nerve such that $\colim \Spec A^\bullet \simeq \M_\Ell^\ori$ and $\sfE_G^\bullet$ for the base change of the universal oriented elliptic curve $\calE_G$ to $\Spec A^\bullet$. We now have functors
\[\QCoh(\calE_G) \xrightarrow{\simeq} \lim \QCoh(\sfE^\bullet_G) \xrightarrow{\lim \Ga^\bullet_G} \lim \Mod_{A^\bullet_G}(\Sp_G) \xleftarrow{\simeq} \Mod_{\TMF_G}(\Sp_G),\]
the first being an equivalence as $\Spec A$ covers $\M_\Ell^\ori$ and the second is formal (\Cref{thm:step_to_universal_case}). We are then reduced to proving that
\[\underline{\Ga}(\sfE_G, -) \colon \QCoh(\sfE_G) \to \Mod_{A_G}(\Sp_G)\]
is an equivalence when working over strictly even-periodic Noetherian $\E_\infty$-rings $A$.

To this end, we simplify the situation in the direction of the group $G$. First, we show that if $\underline{\Ga}(\sfE_G, -)$ and $\underline{\Ga}(\sfE_K, -)$ are equivalences for two (applicable) compact Lie groups $G$ and $K$, then $\underline{\Ga}(\sfE_{G\times K}, -)$ is also an equivalence, reducing us from products $G=\prod_i U(n_i) \times \prod_j SU(n_j) \times \prod_k Sp(n_k)$ to each of the individual factors (\Cref{pr:propegation_to_products}). This relies on $\sfE_G$ being a \emph{perfect stack} for each $G$ as in \Cref{maintheorem_universal}.

We then tackle each of the families $U(n)$, $SU(n)$, and $Sp(n)$ separately via an inductive argument: the base case $SU(1)=1$ is trivial and $Sp(1) = SU(2)$ follows if we prove the special unitary case first. The case $U(1)$ is the first non-trivial case, proven in  (\Cref{thm:ellipticU(1)0affineness}). We now outline the proof. As we explain below, the method of proof for $U(1)$ admits a generalisation to an inductive argument for the other families (\Cref{lem:inductive_step_general}).

Let $p\colon \sfE \to \Spec A$ be an oriented elliptic curve over an $\E_\infty$-ring $A$ with zero section $e\colon \Spec A \to \sfE$. At this stage it is more convenient to work with the left adjoint to $\underline{\Ga}(\sfE_{U(1)}, -)$, a linearisation of (\ref{eq:equiv_ell_hom_intro}), which we also call {equivariant elliptic homology}
\[\O_{\sfE} \otimes_{A_{U(1)}} - \colon \Mod_{A_{U(1)}}(\Sp_{U(1)}) \to \QCoh(\sfE).\]
The fully faithfulness of this functor is a generic feature of the fact that $\sfE$ is perfect (\Cref{pr:left_adjoint_is_fullyfaithful}), so we are left with essential surjectivity. Two nice features of the zero section $e$ are that it defines an effective Cartier divisor (\Cref{ssec:effective_Cartier_divisors}) and that its complement $\sfU$ is affine. This leads to a commutative diagram
\[\begin{tikzcd}
    {\Mod_{A_{U(1)}}(\Sp_{U(1)})^{\loc}}\ar[d, "{F^\loc}"]\ar[r]  &   {\Mod_{A_{U(1)}}(\Sp_{U(1)})}\ar[d, "{\O_{\sfE} \otimes_{A_{U(1)}} - }"]\ar[r]  &   {\Mod_{A_{U(1)}}(\Sp_{U(1)})^\cpl}\ar[d, "{F^\cpl}"]  \\
    {\QCoh(\sfU)}\ar[r, "{i_\ast}"] &   {\QCoh(\sfE)}\ar[r] &   {\QCoh(\sfE)/\QCoh(\sfU),}
\end{tikzcd}\]
where the rows are Verdier sequences of stable $\infty$-categories, $i\colon \sfU \to \sfE$ is the open immersion, and the superscripts denote \emph{local} and \emph{complete} objects with respect to the $U(1)$-spectrum $\widetilde{E}U(1)$. This reduces us to showing that both $F^\loc$ and $F^\cpl$ are essentially surjective. As $\sfU$ is affine, we quickly see that $F^\loc$ is an equivalence, hence essentially surjective. For $F^\cpl$, one crucially uses that $\sfE$ is \emph{oriented},
which implies that the Verdier quotient $\QCoh(\sfE)/\QCoh(\sfU)$, given by quasi-coherent sheaves supported on the zero-section of $\sfE$, is equivalent to the functor $\infty$-category $\Fun(BU(1), \Mod_A(\Sp))$, also known as $A$-modules with a $U(1)$-action. By a form of the Atiyah--Segal completion theorem for $A_{U(1)}$, which again follows from the orientation of $\sfE$, this is in turn equivalent to $\Mod_{A_{U(1)}}(\Sp_{U(1)})^\cpl$. One then shows that these equivalences are inverse to $F^\cpl$, finishing the proof in the $U(1)$-case.

Our approach to the induction step in each case, from $U(n)$ to $U(n+1)$ for example, is similar: we match the relevant isotropy separation sequence on the equivariant side with an open-closed decomposition on the geometric side, using that the associated map on stacks is an effective Cartier divisor and that the complement is affine over a common base. A crucial ingredient here is the descriptions of the underlying schemes of $\sfE_{U(n)}$, $\sfE_{SU(n)}$, and $\sfE_{Sp(n)}$ in terms of Hilbert schemes and related objects, a result to appear in \cite{gm2}; we compile some of this classical material in \Cref{sec:appendix_classical_fun_stuff}. This general inductive argument is summarised as \Cref{lem:inductive_step_general}, and we apply it to the families of groups from \Cref{maintheorem_universal} in \Cref{ssec:explicit_sequences_of_compact_liegroups}. This then finishes the proof.

\subsection*{Outlook}
We restrict our attention to the family of compact Lie groups found in \Cref{maintheorem_universal} because it is for such $G$ that $\sfE_G^\heartsuit$ takes a well-understood form at the time of writing. We believe, however, that generalisations of our theorem deserve investigation. For example, we hope that an improvement in our collective understanding of $\sfE_{Spin(n)}$ might lead to equivariant affineness for the spin groups as well. For nonapplicable but connected compact Lie groups $G$ such as $SO(n)$, much less is known about the geometry of $\sfE_G$; we also do not know of any potential replacement of $\sfE_G$ with any better properties. Whether one uses Gepner--Meier's $\sfE_G$ or a replacement, it is unclear if equivariant affineness is reasonable to expect in general; perhaps it could be used as a test for an alternative definition of $\sfE_G$ for certain $G$, so as one in a list of desiderata.

Many of the arguments made here also work for a large class of oriented abelian stacks. For example, one could probably apply this to \emph{tori} $\sfG$, which are locally equivalent to the multiplicative group scheme $\G_m$, and recover \Cref{ex:MNN_ktheory_affineness_intro} for real and complex $G$-equivariant $K$-theory for various groups $G$. This strategy is clear when $G=U(1)^r$ is a torus, as in this case $\sfG_{U(1)^r} = \sfG^{\times r}$ is affine over its base. For $G=U(n)$, one can take $\sfG_{U(n)} = \sfG^{\times_n}/\Sigma_n$; other connected compact Lie groups $G$ would require more thought.

\subsection*{Outline}
This article is divided into three sections and three short appendices:

The first section \Cref{sec:SAG} contains some required background in spectral algebraic geometry. It begins (\Cref{ssec:immersions}) by discussing open and closed immersions in spectral algebraic geometry. Then (\Cref{ssec:effective_Cartier_divisors}) we define effective Cartier divisors in this context and prove their basic properties. Finally, we continue (\Cref{ssec:basics_of_perfect_stacks}) by developing the foundations of perfect stacks in this setting. Many of the stacks $\sfE_G$ of Gepner--Meier mentioned above are perfect and some of the crucial maps between them, such as $\sfE_{U(n)} \to \sfE_{U(n+1)}$, are effective Cartier divisors, highlighting the importance of these two concepts.

In \Cref{sec:review_of_GM}, we review the basics of equivariant elliptic cohomology that are needed to even state the theorems that we prove here. We begin (\Cref{ssec:EHT}) by summarising our notation from equivariant homotopy theory. The next two subsections summarise the input we require from Gepner--Meier's work-in-progress \cite{gm2}. The stacks $\sfE_G$ are defined and some of their nice properties stated (\Cref{ssec:GM_geometric_model}), then we show how they lead to equivariant elliptic (co)homology functors and the equivariant global sections functor at the heart of this article (\Cref{ssec:EEcohomology}).

Then in (\Cref{sec:affine_case}), we prove \Cref{maintheorem_universal,maintheorem_arbitrary_base} as outlined above. In \Cref{ssec:inductive_step}, we set up an induction procedure (\Cref{lem:inductive_step_general}), which gives conditions allowing one to lift equivariant affineness results from a subgroup of a compact Lie group to the group itself. In \Cref{ssec:explicit_sequences_of_compact_liegroups}, we apply this inductive approach to the sequences of applicable Lie groups $U(n)$, $SU(n)$, and $Sp(n)$, and deduce equivariant affineness in these cases over a strictly even-periodic Noetherian affine base. In \Cref{ssec:extension_to_finite_products}, we show how $G$- and $K$-equivariant affineness implies $G\times K$-equivariant affineness. In \Cref{ssec:equivariant_qcoh_and_base_change}, we prove some general base change properties of equivariant elliptic cohomology, perhaps of independent interest, which play an important role in the final subsection. Lastly, in \Cref{ssec:general_case}, we prove both \Cref{maintheorem_universal,maintheorem_arbitrary_base}.

The three appendices summarise some facts used in the main body of the article that we expect are known to experts, but for which we could not find direct references. The first appendix (\textsection\ref{ssec:properties_preserved_by_h}) proves some basic properties of the functor-of-points $h_{(-)} \colon \SpDM^\nc \to \Stk$ defined in \cite[\textsection3.2.2]{reconstruction}, such as how it behaves with respect to covers, affine maps, and immersions. The second appendix (\textsection\ref{sec:appendix_classical_fun_stuff}) records a collection of classical facts about Hilbert schemes and symmetric products on elliptic curves. We hope this summary may be helpful for those coming to this article from equivariant homotopy theory. The final appendix (\textsection\ref{sec:appendix2}) shows how a base change theorem for equivariant elliptic cohomology (\Cref{pr:base_change_for_equivariant_theories}) leads to an equivariant Wood's theorem for $\TMF$ and generic descendability results for equivariant $\TMF$-algebras.

\subsection*{Notation}
In this article, we use the language of $\infty$-categories, developed by Lurie in \cite{htt,ha,sag}. From now on, we drop the ``$\infty$-'' prefix. We also write mapping spaces $\Map_{\calC}(X,Y)$ simply as $\calC(X,Y)$. Generally, we follow the notation used by previous articles in equivariant homotopy theory and spectral algebraic geometry such as \cite{reconstruction,temperedglobal,gepner2024global2ringsgenuinerefinements,davidandlennart,gm2}. In particular, we work primarily in $\Stk$, the category of (fpqc-) stacks on $\E_\infty$-rings of \cite[Df.3.1.1.3]{reconstruction} with affine objects denoted by $\Spec A$. Lurie's category of connective spectral Deligne--Mumford stacks is denoted as $\SpDM$, with nonconnective variant $\SpDM^\nc$, whose affine objects are denoted by $\Spet A$; see \cite[\textsection1.4]{sag}. The category of connective spectral schemes is denoted by $\SpSch$, with nonconnective variant $\SpSch^\nc$. To reiterate what is often implicit: the adjective ``nonconnective'' means ``not necessarily connective'', hence it is the more general setting. The 2-category of classical Deligne--Mumford stacks will be written as $\DM$ and is viewed as a subcategory of the 1-localic 0-truncated connective spectral Deligne--Mumford stacks via \cite[Rmk.1.4.8.3]{sag}. The connective cover of some $\X\in \SpDM^\nc$ is denoted by $\X_{\geq 0}$, and similarly for $\E_\infty$-rings, and its further \(0\)-truncation by \(\X^\heartsuit\).

\subsection*{Acknowledgements}
Thank you to Lennart Meier and David Gepner for happily sharing a copy of \cite{gm2} with us and for the exposition of their work. Further thank you's to Lennart for reading a draft of this paper and Fabio Neugebauer for some fruitful conversations. SL would like to thank David Gepner for helpful discussions concerning this project.

During the writing of this article, JMD was supported by the DFG-funded research training group GRK 2240: Algebro-Geometric Methods in Algebra, Arithmetic and Topology and SL was an associate member of the SFB: Higher invariants. 

\subsection*{AI declaration}
This article was written without the use of generative AI tools nor LLMs. 

\section{Requisite spectral algebraic geometry}\label{sec:SAG}
The first two subsections (\Cref{ssec:immersions}-\Cref{ssec:effective_Cartier_divisors}) below, concerning immersions and effective Cartier divisors, take place entirely within Lurie's framework of nonconnective spectral Deligne--Mumford stacks; the relations to similar notions in $\Stk$ are discussed in \Cref{ssec:properties_preserved_by_h}. The third subsection (\Cref{ssec:basics_of_perfect_stacks}) on perfect stacks takes place in the generality of $\Stk$.

\subsection{Open and closed immersions of stacks}\label{ssec:immersions}
We begin with closed immersions. Recall from \cite[Th.3.1.2.1 \& Rmk.3.1.2.2]{sag} that an affine map of spectral Deligne--Mumford stacks $j \colon \sfZ \to \X$ is a \emph{closed immersion} if for all étale morphisms $\Spet A \to \X$ from a (necessarily connective) $\E_\infty$-ring $A$, the induced map of commutative rings
\[\pi_0 A \to \pi_0 \Ga(\sfZ \times_\X \Spet A)\]
is a surjection. The following is an adaptation of this definition to nonconnective Deligne--Mumford stacks; a variant of which appeared in \cite[Df.5.5.3]{temperedglobal}.

\begin{mydef}\label{df:closed_immersions}
    An affine morphism of nonconnective spectral Deligne--Mumford stacks $j \colon \sfZ \to \X$ is a \emph{closed immersion} if for all étale maps from an affine $\Spet A \to \X$, the induced map on commutative rings
    \[\pi_0 A \to \pi_0\Ga(\sfZ \times_\X \Spec A)\]
    is surjective.
\end{mydef}

By \Cref{lm:h_closed_immersion}, the functor of points $h_{(-)} \colon \SpDM^\nc \to \Stk$ preserves closed immersions, the latter being defined in \cite[Df.5.5.3]{temperedglobal}.

This definition of a closed immersion in $\SpDM^\nc$ does not appear in the literature so far, essentially due to the following, perhaps counterintuitive, fact.

\begin{warning}
    The natural embedding $\SpDM \to \SpDM^\nc$ does \textbf{not} generally preserve closed immersions. For example, $\Spec \Z \to \Spec \Sph$ is a closed immersion in $\SpDM$, but not in $\SpDM^\nc$ as its $\KU$-linearisation $\KU \to \KU \otimes \Z \simeq \KU \otimes \Q$ is not surjective on $\pi_0$. 
\end{warning}

If certain flatness conditions relative to a base are satisfied, then one can freely pass between various notions of closed immersion.

\begin{prop}\label{pr:equiv_of_closed_immersions}
    Let $\M$ be a nonconnective spectral Deligne--Mumford stack and $j \colon \sfZ \to \X$ a map between flat nonconnective spectral Deligne--Mumford stacks over $\M$. Then
    \begin{enumerate}
        \item $j$ is affine if and only if $j_{\geq 0}$ is affine, and
        \item $j$ is a closed immersion if and only if $j_{\geq 0}$ is a closed immersion.
    \end{enumerate}
\end{prop}

Proving this will use the following series of lemmata, surely well-known to experts.

\begin{lemma}\label{rmk:affine_maps_connectiveVSnonconnective}
    Let \(f:\sfY \to \sfX\) be a map of spectral Deligne--Mumford stacks. Then $f$ is affine in nonconnective spectral Deligne--Mumford stacks if and only if $f$ is affine in spectral Deligne--Mumford stacks. 
\end{lemma}

\begin{proof}
    The inclusions $\SpDM \to \SpDM^\nc$ preserves limits by \cite[Cor.1.4.5.2]{sag}, so if $f$ is affine in $\SpDM^\nc$, it also is in $\SpDM$. On the other hand, if \(f\) is affine in \(\SpDM\), then it is affine in \(\SpDM^\nc\) by adjunction; any map \(\Spet A\to \sfY\) factors over \(\Spet A\to \Spet (A_{\geq0})\), and any pullback between affines is affine.
\end{proof}

\begin{lemma}\label{lm:affine_maps_and_descent}
    Affine morphisms of (nonconnective) spectral Deligne--Mumford stacks satisfy étale descent, meaning that a map $f\colon \Y\to\X$ of (nonconnective) spectral Deligne--Mumford stacks is affine if and only if for all étale maps $\Spet A \to \X$, the pullback $\Y\times_\X \Spet A$ is affine.
\end{lemma}

\begin{proof}
    By \Cref{rmk:affine_maps_connectiveVSnonconnective}, it suffices to prove the nonconnective statement. To that end, the functor of points into étale stacks $h^\et_{(-)}\colon \SpDM^\nc \to \Stk^\et$ used in \Cref{lem:h_finite_limits} commutes with limits, sends affines to affines, and is fully faithful, it suffices to prove this statement in $\Stk^\et$. The fact that affine morphisms in $\Stk^\et$ satisfy étale descent follows from the same proof used to show that affine morphisms satisfy fpqc descent in $\Stk$; see \cite[Pr.A.3]{Davies_galois_geometric_FP}.
\end{proof}

\begin{lemma}\label{lm:connective_cover_maps_are_affine}
    For a (nonconnective) spectral Deligne--Mumford stack $\X$, the natural connective cover map $c_\sfX:\sfX\to \sfX_{\geq0}$ is affine.
\end{lemma}

\begin{proof}
    Let \(\Spet A\to \sfX\) be an étale map.
    From the explicit description of pullbacks along étale maps of \cite[Rmk.21.4.6.7]{sag}, we obtain a pullback square
\[\begin{tikzcd}
	{\Spet A} & \sfX \\
	{\Spet A_{\geq0}} & {\sfX_{\geq0}}
	\arrow[from=1-1, to=1-2]
	\arrow["{c_A}"', from=1-1, to=2-1]
	\arrow["{c_\sfX}", from=1-2, to=2-2]
	\arrow[from=2-1, to=2-2]
\end{tikzcd}\]
    Thus, \(c_\sfX\) is étale-locally an affine map, hence affine by \Cref{lm:affine_maps_and_descent}.
\end{proof}

\begin{lemma}\label{lem:flat_maps_are_good_in_SpDMnc}
    Let \(f:\sfY\to\sfX\) be a flat map of nonconnective spectral Deligne--Mumford stacks. Then the natural commutative square
\[\begin{tikzcd}
	\sfY & {\sfY_{\geq0}} \\
	\sfX & {\sfX_{\geq0}}
	\arrow["{c_\sfY}", from=1-1, to=1-2]
	\arrow["f"', from=1-1, to=2-1]
	\arrow["{f_{\geq0}}", from=1-2, to=2-2]
	\arrow["{c_\sfX}"', from=2-1, to=2-2]
\end{tikzcd}\]
    is a pullback in \(\SpDM^\nc\). Consequently, the functor
    \[(-)_{\geq 0} \colon \SpDM^{\nc}_{/\X} \to \SpDM_{/\X_{\geq0}}\]
    induces an equivalence when restricted to flat objects on both sides.
\end{lemma}

Variants of this appear in \cite[Lm.3.4.4.3]{reconstruction}, \cite[Lm.2.1]{elltempcomp}, and \cite[Pr.2.8.2.10]{sag}.

\begin{remark}
    The class of morphisms of stacks $f\colon \Y \to\X$ such that the commutative square of \Cref{lem:flat_maps_are_good_in_SpDMnc} is a pullback square is actually much larger than only flat maps. For instance, if $f \in \SpDM^\nc_{/\M}$ and both $\X$ and $\Y$ are flat over $\M$, then the above diagram is a pullback by the pasting lemma. More interestingly, it also holds for maps between complex-periodic (\Cref{def:even_periodic}) objects of $\SpDM^\nc$, as locally the map $\X\to\X_{\geq0}$ is given by inverting an element of degree 2. In fact, it holds in slightly more generality with only mild periodicity requirements, as exhibited by
    \[\KO \otimes_{\ko} \ku \simeq \ko[\beta^{-1}] \otimes_{\ko} \ku \simeq \ku[\beta^{-1}] \simeq \KU.\]
    Many of the surrounding statements, such as \Cref{pr:equiv_of_closed_immersions,prop:nc_ECD_flat_over_base}, hold in the generality of \(j:\sfZ\to\sfX\) such that square of \Cref{lem:flat_maps_are_good_in_SpDMnc} is a pullback square. This might be useful to extend the arguments here to stacks which are not flat over a common base.
\end{remark}

\begin{proof}[Proof of \Cref{lem:flat_maps_are_good_in_SpDMnc}]
    Given an étale map $\Spet A \to \X$, then just as in the proof of \Cref{lm:connective_cover_maps_are_affine}, we have a Cartesian square
\[\begin{tikzcd}
	{\Spet A} & {\Spet A_{\geq0}} \\
	{\sfX} & {\sfX_{\geq0}.}
	\arrow["g", from=1-1, to=2-1]
	\arrow["{c_A}"', from=1-1, to=1-2]
	\arrow["{c_\sfX}", from=2-1, to=2-2]
	\arrow["{g_{\geq0}}"', from=1-2, to=2-2]
\end{tikzcd}\]
    The same remark also shows that the base change of \(\sfY_{\geq0}\) along \(g_{\geq0}\) agrees with \((\sfY_A)_{\geq0}\).
    It follows that the base change of the original square along \(g_{\geq0}\) has the form
\[\begin{tikzcd}
	{\sfY_A} & {(\sfY_A)_{{\geq0}}} \\
	{\Spet A} & {\Spet A_{\geq0}.}
	\arrow["{c_{\sfX_A}}", from=1-1, to=1-2]
	\arrow["{f_A}"', from=1-1, to=2-1]
	\arrow["{(f_A)_{\geq0}}", from=1-2, to=2-2]
	\arrow["{c_A}"', from=2-1, to=2-2]
\end{tikzcd}\]
    Since flatness is preserved by base change, both of the vertical maps above are flat, hence this square is a pullback from \cite[Pr.2.8.2.10]{sag} and its proof. Writing $\X \simeq \colim \Spet A$ as an étale colimit of affines, the result follows from the universality of étale colimits in $\SpDM^\nc$: this in turn follows from universality of colimits in $\Stk^\et$, that \(h_{(-)}^\et\) is fully faithful and preserves étale colimits by \cite[Th.3.2.1.3]{reconstruction}, and  that \(h_{(-)}^\et\) preserves finite limits by \Cref{lem:h_finite_limits}.
    Gluing the two pullback diagrams above together then gives
    \[\Y \simeq \colim \Y_A \simeq \colim \left(\X \times_{\X_{\geq 0}} (\Y_A)_{\geq 0} \right) \simeq \X \times_{\X_{\geq 0}} \colim (\Y_A)_{\geq 0} \simeq \X \times_{\X_{\geq 0}} \Y_{\geq 0}.\]
    For the second statement, simply note that the inverse to the connective cover functor is base change along $\X \to \X_{\geq0}$.
\end{proof}

\begin{proof}[Proof of \Cref{pr:equiv_of_closed_immersions}]
    By \Cref{rmk:affine_maps_connectiveVSnonconnective}, the map $j_{\geq0}$ is affine in $\SpDM$ if and only if it is affine in $\SpDM^\nc$. By \Cref{lm:affine_maps_and_descent}, we may work étale locally, choose an étale cover of $\X$ by affines $\Spet A$, inducing an étale cover $\Spet A_{\geq 0}$ of $\X_{\geq 0}$. We then have the two pullbacks in $\SpDM^\nc$
    \[\begin{tikzcd}
        {P}\ar[r]\ar[d]     &   {\Spet A}\ar[d] \\
        {\sfZ}\ar[r, "j"]   &   {\sfX}
    \end{tikzcd}\qquad
    \begin{tikzcd}
        {P_{\geq 0}}\ar[r]\ar[d]     &   {\Spet A_{\geq0}}\ar[d] \\
        {\sfZ_{\geq 0}}\ar[r, "{j_{\geq0}}"]   &   {\sfX_{\geq0}.}
    \end{tikzcd}\]
    As all stacks in the left-hand square are flat over $\M$ and all those in the right-hand square are flat over $\M_{\geq 0}$, then \Cref{lem:flat_maps_are_good_in_SpDMnc} implies that $P_{\geq 0}$ really is the connective cover of $P$. In particular, by \cite[Cor.1.4.7.3]{sag}, we see that $P$ is affine if and only if $P_{\geq0}$ is affine, so $j$ is affine if and only if $j_{\geq0}$ is affine.

    Writing $P=\Spet B$ such that $P_{\geq 0} = \Spet B_{\geq0}$, it also immediately follows that $j$ is a closed immersion if and only if $j_{\geq 0}$ is a closed immersion, as $A \to B$ is surjective on $\pi_0$ if and only if $A_{\geq 0} \to B_{\geq0}$ is surjective on $\pi_0$.
\end{proof}

Complementary to closed immersions are open immersions.

\begin{mydef}[{\cite[Df.1.6.7.2]{sag}}]\label{df:open_immersion_SPDM}
    A map of nonconnective spectral Deligne--Mumford stacks $i\colon \sfY \to \X = (\calX, \O_\X)$ is an \emph{open immersion} if $\Y$ is equivalent to $(\calX_{/u}, \O_\X|_u)$ for a $(-1)$-truncated object in the underlying topos $\calX$ of $\X$. Such an open immersion is \emph{quasi-compact}, or simply a \emph{qc-open immersion}, if for each affine object $u\in \calX$, the pullback $i^\ast u$ is $0$-coherent in the underlying topos $\calY$ of $\Y$, meaning that every covering of $i^\ast u$ in $\calY$ admits a finite subcovering.
\end{mydef}

By \Cref{prop:h_open_immersion} below, (qc-) open immersions in $\SpDM^\nc$ induce (qc-) open immersions in $\Stk$.

\begin{construction}\label{con:connective_open_complement}
    Let $j \colon \sfZ \to \X$ be a closed immersion of connective spectral Deligne--Mumford stacks. In Lurie's definition of a closed immersion, equivalent to the above by \cite[Rmk.3.1.2.2]{sag}, there is a unique $(-1)$-truncated object $u$ in the underlying topos $\calX$ of $\X$, such that the underlying geometric morphism of \(\infty\)-topoi of \(j\) has the form $\mathcal{X}/ u\hookrightarrow\mathcal{X}$, where $\mathcal{X}/ u$ denotes the subtopos of $\mathcal{X}$ of objects $\mathcal{F}$ such that the projection $\mathcal{F}\times u \to \mathcal{F}$ is an isomorphism. The {\it complementary open immersion} of \(j\) is defined as
\[i\colon \sfU = (\mathcal{X}_{/u},\O_\X\vert_u)\to (\mathcal{X},\O_\X) = \sfX.\qedhere\]
\end{construction}

\begin{variant}\label{var:nonconnective_open_complement}
    Let $\M$ be a nonconnective spectral Deligne--Mumford stack and $j\colon \sfZ \to \X$ a closed immersion between flat nonconnective spectral Deligne--Mumford $\M$-stacks. Taking connective covers induces a closed immersion $j_{\geq0}$ by \Cref{pr:equiv_of_closed_immersions}, so there is an essentially unique $(-1)$-truncated $u$ in $\calX$, which is the underlying topos of both $\X$ and $\X_{\geq 0}$ by construction (\cite[Cor.1.4.5.2]{sag}), defining the complementary open immersion to $j_{\geq 0}$. The \emph{complementary open immersion} of $j$ is then defined as 
    \[i\colon \sfU = (\mathcal{X}_{/u},\O_\X\vert_u)\to (\mathcal{X},\O_\X) = \sfX.\]
    As $\sfU$ is flat over $\X$ and hence also over $\M$, and likewise the complementary open substack of $\X_{\geq 0}$ to $j_{\geq 0}$ is flat over $\M_{\geq 0}$, then \Cref{lem:flat_maps_are_good_in_SpDMnc} implies that the latter is the connective cover of the former, and that
    \[\sfU \simeq \sfU_{\geq 0} \times_{\sfX_\geq 0} \sfX \simeq \sfU_{\geq 0} \times_{\sfM_\geq 0} \sfM .\qedhere\]
\end{variant}

The following are generic facts concerning open immersions.

\begin{prop}[{\cite[Lm.2.22 \& Cor.2.23]{Tokic_2026_FamilyCompletionTheorem}}]\label{prop:open_immersion_omnibus}
    Let \(i\colon \mathsf{U}\to \X\) be a qc-open immersion of nonconnective spectral Deligne--Mumford stacks. Then \(i_*\colon \QCoh(\mathsf{U})\to \QCoh(\X)\) is fully faithful, \(i_*\O_{\mathsf{U}}\) is an idempotent \(\O_\X\)-algebra, and the essential image of \(i_*\) agrees with those objects in \(\QCoh(\X)\) that admit an \(i_*\O_{\mathsf{U}}\)-module structure.
\end{prop}

The referenced statements from \cite{Tokic_2026_FamilyCompletionTheorem} are stated in $\Stk$, but as the functor of points $h_{(-)} \colon \SpDM^\nc \to \Stk$ preserves qc-open immersions by \Cref{prop:h_open_immersion} and there is a natural equivalence $\QCoh(\X) \simeq \QCoh(h_\X)$ by \cite[Ex.3.2.2.3]{reconstruction}, we obtain the above statement.

\begin{prop}\label{cor:compact_generator_for_complement}
    Let $A$ be an $\E_\infty$-ring and $i\colon \sfU \to \Spec A$ be a qc-open immersion, so necessarily $\sfU$ is a nonconnective spectral scheme. Then the kernel of $i^\ast$ is generated by a compact object of $\Mod_A$.
\end{prop}

\begin{proof}
    The map $i$ is defined by a finite collection of elements $f_1, \ldots, f_n \in \pi_0 A$, and $i_\ast \O_\sfU$ is given by the localisation of $A$ at these elements. The kernel of $i^\ast$ is then generated by the $A$-module $A/(f_1,\ldots, f_n)$, which is compact as we only consider a finite set of elements.
\end{proof}

\subsection{Effective Cartier divisors}\label{ssec:effective_Cartier_divisors}
A particularly well-behaved class of closed immersions are the following:

\begin{mydef}\label{def:ECD_SpDM}
A closed immersion of connective (resp.\ nonconnective) spectral Deligne--Mumford stacks \(j\colon \sfZ\to \X\) is an \emph{effective Cartier divisor} if 
\[\O_\sfX(-\sfZ) = \fib(\O_\sfX\to j_*\O_\sfZ) \in \QCoh(\X)\] 
is étale locally isomorphic to \(\O_\sfX\). This implies that $\O_\sfX(-\sfZ)$ is a tensor invertible object of $\QCoh(\sfX)$, whose inverse, or equivalently dual, we denote by \(\O_\sfX(\sfZ)\).
\end{mydef}

This definition is most sensible when all stacks are flat over a common base, reducing us to the connective situation by \Cref{pr:equiv_of_closed_immersions}.

\begin{remark}
    The above notion of an effective Cartier divisor differs to that of Ma--Zhu \cite[Df.2.12]{ma_zhu_spectral_moduli_Jacquetlandlands} as we do not assume any extra hypotheses on $j$ such as properness, flattness, nor finiteness; this difference does not play a role in this article, but plays a key role in the use of spectral Artin representability found in \emph{ibid}.
\end{remark}

\begin{prop}\label{lem:locally_principal_conservative}
    Let \(j\colon \sfZ\to \sfX\) be an effective Cartier divisor of connective spectral Deligne--Mumford stacks. Then
    \begin{enumerate}
        \item the complementary open immersion \(i\colon \sfU\to \sfX\) of \Cref{con:connective_open_complement} is quasi-compact,
        \item \(i_*\O_\sfU\) is equivalent, as an \(\E_0\)-algebra, to \(\O_\sfX(\infty \sfZ)=\colim_n \O_\sfX(\sfZ)^{\otimes n}\), and
        \item the kernel of \(j^*\) agrees with the essential image of \(i_*\).
    \end{enumerate} 
\end{prop}

\begin{proof}
    Let \(f\colon \Spet A\to \sfX\) be an étale map equipped with an equivalence \(f^*\O_\sfX(-\sfZ)\simeq A\), and denote the base change of \(j\colon \sfZ\to \sfX\) along $f$ by \(j_A\colon \sfZ_A\to \Spet A\), and the base change of $f$ along $j$ by $f' \colon \sfZ_A \to \sfZ$. By \cite[Cor.3.1.2.3]{sag}, $j_A$ is a closed immersion and \cite[Th.3.1.2.1]{sag} tells us that \(\sfZ_A\simeq \Spet B\) is affine and that \(A\to B\) is surjective on \(\pi_0\).
    The fibre sequence \(\O_\sfX(-\sfZ)\to \O_\sfX\to j_*\O_\sfZ\) pulls back along $f$ to 
    \[A\xrightarrow{\cdot a} A\to B \simeq (j_A)_\ast f'^\ast \O_\sfZ \simeq f^\ast j_\ast \O_\sfZ \in \Mod_A,\]
    using the Beck--Chevalley equivalence for affine morphisms, see \cite[Cor.2.5.4.6]{sag} or \cite[Pr.2.2.2.5]{reconstruction}, for some \(a\in \pi_0 A\). In particular, we must have \(\pi_0B\simeq (\pi_0A)/(a)\) as $A$ is connective.
    Thus, the closed subspace \(|\sfZ_A|\subset |\Spet A|\) is given by \(V(a)\).
    It follows that the complementary open immersion \(\sfU_A \simeq \Spet A[a^{-1}] \to \Spet A\) is quasi-compact. As \(\sfU_A \simeq \Spet A[a^{-1}]\) is the base change of \(i\colon \sfU\to \sfX\) along \(f\colon \Spet A\to \sfX\), and as $f$ varies it covers $\sfX$, then \cite[Prs.2.3.3.1 \& 2.3.2.1(3)]{sag} tell us that \(i\colon \sfU\to \sfX\) must already be quasi-compact. This proves part 1.
    
    For parts 2 and 3, note that a closed immersion in \(\SpDM^\nc\) induces a closed immersion in \(\Stk\) by \Cref{lm:h_closed_immersion}. By \cite[Pr.2.2.2.5]{reconstruction}, the pullback square
\begin{equation}\label{eq:pullback_open_immersion}\begin{tikzcd}
	P\ar[r]\ar[d] & {h_\sfU}\ar[d, "i"] \\
	{h_\sfZ}\ar[r, "j"] & {h_\sfX}
	\arrow["\lrcorner"{anchor=center, pos=0.125}, draw=none, from=1-1, to=2-2]
\end{tikzcd}\end{equation}
    is adjointable; recall this definition from \cite[Df.2.2.1.8]{reconstruction}, reviewed as (\ref{eq:adjointable_map_of_stacks}) below. According to \cite[Pr.7.3.2.7]{htt}, the preimage of \(|\sfU|\subset |\sfX|\) in \(|\sfZ|\) on underlying topological spaces is empty.
    From \cite[Rmk.A.7]{Tokic_2026_FamilyCompletionTheorem}, we know that the natural comparison \(|h_\bullet|\to |{\bullet}|\) is {\it sobrification}, that is, the reflection into the subcategory of sober topological spaces.
    As the map \(h_\sfU\to h_\sfX\) is a quasi-compact open immersion by \Cref{prop:h_open_immersion} and sobrification induces bijections between the sets of open subsets, we find that the preimage of \(|h_\sfU|\subset |h_\sfX|\) along $j$ is empty, so we have \(|P|\simeq \emptyset\) by \cite[Cor.A.28]{Tokic_2026_FamilyCompletionTheorem}.
    That implies \(P = \Spec 0 = \varnothing\) and \(i^*j_*\O_\sfZ\simeq 0\) vanishes.

    Consequently, the map \(\O_\sfX\to \O_\sfX(\sfZ)\), dual to the defining $\O_\X(-\sfZ) \to \O_\X$, is an equivalence after applying \(i^*\). As $i^\ast$ is symmetric monoidal, this also means that \(i^*\O_\sfX\to i^*\O_\sfX(\infty \sfZ)\) is an equivalence. Taking an inverse, composing with \(i^*\O_\sfX\to \O_\sfU\), and taking adjoints, we obtain a map \(\O_\sfX(\infty \sfZ)\to i_*\O_\sfU\) of \(\E_0\)-algebras in $\QCoh(\X)$, which is an equivalence after applying \(i^*\). 
    
    For an étale map \(f\colon \mathrm{Sp\et}A\to \sfX\) such that \(f^*\O_\sfX(-\sfZ)\simeq A\), the map \(f^*\O_\sfX\to f^*\O_\sfX(\sfZ)\) can be identified with \(A\xrightarrow{\cdot a}A\), for some \(a\in \pi_0 A\). Similarly, \(f^*\O_\sfX\to f^*\O_\sfX(\infty \sfZ)\) can be identified with the canonical map from \(A\) to $Tel(A,a)$, the colimit of iterated multiplication by $a$ on $A$. We saw above that the base change of \(\sfU\to \sfX\) to \(\mathrm{Sp\et} A\) is given by \(\mathrm{Sp\et}A[a^{-1}]\to \mathrm{Sp\et} A\), so the map \(f^*\O_\sfX\to f^*i_*\O_\sfU\) in $\Mod_A$ can be identified with the localisation \(A\to A[a^{-1}]\). Applying \(f^*\) to \(\O_\sfX(\infty \sfZ)\to i_*\O_\sfU\) yields \(Tel(A,a)\to A[a^{-1}]\), which is an equivalence. As \(\sfX\) is covered by such étale affines, the original map \(\O_\sfX(\infty \sfZ)\to i_*\O_\sfU\) must be an equivalence.

    For part 3, let \(\calM\in \QCoh(\sfX)\) be such that \(j^*\calM\simeq 0\).
    Since \(j\colon h_\sfZ\to h_\sfX\) satisfies the projection formula, by \Cref{cor:h_affine_map} and \cite[Pr.2.2.2.5(2)]{reconstruction}, we conclude that \(j_*\O_\sfZ\otimes \calM\simeq 0\), and therefore the map \(\O_\sfX(-\sfZ)\otimes \calM \to \calM\) is an equivalence. Tensoring by $\O_\sfX(\sfZ)$, which recall is invertible by assumption, we find that \(\calM\to \O_\sfX(\sfZ)\otimes \calM\) is an equivalence. Therefore \(\O_\sfX\otimes \calM\to \O_\sfX(\infty \sfZ)\otimes \calM\simeq i_*\O_\sfU\otimes \calM\) is an equivalence, so \(\calM\) is in the essential image of \(i_*\) by \Cref{prop:open_immersion_omnibus}. Conversely, if \(\calN\in\QCoh(\sfU)\), then we compute that \(j^*i_*\calN\simeq 0\) using the {\it vertical} adjointability of the pullback square (\ref{eq:pullback_open_immersion}) above, which comes from \(i\) being a qc-open immersion (\cite[Pr.2.4.1.5(3)]{sag}) together with \(P\simeq\Spec 0 = \varnothing\).
\end{proof}

To make sensible statements in the nonconnective setting, we prefer to work with the following classes of stacks.

\begin{mydef}\label{def:even_periodic}
    A nonconnective spectral Deligne--Mumford stack $\X$ is \emph{complex-periodic} if for each map $\Spet A \to \X$, the $\E_\infty$-ring $A$ is complex-periodic à la \cite[Df.4.1.8]{ec2}. Such an $\X$ is \emph{even-periodic} if in addition $\pi_{2k+1}\O_\X$ vanishes for all integers $k$. On the other hand, an $\E_\infty$-ring $A$ is \emph{strictly even-periodic} if $\pi_2 A$ contains a unit and $\pi_1 A=0$.
\end{mydef}

\begin{remark}
    An object $\X$ in $\SpDM^\nc$ is complex-periodic if and only if its functor of points $h_\X$ is complex-periodic in $\Stk$ as defined in \cite[Df.4.0.0.1]{reconstruction}. Indeed, by \Cref{lem:h_map_from_affine_local_lift}, for each map $\Spec A \to h_\X$ there is a faithfully flat $A\to B$ and a lift $\Spet B \to \X$ in $\SpDM^\nc$ of the composite $\Spec B \to h_\X$. In particular, we see that $B$ is complex-periodic by assumption. Being complex-orientable is an fpqc-local condition by \cite[Lm.4.1.2.3]{reconstruction}, and being weakly 2-periodic is clearly fpqc-local, hence $A$ is complex-periodic, and so is $h_\X$. Conversely, if $h_\X$ is complex-periodic, then $\X$ is complex-periodic straight from the definitions by applying $h_{(-)}$.
\end{remark}

The following gives a criterion for producing effective Cartier divisors, in terms of classical algebraic geometry. Recall the notion of \emph{1-localic} objects of $\SpDM^\nc$ from \cite[\textsection6.4.5]{htt}.

\begin{lemma}\label{lem:even_periodic_underlying_ECD}
    Let \(j\colon \sfZ\to \sfX\) be a map of even-periodic 1-localic nonconnective spectral Deligne--Mumford stacks such that \(j^\heartsuit\colon \sfZ^\heartsuit\to \sfX^\heartsuit\) is an effective Cartier divisor of classical Deligne--Mumford stacks. Then $j$ and $j_{\geq0}$ are effective Cartier divisors in $\SpDM^\nc$ and $\SpDM$, respectively.
\end{lemma}

In particular, note that nonconnective spectral schemes are always 1-localic.

\begin{proof}
    As $j^\heartsuit$ is an effective Cartier divisor of classical Deligne--Mumford stacks, we see that the 1-topos $\X^\heartsuit$ admits an étale cover by opens $V$ such that the classical Deligne--Mumford stacks $\sfV = (V,\pi_0 \O_{\X^\heartsuit}|_V)$ are affine and the kernel of $\O_{\sfX^\heartsuit} \to j_\ast^\heartsuit \O_{\sfZ^\heartsuit}$ restricted to $\sfV$ is isomorphic in $\O_\sfV$ as an $\O_\sfV$-module. As $\X$ is 1-localic, we also have associated nonconnective spectral Deligne--Mumford stacks $\Spet A = \sfU = (V,\O_\X|_V)$ covering $\X$ such that $\sfU^\heartsuit \simeq \sfV$ as classical Deligne--Mumford stacks. The $\E_\infty$-rings $A$ are even-periodic by hypothesis, and taking further Zariski refinements, we can also assume that $\pi_2 A$ contains a unit of $\pi_\ast A$, ie, they are \emph{strictly 2-periodic}.   
    
    To see that $j$ and $j_{\geq 0}$ are affine, consider the commutative diagrams
    \[\begin{tikzcd}
        {\sfZ_A}\ar[d]\ar[r]  &   {\Spet A}\ar[d]   \\
        {\sfZ}\ar[r, "{j}"]   &   {\X}
    \end{tikzcd}\qquad
    \begin{tikzcd}
        {(\sfZ_A)_{\geq0}}\ar[d]\ar[r]  &   {\Spet A_{\geq0}}\ar[d]   \\
        {\sfZ_{\geq0}}\ar[r, "{j_{\geq0}}"]   &   {\X_{\geq0}}
    \end{tikzcd}\qquad
    \begin{tikzcd}
        {(\sfZ_A)^\heartsuit}\ar[d]\ar[r]  &   {\Spet \pi_0 A}\ar[d]   \\
        {\sfZ^\heartsuit}\ar[r, "{j^\heartsuit}"]   &   {\X^\heartsuit}
    \end{tikzcd}\]
    in $\SpDM^\nc$, $\SpDM$, and $\DM$, from left to right, where the left-most square is a pullback by definition. The centre square above is a pullback from the description of pulling back étale maps of spectral Deligne--Mumford stacks via their underlying topoi; see \cite[Rmk.21.4.6.7]{sag}. It follows that the right-most square is also a pullback as the truncation functor $\SpDM \to \DM$ of \cite[Cor.1.4.6.4]{sag} is a right adjoint. By assumption then, $(\sfZ_A)^\heartsuit$ is affine. The underlying topoi of both $\sfZ_A$ and $(\sfZ_A)_{\geq0}$ are 1-localic by \cite[Lm.1.4.7.7(2)]{sag}, so \cite[Cor.1.4.7.3]{sag} implies that they are also both affine, say $\Spet B$ and $\Spet B_{\geq 0}$, respectively. Hence $j$ and $j_{\geq 0}$ are both affine. Moreover, as $j^\heartsuit$ is a closed immersion and $(\sfZ_A)^\heartsuit \simeq \Spet \pi_0 B$, we see that $\pi_0 A \to \pi_0 B$ is surjective, so both $j$ and $j_{\geq 0}$ are closed immersions in $\SpDM^\nc$ and $\SpDM$, respectively.

    As $\Spet B \to \sfZ$ is étale and $\sfZ$ is even-periodic, then $B$ is necessarily even-periodic. By assumption, we have an equivalence of $\pi_0 A$-modules $\pi_0 A \simeq \ker(\pi_0A \to \pi_0 B)$. By the evenness of $B$ and the fact that $\pi_0A \to \pi_0 B$ is surjective, we can also identify $\pi_0 A$ with $\pi_0 F$, where $F=\fib(A \to B)$, as $\pi_0 A$-modules. In particular, this gives a map of $A$-modules $A \to F$ that is an equivalence on $\pi_0$. As $A$ is even-periodic with a unit in $\pi_2 A$, we see that $A \to F$ is an equivalence, hence $j$ is an effective Cartier divisor. The same argument shows that the fibre of $A_{\geq0} \to B_{\geq0}$ is equivalent to $A_{\geq0}$ as an $A_{\geq0}$-module, so $j_{\geq 0}$ is also an effective Cartier divisor.
\end{proof}

Finally, we would like a version of \Cref{lem:locally_principal_conservative} that holds in the nonconnective setting. This will come with an additional flatness assumption.

\begin{prop}\label{prop:nc_ECD_flat_over_base}
    Let $\M$ be a nonconnective spectral Deligne--Mumford stack and \(j:\sfZ\to\sfX\) be a map of nonconnective spectral Deligne--Mumford stacks over $\M$, where both $\sfZ$ and $\X$ are flat over $\M$. Suppose that $j_{\geq 0}$ is an effective Cartier divisor in $\SpDM$.
    Then
    \begin{enumerate}
        \item $j$ is an effective Cartier divisor in $\SpDM^\nc$,
        \item the complementary open immersion \(i \colon \sfU \to \X\) of \Cref{var:nonconnective_open_complement} is quasi-compact,
        \item \(i_*\O_\sfU\) is equivalent, as an \(\E_0\)-algebra, to \(\O_\sfX(\infty \sfZ) = \colim_n \O_\sfX(\sfZ)^{\otimes n}\), and 
        \item the kernel of \(j^*\) agrees with the essential image of \(i_*\).
    \end{enumerate}
\end{prop}
\begin{proof}
    For the first statement, consider the commutative diagram in $\SpDM^\nc$
    \[\begin{tikzcd}
        {\sfZ}\ar[r, "j"]\ar[d, "{c_\sfZ}"] &   {\X}\ar[r]\ar[d, "
        {c_\X}"] &   {\M}\ar[d, "{c_\M}"] \\
        {\sfZ_{\geq0}}\ar[r, "{j_{\geq0}}"] &   {\X_{\geq0}}\ar[r] &   {\M_{\geq0}.}
    \end{tikzcd}\]
    As $\X$ and $\sfZ$ are flat over $\M$, then \Cref{lem:flat_maps_are_good_in_SpDMnc} shows that the right-hand and outer squares are pullbacks, hence the left is also a pullback. Choose connective affines $\Spet A$ that form an étale cover of $\X_{\geq0}$ such that $\O_{\sfX_{\geq0}}(-\sfZ_{\geq 0})$ restricts to the structure sheaf on each $\Spet A$. By \Cref{lem:flat_maps_are_good_in_SpDMnc}, these $A$ are the connective cover of $\E_\infty$-rings $B$ that form an étale cover of $\X$ and form a pullback diagram
    \[\begin{tikzcd}
        {\Spet B}\ar[d, "h"]\ar[r, "g"]   &   {\X}\ar[d, "{c_X}"]\\
        {\Spet A}\ar[r, "f"]         &   {\X_{\geq0}.}
    \end{tikzcd}\]
    It then follows that $\O_\X(-\sfZ)$ pulls back to the structure sheaf of $\Spet B$
    \[g^\ast \O_\X(-\sfZ) \simeq \fib(g^\ast \O_{\X} \to g^\ast j_\ast \O_\sfZ) \simeq \fib(\O_{\Spet B} \to g^\ast j_\ast c_\sfZ^\ast \O_{\sfZ_{\geq0}})\]
    \[\simeq \fib(\O_{\Spet B} \to g^\ast c_\X^\ast (j_{\geq 0})_\ast \O_{\sfZ_{\geq0}}) \simeq \fib(h^\ast \O_{\Spet A} \to h^\ast f^\ast (j_{\geq 0})_\ast \O_{\sfZ_{\geq0}})\]
    \[\simeq h^\ast \O_{\X_\geq 0}(-\sfZ_{\geq0}) \simeq h^\ast \O_{\Spet A} \simeq \O_{\Spet B},\]
    the third equivalence using the fact that affine maps such as $j$ and $j_{\geq 0}$, see \Cref{rmk:affine_maps_connectiveVSnonconnective} for why the latter is affine in $\SpDM^\nc$, are adjointable; see \cite[Cor.2.5.4.6]{sag} or \cite[Pr.2.2.2.5]{reconstruction}. This proves part 1.

    For part 2, one simply notes that $i \colon \sfU \to \sfX$ is quasi-compact if and only if $i_{\geq 0}$ is quasi-compact, as by construction these maps induce the same map on underlying topoi and quasi-compactness is a condition only concerning topoi. The result follows from \Cref{lem:locally_principal_conservative}.

    For parts 3 and 4 concerning quasi-coherent sheaves, by \cite[Pr.3.2.2.2]{reconstruction}, it suffices to pass to $\Stk$; we suppress the functor \(h_{(-)}\) for now. As everything is flat over \(\sfM\), \Cref{lem:flat_maps_are_good_in_SpDMnc} gives us pullback squares in \(\SpDM^\nc\) and, by \Cref{lem:h_finite_limits}, in \(\Stk\)
\[\begin{tikzcd}
	\sfZ & \sfX & \sfU\ar[l, "i", swap] \\
	{\sfZ_{\geq0}} & {\sfX_{\geq0}} & {\sfU_{\geq0}}\ar[l, "{i_{\geq0}}"]
	\arrow["j", from=1-1, to=1-2]
	\arrow["{c_\sfZ}"', from=1-1, to=2-1]
	\arrow["{c_\sfX}", from=1-2, to=2-2]
	\arrow["{c_\sfU}"', from=1-3, to=2-3]
	\arrow["{j_{\geq0}}"', from=2-1, to=2-2]
\end{tikzcd}\]
    which are both horizontally and vertically adjointable: horizontally since \(j_{\geq0}\) is affine and \(i_{\geq0}\) is a qc-open immersion, vertically since the map \(c_\sfX\) is affine by \Cref{lm:connective_cover_maps_are_affine}; see \cite[Pr.2.2.2.5 \& Th.3.3.2.1]{reconstruction}. This adjointablity furnishes the equivalence
    \[(\O_\sfX\to j_*\O_\sfZ)\simeq (c_\sfX)^*(\O_{\sfX_{\geq0}}\to (j_{\geq0})_*\O_{\sfZ_{\geq0}}),\]
    so taking fibres, dualising, and taking the colimit identifies the maps
    \[(\O_\sfX\to \O_{\sfX}(\infty \sfZ))\simeq (c_\sfX)^*(\O_{\sfX_{\geq0}}\to \O_{\sfX_{\geq0}}(\infty\sfZ_{\geq0})).\]
    Part 2 of \Cref{lem:locally_principal_conservative} identifies this further with $(c_\sfX)^*(\O_{\sfX_{\geq0}}\to (i_{\geq0})_*\O_{\sfU_{\geq0}})$, and another adjointability identifies this with \(\O_\sfX\to i_*\O_\sfU\), proving part 3 of the statement. For part 4, we note that, as \(c_\sfZ\) is affine, the functor \((c_\sfZ)_*\) is conservative, so the kernel of \(j^*\) agrees with the kernel of \((j_{\geq0})^*\circ(c_\sfX)_*\).
    By \Cref{lem:locally_principal_conservative} and \cite[Cor.2.23]{Tokic_2026_FamilyCompletionTheorem}, the kernel of \((j_{\geq0})^*\) agrees with those \(A\in\QCoh(\sfX_{\geq0})\) such that the adjunction unit
    \[A\to (i_{\geq0})_*(i_{\geq0})^*A\]
    is an equivalence.
    Applying this to \(A=(c_\sfX)_*A'\) and using, again, the adjointability of the diagrams above, we see that \(j^*A'\simeq 0\) if and only if the unit $(c_\sfX)_*A'\to (c_\sfX)_*i^* i_*A'$ is an equivalence.
    Since \((c_\sfX)_*\) is conservative, this happens if and only if $A'\to i_*i^*A'$ is an equivalence, so by \cite[Cor.2.23]{Tokic_2026_FamilyCompletionTheorem} if and only if \(A'\) is in the essential image of \(i_*\).
\end{proof}

\subsection{Perfect stacks}\label{ssec:basics_of_perfect_stacks}
First, we review some notions from \cite{perfectstacks} in the context of our working notion of $\Stk$.

\begin{mydef}[{\cite[Pr.3.9]{perfectstacks}}]
    A stack $\X$ is \emph{perfect} if $\QCoh(\X)$ is compactly generated and compact and dualisable objects agree. A morphism of stacks $\Y \to \X$ is \emph{perfect} if for all $\Spec A \to \X$, the pullback stack $\Y\times_\X \Spec A$ is perfect.
\end{mydef}

The same definition also appears in \cite[Df.2.24]{Davies_galois_geometric_FP}. We have chosen to ignore the affine diagonal condition from \cite[Df.3.2]{perfectstacks} to de-emphasise its importance; note that most of our statements below can be made without any reference to affine diagonals.

Affine stacks and 0-affine stacks, so those stacks $\X$ where $\QCoh(\X)$ is affine, are clearly perfect. The following gives us another useful class of perfect stacks. Recall that a morphism of stacks $\Y \to \X$ is a \emph{relative qcqs nonconnective spectral scheme} if for all maps $\Spec A\to \X$, the pullback $\Y \times_\X \Spec A$ is a qcqs nonconnective spectral scheme.

\begin{prop}\label{ex:schemes_are_passable}
    A relative qcqs nonconnective spectral schemes $f\colon \Y \to \X$ is a perfect morphism. Moreover, if $\X$ is a qcqs nonconnective spectral scheme, then $\QCoh(\X)$ has a single compact generator.
\end{prop}

This result appears in \cite[Pr.3.19]{perfectstacks}, where much credit is given to Thomason--Trobaugh \cite{thomasontrobaugh}, Neeman \cite{neeman_grothendieck_duality}, and Bondal--van den Bergh \cite{bondal_van_den_bergh}. An affine version of the key trick, sometimes called \emph{Thomasons trick} also appears in \cite[Pr.D.5.3.4]{sag}. Our proof below mostly cribs from \cite{bondal_van_den_bergh}. 

First, a short well-known proposition. Recall from \cite[Df.2.2.1.5]{reconstruction}, that a stack $\X$ is \emph{0-semiaffine} if the global sections functor $\Ga(\sfX,-)\colon \QCoh(\X) \to \Sp$ preserves colimits and satisfies the projection formula.

\begin{prop}\label{pr:0semiaffine_spectral_schemes}
    Let $\X$ be a nonconnective qcqs spectral scheme. Then $\X$ is 0-semiaffine. In particular, \(\O_\sfX\otimes (-)\), the left adjoint of \(\Ga(\sfX,-):\QCoh(\sfX)\to\Sp\), preserves compact objects and dualisable objects of $\QCoh(\X)$ are compact.
\end{prop}


\begin{proof}
    We follow \cite[Cor.3.25]{akhilandlennart}. To see that $\Ga(\X,-)$ commutes with colimits, it suffices to show it commutes with arbitrary direct sums. As homotopy groups are conservative and commute with direct sums, we are reduced to showing that $\pi_d \Ga(\X,-)$ commutes with direct sums. As sheaf cohomology commutes with direct sums, the $E_2$-page of the descent spectral sequence for $\Ga(\X, -)$ commutes with direct sums, hence this is also true on the $E_\infty$-page. As the underlying classical scheme of $\X$ is qcqs, it follows that this $E_2$-page has a horizontal vanishing line, meaning that the abutment of this descent spectral sequence, so $\pi_d \Ga(\X,-)$, also commutes with direct sums, as desired.
    
    For the ``in particular'' statement, a left adjoint preserves compacts if its right adjoint preserves filtered colimits, a consequence of the previous paragraph. Next, note that $\X$ is 0-semiaffine if and only if $\O_\X$ is compact in $\QCoh(\X)$. It follows that all dualisable objects $\calM$ are compact as
    \[\QCoh(\X)(\calM, -) \simeq \QCoh(\X)(\O_\X, \calM^\vee \otimes -)\]
    commutes with filtered colimits.
\end{proof}

\begin{proof}[Proof of \Cref{ex:schemes_are_passable}]
    From the definitions, it suffices to show that an arbitrary qcqs nonconnective spectral scheme $\X$ is perfect and that $\QCoh(\X)$ has a single compact generator. If $\X=\Spec A$ is affine, this is clear. As $\X$ is quasi-compact, we can induct on a finite open cover of $\X$, so assume that $\Y \sqcup \sfU \to \X$ is a Zariski cover of $\X$ where $\QCoh(\Y)$ has a single compact generator $\calM$ and $\sfU = \Spec A$ is affine. We then have a pullback of qcqs nonconnective spectral schemes
    \[\begin{tikzcd}
        {\sfW}\ar[r, "{i'}"]\ar[d, "g"] &   {\sfU = \Spec A}\ar[d, "{f}"] \\
        {\sfY}\ar[r, "i"]               &   {\sfX,}
    \end{tikzcd}\]
    where $i$ and $f$ are qc-open immersions by hypothesis, and $i'$ and $g$ are qc-open immersions by base change (\cite[Lm.A.30]{Tokic_2026_FamilyCompletionTheorem}). As $i\colon \sfW \to \sfU$ is a qc-open immersion, it follows that the kernel of $i'^\ast$ is generated by a compact object $\calQ$ of $\QCoh(\sfU)$ courtesy of \Cref{cor:compact_generator_for_complement}; the sheaf $\calQ$ is intuitively the pushforward of the structure sheaf of ``the'' closed immersion complementing $i'$ (and also $i$). Moreover, we also see that $f_\ast \calQ$ lies in the kernel of $i^\ast$, as $i^\ast f_\ast \calQ \simeq g_\ast i'^\ast \calQ = 0$ by assumption and the adjointablity of qc-open immersions; this later point follows from the facts that qc-open immersions are quasi-affine, that quasi-affine maps are universally 0-affine (\cite[Th.3.3.2.1]{reconstruction}), and finally that universally 0-affine maps are adjointable (\cite[Pr.2.2.2.5(1)]{reconstruction}). Moreover, $f_\ast \calQ$ is dualisable in $\QCoh(\X)$ as $f^\ast f_\ast \calQ = \calQ$ is dualisable in $\QCoh(\sfU)=\Mod_A$.    
    
    A key point now is the existence of a compact object $\calF \in \QCoh(\sfU)$ such that $i'^\ast \calF \simeq g^\ast(\calM \oplus \calM[1])$. To construct this object, note that $i'^\ast$ preserves compact objects, itself a consequence of \Cref{pr:0semiaffine_spectral_schemes}. The induced functor $i'^{\ast,\omega} \colon \QCoh(\sfU)^\omega \to \QCoh(\sfW)^\omega$ on compact objects has kernel generated by a compact object $\calQ$ and $\QCoh(\sfU)$ itself is compactly generated, all mentioned above. Hence we can apply \cite[Th.3.2.1 \& Pr.3.2.2]{bondal_van_den_bergh}, which is sometimes known as Thomason's trick, which shows that an object of $\QCoh(\sfW)^\omega$ lies in the essential image of $i'^{\ast,\omega}$ if its representative in $K_0(\sfW)$ lies in the image of $K_0(\sfU)$. As $g^\ast \calM \oplus g^\ast \calM[1]$ represents the zero class in $K$-theory, we obtain the desired class $\calF$ in $\QCoh(\sfU)$. One now glues $\calF$ and $\calM\oplus \calM[1]$ together in $\QCoh(\X)$, ie, we define $\calP$ using the fibre sequence
    \[\calP \to f_\ast \calF \oplus i_\ast(\calM\oplus \calM[1]) \to i_\ast g_\ast g^\ast (\calM\oplus \calM[1]) \simeq f_\ast i'_\ast i'^\ast \calF,\]
    where the second map is the sum of the indicated units and the already discussed adjointability of the pullback square above. Note that $\calP$ is dualisable as one computes $f^\ast \calP \simeq \calF$ and $i^\ast \calP \simeq \calM\oplus \calM[1]$, which are both dualisable by hypothesis. We then set $\calC = \calP \oplus f_\ast\calQ$ in $\QCoh(\X)$. To see that $\calC$ is a generator of $\QCoh(\X)$, one copies the arguments from the proof of \cite[Th.3.1.1]{bondal_van_den_bergh}, a sequence of Mayer--Vietoris sequences from the constructions above. To see that $\calC$ is compact, it suffices to show dualisability by \Cref{pr:0semiaffine_spectral_schemes}, which is clear, as both $\calP$ and $f_\ast \calQ$ are dualisable as discussed above. 
    Finally, to conclude that $\X$ is perfect, we already mentioned that dualisable objects of the compactly generated $\QCoh(\X)$ are compact. Conversely, all compact objects are retracts of finite colimits of the compact generator of $\QCoh(\X)$, hence retracts of perfect objects, which shows they are perfect, as this can be checked locally on affines.
\end{proof}

Now we prove the basic properties of perfect stacks.

Recall the notion of universally 0-affine morphisms from \cite[Df.2.2.2.1]{reconstruction} as well as \cite[Pr.2.2.2.5 \& Th.3.3.2.1]{reconstruction} that state that affine and quasi-affine morphisms are universally 0-affine, respectively.

\begin{prop}\label{pr:permancence_per_proper}
    Let $g\colon \sfZ \to \Y$ and $f\colon \Y \to \X$ be maps of stacks.
    \begin{enumerate}
        \item Perfect morphisms of stacks satisfy base change.
        \item If $f$ is perfect and $g$ is universally 0-affine, then $f\circ g$ is perfect. 
        \item If $f$ is affine and $g$ is perfect, then $f\circ g$ is perfect.
        \item If $f\circ g$ is perfect and $f$ has universally 0-affine diagonal, then $g$ is perfect.
    \end{enumerate}
\end{prop}

\begin{proof}
    Base change is clear from the definition. For part 2, if $f$ is perfect and $g$ is universally 0-affine, then for each $\Spec A \to \Y$ which have a diagram of Cartesian squares
    \[\begin{tikzcd}
        {\sfZ_A}\ar[r, "{g'}"]\ar[d]    &   {\Y_A}\ar[r]\ar[d]  &   {\Spec A}\ar[d] \\
        {\sfZ}\ar[r, "g"]   &   {\Y}\ar[r, "f"] &   {\X.}
    \end{tikzcd}\]
    The stack $\Y_A$ is perfect as $f$ is perfect, and $g'$ is universally 0-affine by base change (\cite[Pr.2.2.2.5(1)]{reconstruction}). In particular, $g'$ is 0-affine by \emph{ibid}, hence
    \[\QCoh(\sfZ_A) \simeq \Mod_{g'_\ast \O_{\sfZ_A}}(\QCoh(\Y_A))\]
    is compactly generated and compacts and dualisables agree. Part 3 easily falls out of the definitions as then the stack $\Y_A$ above is affine. Finally, for part 4, consider the diagram of stacks
    \[\begin{tikzcd}
        {\sfZ}\ar[r, "{\Ga_g}"]\ar[d, "g"] &   {\sfZ \times_\X \Y}\ar[d]\ar[r, "{\pi_\Y}"]   &   {\Y}  \\
        {\Y}\ar[r, "{\Delta_f}"]    &   {\Y \times_\X \Y,} &
    \end{tikzcd}\]
    where the upper horizontal composite is $g$. The map $\pi_\Y$ is perfect as the base change of the perfect map $f\circ g$, and the map $\Ga_g$ is universally 0-affine also by base change; see \cite[Pr.2.2.2.5(1)]{reconstruction}. Now apply part 2 to the upper composite.
\end{proof}

\begin{prop}\label{pr:main_perfect_stacks_statement}
    Let $f\colon \Y \to \X$ be a perfect map of stacks.
    \begin{enumerate}
        \item Then $f$ satisfies the projection formula, meaning that the canonical map
        \begin{equation}\label{eq:projection_formula}f_\ast(\calM) \otimes \calN \to f_\ast(\calM \otimes f^\ast(\calN))\end{equation}
        is an equivalence for all $\calM \in \QCoh(\Y)$ and all $\calN \in \QCoh(\X)$.
        \item\label{it:adjointableperfect_stacks} If $g\colon \X' \to \X$ is another map of stacks, then the Cartesian square
        \[\begin{tikzcd}
            {\Y'}\ar[r, "{f'}"]\ar[d, "{g'}"]   &   {\X'}\ar[d, "{g}"]  \\
            {\Y}\ar[r, "{f}"]   &   {\X}
        \end{tikzcd}\]
        is adjointable, meaning the natural Beck--Chevalley transformation
        \[g^\ast f_\ast \Rightarrow f'_\ast g'^\ast \colon \QCoh(\Y) \to \QCoh(\X')\]
        is an equivalence.
        \item Then $f$ is $0$-semiaffine, meaning that the pushforward functor
        \[f_\ast \colon \QCoh(\Y) \to \QCoh(\X)\]
        preserves colimits. In particular, $f_\ast$ admits a right adjoint
        \[f^{(1)} \colon \QCoh(\X) \to \QCoh(\Y),\] and if $\X=\Spec A$ is affine, then $\O_\Y \in \QCoh(\Y)$ is compact. 
    \end{enumerate}
\end{prop}

This proof uses many of the arguments of \cite[\textsection3.2]{perfectstacks} verbatim. 

\begin{proof}
    Suppose that $\X=\Spec A$ is affine. In this case, $\Y$ is perfect by definition, hence the global sections functor $\Ga \colon \QCoh(\Y) \to \Sp$ preserves colimits as it is represented by the obviously dualisable and therefore also compact unit $\O_\Y$. As $\Ga$ factors through $\Mod_A$ via the pushforward $f_\ast$ and the conservative and colimit-preserving forgetful functor $\Mod_A \to \Sp$, we see $f_\ast$ preserves colimits if $\X$ is affine. In particular, the functors $f_\ast \calM \otimes -$ and $f_\ast(\calM \otimes f^\ast(-))$ are both colimit-preserving endomorphisms of $\Mod_A$, hence are determined by their value on the generator $A$, where they both naturally take the value of $f_\ast \calM$. This proves the projection formula when $\X$ is affine. Let us further suppose that $\X' = \Spec B$ is affine. This means that $g$, and by base change also $g'$, are affine too. The relative spectrum equivalence of \cite[Th.2.2.3.1]{reconstruction} identifies $\X' = \relSpec_\X C$ for $C = g_\ast B$ in $\CAlg(\QCoh(\X))$ and $g_\ast$ induces the equivalence
    \[\QCoh(\Y') \simeq \Mod_C(\QCoh(\Y)).\]
    Likewise, we have $\Y' \simeq \relSpec_\Y C'$ with $C' = f^\ast C \simeq f^\ast g_\ast B$. This also shows that $g^\ast$ and $g'^\ast$ are given by tensoring with $C$ and $C'$, respectively, and that $f'_\ast$ is a linearised version of $f_\ast$. In particular, this shows that $f'_\ast g'^\ast(\calM) \simeq f_\ast(\calM \otimes C')$. This, combined with the projection formula, then identifies, for any $\calM$ in $\QCoh(\Y)$,
    \[f'_\ast g'^\ast (\calM) \simeq f_\ast(\calM \otimes f^\ast g_\ast B) \simeq f_\ast(\calM) \otimes_A g_\ast(B) \simeq g^\ast f_\ast(\calM) \otimes_B B \simeq g^\ast f_\ast(\calM).\]
    In other words, we have parts 1-3 above when $\X$ and $\X'$ are affine.

    Generally, consider a presentation $\X \simeq \colim \Spec A$ as a colimit of affines and write $\Y_A= \Y \times_\X \Spec A$ and $\Y \simeq \colim \Y_A$. By \cite[Lm.2.2.1.9]{reconstruction} and the restricted cases above, we see that $f_\ast$ preserves colimits and that every Cartesian square
    \[\begin{tikzcd}
        {\Y_A}\ar[r, "{f'}"]\ar[d, "{g'}"]  &   {\Spec A}\ar[d, "g"] \\
        {\Y}\ar[r, "f"] &   {\X}
    \end{tikzcd}\]
    is adjointable. This proves part 3. For part 1, the projection formula, we note that it suffices to prove that (\ref{eq:projection_formula}) is an equivalence after pullback along $g^\ast$ for each $g\colon \Spec A \to \X$ in our fixed presentation of $\X$. Courtesy of the adjointability of the above Cartesian square, this identifies the pullback of (\ref{eq:projection_formula}) with
    \[f'_\ast g'^\ast \calM \otimes g^\ast \calN \simeq g^\ast f_\ast \calM \otimes g^\ast\calN \simeq g^\ast(f_\ast \calM \otimes \calN)\]
    \[\to g^\ast(f_\ast(\calM \otimes f^\ast \calN)) \simeq f'_\ast(g'^\ast \calM \otimes g'^\ast f^\ast \calN) \simeq f'_\ast(g'^\ast \calM \otimes f'^\ast g^\ast \calN),\]
    so the canonical projection formula map for $f'$. This is an equivalence by the restricted version of part 1 proven for affine $\X$. A similar argument shows part 2 in general, by identifying the pullback of the Beck--Chevalley map along $\Spec A \to \X$ with a Beck--Chevalley map in the case where $\X$ is affine.
\end{proof}

An alternative form of part 3 above is the next result.

\begin{prop}\label{pr:pullback_has_left_adjoint_perfect_stacks}
    Let $f\colon \Y \to \X$ be a map between perfect stacks. Then $f$ is 0-semiaffine, and the following are equivalent:
    \begin{enumerate}
        \item $f_\ast \colon \QCoh(\Y) \to \QCoh(\X)$ preserves compact objects.
        \item $f^\ast\colon \QCoh(\X) \to \QCoh(\Y)$ admits a left adjoint.
        \item The right adjoint $f^{(1)}$ to $f_\ast$ commutes with small coproducts.
    \end{enumerate}
\end{prop}

\begin{proof}
    The 0-semiaffineness if simply \cite[Th.1.3]{BalmerDellAmbrogioSanders_2016_GrothendieckNeemanDualityWirthmuller} and the equivalences follow from Neeman's criteria listed in \cite[Th.1.7(GN2)]{BalmerDellAmbrogioSanders_2016_GrothendieckNeemanDualityWirthmuller}.
\end{proof}

Next, we have a Künneth theorem for quasi-coherent sheaves of products of perfect stacks.

\begin{prop}\label{pr:general_kuenneth}
    Let $\Y,\X$ be perfect stacks and $f\colon \Y \to \X$ be a map between them. Suppose that $\X$ has universally $0$-affine diagonal. Then for any map of stacks $\X' \to \X$, the functor induced by pullbacks along projections
    \[\QCoh(\X') \otimes_{\QCoh(\X)} \QCoh(\Y) \xrightarrow{\simeq} \QCoh(\X' \times_\X \Y)\]
    is an equivalence. In particular, the product of perfect stacks is perfect.
\end{prop}

This statement and proof are inspired by \cite[\textsection3.Pr.3.5.3]{gaitsgoryrosenblum}.

\begin{proof}
    As $\X$ and $\Y$ are perfect, their categories of quasi-coherent sheaves are \emph{rigid} in the sense of \cite[\textsection1.Df.9.1.2]{gaitsgoryrosenblum}; this is clear through the lens of \cite[\textsection1.Lm.9.1.5]{gaitsgoryrosenblum}, for example. An application of \cite[\textsection1.Pr.9.4.4]{gaitsgoryrosenblum} shows that $\QCoh(\Y)$ is dualisable over $\QCoh(\X)$ if and only if $\QCoh(\Y)$ is dualisable as a stable category, which follows again from the fact that $\Y$ is perfect. This implies that the endofunctor
    \[\QCoh(\Y) \otimes_{\QCoh(\X)} - \colon \Mod_{\QCoh(\X)}(\PrLst) \to \Mod_{\QCoh(\X)}(\PrLst)\]
    preserves limits. By choosing an affine colimit-presentation of $\X' \simeq \colim \Spec A$, we are then reduced to the case where $\X' = \Spec A$ is affine. As $\X$ has universally $0$-affine diagonal, we see that $g\colon \X' \to \X$ is universally $0$-affine, and by base change, so is $g'\colon \Y' = \X' \times_\X \Y \to \Y$. By \cite[Pr.2.2.2.5(2)]{reconstruction}, this implies $g$ and $g'$ are $0$-affine, which leads to the identifications
    \[\QCoh(\X')\otimes_{\QCoh(\X)} \QCoh(\Y) \simeq \Mod_{g_\ast \O_{\X'}}(\QCoh(\X)) \otimes_{\QCoh(\X)} \QCoh(\Y)\]
    \[ \simeq \Mod_{f^\ast g_\ast \O_{\X'}}(\QCoh(\Y)) \simeq \Mod_{g'_\ast f'^\ast \O_{\X'}}(\QCoh(\Y)) \simeq \Mod_{g'_\ast \O_{\Y'}}(\QCoh(\Y)) \simeq \QCoh(\Y'),\]
    where the third equivalence uses the adjointability of universally $0$-affine maps from \cite[Pr.2.2.2.5]{reconstruction}. This gives our desired conclusions. For the ``in particular'' statement, set $\X=\Spec \Sph$ as the terminal stack, from which the equivalence
    \[\QCoh(\X') \otimes \QCoh(\Y) \xrightarrow{\simeq} \QCoh(\X' \times \Y)\]
    shows that $\X'\times \Y$ is perfect as both $\X'$ and $\Y$ are.
\end{proof}

While it is unclear whether one can test perfectness locally on the target, one of the key base change properties, the adjointability from \Cref{pr:main_perfect_stacks_statement}(\ref{it:adjointableperfect_stacks}), can be checked locally.

Recall from \cite[Df.2.2.1.8]{reconstruction} that a commutative diagram of stacks
\begin{equation}\label{eq:adjointable_map_of_stacks}\begin{tikzcd}
    {\Y'}\ar[r, "{f'}"]\ar[d, "{h'}"]   &   {\X'}\ar[d, "{h}"]  \\
    {\Y}\ar[r, "{f}"]   &   {\X}
\end{tikzcd}\end{equation}
is \emph{adjointable} if the Beck--Chevalley map
\[h^\ast f_\ast \Rightarrow f'_\ast h'^\ast \colon \QCoh(\Y) \to \QCoh(\X')\]
is an equivalence.

\begin{mydef}\label{def:universally_adjointable}
    Call a map \(f:\sfY\to\sfX\) of stacks {\it universally adjointable} if, in every diagram
\[\begin{tikzcd}
	{\sfY''} & {\sfX''} \\
	{\sfY'} & {\sfX'} \\
	\sfY & \sfX
	\arrow[from=1-1, to=1-2]
	\arrow[from=1-1, to=2-1]
	\arrow["\lrcorner"{anchor=center, pos=0.125}, draw=none, from=1-1, to=2-2]
	\arrow[from=1-2, to=2-2]
	\arrow[from=2-1, to=2-2]
	\arrow[from=2-1, to=3-1]
	\arrow["\lrcorner"{anchor=center, pos=0.125}, draw=none, from=2-1, to=3-2]
	\arrow[from=2-2, to=3-2]
	\arrow["f", from=3-1, to=3-2]
\end{tikzcd}\]
    where both squares are pullbacks, the top square is adjointable.
\end{mydef}

Perfect and universally 0-affine maps of stacks are universally adjointable by \Cref{pr:main_perfect_stacks_statement} and \cite[Pr.2.2.2.5(1)]{reconstruction}, respectively..

\begin{lemma}\label{lem:uni_adjointable_local_finitary}
    The class of universally adjointable maps is local and finitary in the sense of \cite[Df.A.9]{Tokic_2026_FamilyCompletionTheorem}, meaning this class contains isomorphisms, is closed under base change, composition, and arbitrary coproducts, and satisfies fpqc-descent. 
\end{lemma}

In particular, universally adjointable maps satisfy descent à la \cite[Df.A.1]{Davies_galois_geometric_FP}.

\begin{proof}
    It is clear from the definition that the class of universally adjointable maps contains all equivalences, and is stable under composition and base changes. That this class is finitary is a direct consequence of disjointness of coproducts, as arbitrary coproducts of adjointable squares are adjointable. It remains to show some form of descendability. Let \(f:\sfY\to\sfX\) be a map, and let \(\sfZ\to\sfX\) be an effective epimorphism such that the base change \(f_{\sfZ}:\sfY_{\sfZ}\to \sfZ\) is universally adjointable.
    Let \(\sfZ^\bullet\) be the Čech nerve of \(\sfZ\to\sfX\).
    It follows that all the base changes \(\sfY_{\sfZ^\bullet}\to\sfZ^\bullet\) are universally adjointable. Now, let \(\sfX''\to \sfX'\to \sfX\) and consider the associated pullbacks as in \Cref{def:universally_adjointable}.
    Also, denote the base changes of \(\sfZ\) to \(\sfX'\) and \(\sfX''\) by \(\sfZ'\) and \(\sfZ''\), with Čech nerves \(\sfZ'^\bullet\) and \(\sfZ''^\bullet\), and denote the base changes of \(\sfY'\) and \(\sfY''\) to these by \(\sfY'_{\sfZ'^\bullet}\) and \(\sfY''_{\sfZ''^\bullet}\).
    Universality of colimits implies that our desired square is the colimit of the pullback squares
\[\begin{tikzcd}
	{\sfY''_{\sfZ''^\bullet}} & {\sfZ''^\bullet} \\
	{\sfY'_{\sfZ'^\bullet}} & {\sfZ'^\bullet}
	\arrow[from=1-1, to=1-2]
	\arrow[from=1-1, to=2-1]
	\arrow["\lrcorner"{anchor=center, pos=0.125}, draw=none, from=1-1, to=2-2]
	\arrow[from=1-2, to=2-2]
	\arrow[from=2-1, to=2-2]
\end{tikzcd}\]
    As they are base changed from \(\sfY_\sfZ\to\sfZ\), the horizontal maps above are universally adjointable. By \cite[Lm.A.21]{Tokic_2026_FamilyCompletionTheorem}, we are done.
\end{proof}

\section{Review of equivariant elliptic cohomology theories}\label{sec:review_of_GM}
In this section, we recall the construction of the equivariant elliptic cohomology associated to an oriented elliptic curve $\sfE$ over a stack $\sfM$, due to Gepner--Meier \cite{davidandlennart,gm2}. This then leads to the definition of the $G$-equivariant global sections functor featuring in our main theorem.

The content of this section will appear in Gepner--Meier's forthcoming work \cite{gm2}; we will update the references below in due time.

\subsection{Equivariant homotopy theory}\label{ssec:EHT}
First, we establish some notation.

\begin{mydef}
	Write $\Glo$ for the \emph{global orbit category}; see \cite[Df.2.7]{davidandlennart}, where $\Glo$ is written as $\Orb$. Recall that its objects are given by \emph{orbit stacks} $\B G$ where $G$ is a compact Lie group, while the morphisms are given by continuous group homomorphisms modulo conjugation in the target. Write $\Spc_{\gl} = \Fun(\Glo^\op, \Spc)$ for the category of \emph{global spaces}.
\end{mydef}

A map $f\colon X \to Y$ in $\Spc_{\gl}$ is \emph{faithful} if it is right orthogonal to the morphisms $\B G \to \B G/N$ for all closed normal subgroups $N\subset G$. Note that a map between orbit stacks is faithful if and only if it corresponds to the inclusion of a closed subgroup.  
    
\begin{mydef}
	By \cite[Pr.2.16]{davidandlennart}, the orbit category $\Orb_G$ of a compact Lie group $G$ is equivalent to the full subcategory of $\Glo_{/\B G}$ spanned by the faithful maps $\B H \to \B G$. Using Elmendorf's theorem identifying the category of $G$-spaces $\Spc_G$ with the category of presheaves on $\Orb_G$, write the \emph{global quotient} functor
	\[[-/G]\colon \Spc_G\to \Spc_{\gl}, \quad X\mapsto [X/G]\]
    for the left Kan extension along the $\Orb_G\subset \Glo_{/\B G} \to \Glo$
\end{mydef}
	
\begin{mydef}\label{df:finite_global_quotient}
	A global space $X$ is a \emph{finite global quotient} if it is equivalent to $[Y/G]$ for some compact Lie group $G$ and finite $G$-space $Y\in \Spc^\fin_G$. The full subcategory of finite global quotients is denoted by $\Spc_{\gl}^{\mathrm{fgq}}$.
\end{mydef}

Now we isolate the particular family of compact Lie groups of interest.

\begin{mydef}\label{df:applicableCLG}
	A compact Lie group $G$ is \emph{applicable} if it is connected and for a maximal torus $T$, every abelian subgroup $H\subset G$ with $\pi_0 H$ generated by at most 2 elements is conjugate to a subgroup of $T$. 
\end{mydef}
	
\begin{example}\label{ex:borel_reference}
	By a theorem of Borel \cite[Cor.3.9]{Borel_61_Lie}, every connected compact Lie group with torsion-free fundamental group is applicable. This includes the family of groups $U(n),SU(n), Sp(n)$, and $Spin(n)$, as well as any finite product of such groups. 
\end{example}

The above definition for applicable Lie groups obscures the role they play here. In \cite
{gm2}, it is shown that for $G$ applicable, the map $\B T \to \B G$ induced by the inclusion of a maximal torus is a ``cover'' from the perspective of equivariant elliptic cohomology. This allows us to descend the excellent properties of equivariant elliptic cohomology at a torus to equally excellent properties of equivariant elliptic cohomology for all applicable compact Lie groups. One illustration of this is \Cref{thm:Ell_geo_nice_prop}.

\subsection{The Gepner--Meier geometric model}\label{ssec:GM_geometric_model}
Recall the notion of an elliptic curve over an $\E_\infty$-ring from \cite[Df.2.0.2]{ec1}, the further notion of an oriented elliptic curve from \cite[Df.7.2.7]{ec2}, and the \emph{moduli stack of oriented elliptic curves} $\M_\Ell^\ori$ of \cite[Pr.7.2.10]{ec2}.

\begin{mydef}
    An \emph{oriented elliptic curve} over a stack $\M$ is an abelian group object $\sfE$ in $\Stk_{/\M}$ equipped with a map of abelian group objects from the constant stack on $BU(1)$ such that for each $\Spec A \to \sfM$, the base change of $\sfE$ to $\Spec A$ defines an oriented elliptic curve over $A$.
\end{mydef}

Equivalently, an oriented elliptic curve over $\M$ is a map of stacks $\M \to \M_\Ell^\ori$. Here we do not differentiate between $\M_\Ell^\ori$ in $\SpDM^\nc$ and its image under $h_{(-)} \colon \SpDM^\nc \to \Stk$ as $\M_\Ell^\ori$ has affine diagonal; see \cite[Prs.3.2.2.2(3) \& 3.2.2.4]{reconstruction} and \Cref{lm:mellori_affine_diagonal}. The equivalence of these notions is formal and will appear in \cite{geometricnorms}.

\begin{theorem}\label{thm:Ell_geo_nice_prop}
	Let $\sfE$ be an oriented elliptic curve over a stack $\M$. Then there exists a functor 
	\[	\sfE[-]\colon \Spc_{\gl}^{\mathrm{fgq}}\to \Stk_{/\M}	\] 
	satisfying the following properties:
	\begin{enumerate}
		\item\label{it:agree_on_abelian} Given an abelian compact Lie group $G$, $\sfE[\B G]$ is given by $\HOM(\widehat{G},\sfE)$, the stack of abelian group homomorphisms from the Pontryagin dual of $G$ to $\sfE$.
		\item\label{it:Ell_is_flat_and_good} Let $G$ be an applicable compact Lie group $G$. Then the stack $\sfE[\B G]$ is geometric (\cite[Df.3.4.1.1]{reconstruction}). If $\M$ is a geometric stack, then $\sfE[\B G] \to \M$ is flat (\cite[Df.3.4.2.1]{reconstruction}).
		\item\label{it:base_change_in_stacks} $\sfE[-]$ is natural in base change of elliptic curves: for every map $f\colon \M' \to \M$ of stacks, there is a natural equivalence
		\[
		(f^* \sfE)[X] \simeq f^*(\sfE[X]).
		\]
		Moreover, $\sfE[-]$ satisfies fpqc descent with respect to such maps of stacks.
		\item\label{it:Ell_closed_embedding_affine} If $X\to Y$ is a faithful map of finite global quotients, then $\sfE[X]\to \sfE[Y]$ is affine.
		\item\label{it:Ell_product_preserving} $\sfE[-]$ is product preserving.
		\item\label{it:base change_for_global_quotients} $\sfE[-]$ preserves pullback squares of the form
		\[\begin{tikzcd}
			{f^*[X/G]} & {[X/G]} \\
			\B H & \B G,
			\arrow[from=1-1, to=1-2]
			\arrow[from=1-1, to=2-1]
			\arrow[from=1-2, to=2-2]
			\arrow["f"', from=2-1, to=2-2]
		\end{tikzcd}\]
		where $G$ is applicable, $f$ is any homomorphism, and $X$ is a finite $G$-space.
	\end{enumerate}
\end{theorem}
	
\begin{proof}
Our $\sfE[-]$ is $\Ellst_{\sfE/\M}(-)$, as defined in \cite{gm2}.
	\begin{itemize}
		\item[(\ref{it:agree_on_abelian})] This is immediate from the construction; see \Cref{rmk:construction_of_Ellst} below.
		\item[(\ref{it:Ell_is_flat_and_good})] Both statements follow from the abelian case by descent to the maximal torus.
		\item[(\ref{it:base_change_in_stacks})] This is proven in \cite
        {gm2}.
		\item[(\ref{it:Ell_closed_embedding_affine})] This follows from descent for affine morphisms, see \cite[Pr.A.3]{Davies_galois_geometric_FP}, from the case of $G$ abelian, which is \cite[Pr.6.3]{davidandlennart}. 
			
		\item[(\ref{it:Ell_product_preserving})] This is proven in \cite
        {gm2}.
			
		\item[(\ref{it:base change_for_global_quotients})] Similarly, this is proven in \cite
        {gm2}.\qedhere
	\end{itemize}  
\end{proof}

The construction of $\Ellst$ given in \cite{gm2} is surprisingly subtle; we give a summary here.
    
\begin{remark}\label{rmk:construction_of_Ellst}
	Let $\sfE$ be an oriented elliptic curve over a nonconnective spectral scheme $\M$. 
		\begin{enumerate}
		\item One begins by applying \cite[Construction 3.13]{davidandlennart} to the oriented elliptic curve $\sfE$, which we view as a strict abelian group object in spectrally ringed topological spaces. This results in the functor
			\[
			\Ell(-)\colon \Glo_{\ab}\rightarrow \Top^{\mathrm{loc}}_{\CAlg/\M}; \qquad \B G \mapsto \Hom(\dual{G}, \sfE),
			\]
            where $\Glo_\ab \subset \Glo$ is the subcategory spanned by orbit stacks of abelian compact Lie groups.
			\item We restrict this to the subcategory $\Glo_{\ab,2}$ spanned by those $BG$ where $\pi_0 G$ has at most $2$ generators, and then left Kan extend to the full subcategory $\Spc^{\abfgq}_{\twoabgl}$ of the associated presheaf category $\Spc_{\twoabgl}$ spanned by global quotients $[X/G]$ where $X$ is compact and $G$ is abelian. This yields a functor
			\[
			\Ell(-)\colon \Spc^{\abfgq}_{\twoabgl}\to \Top^{\mathrm{loc}}_{\CAlg/\M}.
			\]
			\item We then pass through the functor of points embedding $\Top^{\mathrm{loc}}_{\CAlg/\M}\subset \Stk_\M$ and further left Kan extend to the subcategory spanned by global quotients $[X/G]$ such that $X$ is compact and $G$ is arbitrary. This yields a functor 
			\[
			\Ell^{\st}(-)\colon \Spc^{\fgq}_{\twoabgl}\to \Stk_\M.
			\]
			\item Precomposition with the restriction functor $\Spc^{\fgq}_{\gl} \to \Spc^{\fgq}_{\twoabgl}$ produces our desired functor
			\[
			\sfE[-]\colon \Spc^{\fgq}_{\gl}\to \Stk_S
			\]
            when the base $\M$ is a nonconnective spectral scheme.
            \item Finally, for an elliptic curve $\E$ over a general $\M$, we define $\sfE[-]$ by descent: choosing a colimit representation of $\M \simeq \colim \Spec A$ by affines and writing $\sfE_A = \sfE \times_\M \Spec A$, we define
            \[\sfE[-] = \colim \sfE_A[-].\qedhere\]          
		\end{enumerate}
\end{remark}

\begin{notation}
    Let us write $\sfE[\B G]$ as $\sfE_G$. This will help when we want to think of it as functorial in morphisms of compact Lie groups, as opposed to morphisms in the global orbit category $\Glo$. 
\end{notation}
	
The previous theorem is too general to give us adequate control of $\sfE_G$. The following result is then crucial in our understanding of the stacks $\sfE_G$ for particular $G$.
	 
\begin{theorem}\label{thm:extra-applicable}
Let $\sfE$ be an oriented elliptic curve over a strictly even-periodic Noetherian $\E_\infty$-ring $A$; see \Cref{def:even_periodic}. Let $G$ be an applicable compact Lie group and fix a maximal torus $T$. Suppose that the geometric realisation of the simplicial object $\sfE[\B T^{\times_{\B G} \bullet}]^\heartsuit$ in $\Stk^\heartsuit$ is a scheme. Then $\sfE_G$ is an even-periodic qcqs spectral scheme.
Moreover, let \(W\) be the Weyl group of \(i:T\subset G\).
Then the map
\[\sfE_i^\heartsuit\colon \sfE_T^\heartsuit\to \sfE_G^\heartsuit\]
of classical schemes factors uniquely through the categorical quotient \((\sfE_T^\heartsuit)/W\) (which exists by \cite[Rmk.4.9(5) \& Rmk.2.8]{Rydh_2013_ExistencePropertiesGeometric}), and the resulting map
\[(\sfE_T^\heartsuit)/W\to \sfE_G^\heartsuit\]
is an equivalence.
Moreover, these assumptions hold for $G = U(n)$, $SU(n)$, or $Sp(n)$, as well as any finite product of such groups.
\end{theorem}

\begin{remark}\label{ex:u(n)_extra_applicable}
Let us unravel the $G=U(n)$-case. Consider the maximal torus given via the diagonal matrices $U(1)^n \to U(n)$ with Weyl group $W= \Sigma_n$. Then \Cref{thm:extra-applicable} identifies
\[\sfE_{U(n)}^\heartsuit \simeq (\sfE^\heartsuit)^{\times n}/\Sigma_n\]
with the symmetric product. It is a classical fact, reviewed around (\ref{eq:symprod_Hilbert_schemes}), that this symmetric product is equivalent to the Hilbert scheme of $n$-points on $\sfE^\heartsuit$. Moreover, it follows from the definitions that the inclusion $U(1)^n \to U(1)^{n+1}$ into the first $n$-factors induces the map $\sfE^{\times n} \to \sfE^{\times (n+1)}$ into the first $n$-factors. This map is also commutes with the symmetric group actions on both sides, restricting the $\Sigma_{n+1}$-action on the right-hand side along the inclusion induced by choosing the first $n$-letters. By \Cref{thm:extra-applicable}, we see that the inclusion $U(n) \to U(n+1)$ sending $A$ to the diagonal block matrix with $A$ in the upper-left and the singleton $1$ in the lower-right, induces the map
\[\sfE_{U(n)}^\heartsuit \simeq (\sfE^\heartsuit)^{\times n}/\Sigma_n \to (\sfE^\heartsuit)^{\times (n+1)}/\Sigma_{n+1} \simeq \sfE_{U(n+1)}^\heartsuit\]
taking the quotient of the map described above; also see \Cref{prop:Hilb_Schemes_ECD_affine_complement}.
\end{remark}

The cases for $SU(n)$ and $Sp(n)$ are similar, with Hilbert schemes being replaced with related quotients of the $n$-fold products of $\sfE^\heartsuit$; this classical material is reviewed in \Cref{sec:appendix_classical_fun_stuff}. This ensures some of the functoriality of $\sfE_G$ in $G$ can be tracked; see \Cref{ssec:explicit_sequences_of_compact_liegroups}.
		
\begin{proof}
The first part is \cite{gm2}. 
To conclude that $\sfE_G$ is qcqs, it suffices to prove this for $\sfE_G^\heartsuit\simeq (\sfE_T^\heartsuit)/W$. This in turn follows from the fact that $\sfE_T^\heartsuit$ is qcqs and the map $\sfE_T^\heartsuit \to (\sfE_T^\heartsuit)/W$ is an affine cover \cite[Th.4.4]{Rydh_2013_ExistencePropertiesGeometric}; in fact, by \cite[Ex.4.9]{Rydh_2013_ExistencePropertiesGeometric}, the quotient is a projective scheme. That the assumptions hold for finite products of unitary, special unitary, and symplectic groups, independently of \(\sfE\), is proven in \cite{gm2}.
\end{proof}

\begin{remark}
	Gepner--Meier conjecture that for all simple and simply connected compact Lie groups $G$ of rank $n$, the classical stack $\sfE_G^\heartsuit$ is smooth over the base $\M^\heartsuit$ with geometric fibres given by weighted projective stacks $\mathbf{P}(g_1, \cdots, g_n)$ for positive integers $g_i$. Moreover, they prove this is the case for $G=SU(n)$ and $Sp(n)$, which implicitly plays into our proof of \Cref{maintheorem_universal}. Many of our statements building up to equivariant affineness here still hold if $\sfE_G$ is perfect and this conjecture holds.
\end{remark}

\subsection{Equivariant elliptic (co)homology}\label{ssec:EEcohomology}
Using the stacks $\sfE_G$, one can define equivariant elliptic cohomology at a fixed group.
    
\begin{mydef}
	Let $G$ be a compact Lie group and let $X \in \Spc_G^{\fin}$ be a finite $G$-space. Define $\ELL_G(X) := p_* \O_{\sfE[[X/G]]} \in \QCoh(\sfE_G)$
    where 
	\[
	p_\ast \colon \QCoh(\sfE[[X/G]]) \to \QCoh(\sfE_G)
	\] is the pushforward of quasi-coherent sheaves induced by the map $p\colon [X/G]\to [\ast/G] \simeq \B G$. As explained in \cite[\S 7]{gepner2024global2ringsgenuinerefinements}, this defines an \emph{elliptic cohomology} functor
	\[
	\ELL_G\colon (\Spc_G^\fin)^{\op}\to \QCoh(\sfE_G).\qedhere
	\]
\end{mydef}
	
These functors satisfy a host of useful properties, summarised by the following result.
	
\begin{theorem}\label{thm:functor/naturality_GM_theories}
	Let $\sfE$ be an oriented elliptic curve over $\M$. 
	\begin{enumerate}
		\item For all compact Lie groups $G$, the functor 
		\[
		\ELL_G\colon (\Spc_G^\fin)^{\op}\to \QCoh(\sfE_G).
		\]
		is a limit-preserving, lax symmetric monoidal functor.
		\item\label{it:induction_space_cohomology} Let $i\colon H\subset G$ be a closed embedding of compact Lie groups. Then the square
		\[\begin{tikzcd}
			{(\Spc^{\fin}_H)^{\op}} & {\QCoh(\sfE_H)} \\
			{(\Spc^{\fin}_G)^{\op}} & {\QCoh(\sfE_G)}
			\arrow["{\ELL_H}", from=1-1, to=1-2]
			\arrow["{-\times_H G}"', from=1-1, to=2-1]
			\arrow["{i_*}", from=1-2, to=2-2]
			\arrow["{\ELL_G}", from=2-1, to=2-2]
		\end{tikzcd}\]
		commutes, where $i_*$ denotes the pushforward along $\sfE[\B i]$.
		\item Let $\phi\colon H\to G$ be a homomorphism of compact Lie groups such that $G$ is \emph{applicable}. Then the square 
		\[\begin{tikzcd}
			{(\Spc^{\fin}_G)^{\op}} & {\QCoh(\sfE_G)} \\
			{(\Spc^{\fin}_H)^{\op}} & {\QCoh(\sfE_H)}
			\arrow["{\ELL_G}", from=1-1, to=1-2]
			\arrow["{\phi^*}"', from=1-1, to=2-1]
			\arrow["{\phi^*}", from=1-2, to=2-2]
			\arrow["{\ELL_H}"', from=2-1, to=2-2]
		\end{tikzcd}\]
		commutes, where $\phi^*$ denotes on the left the restriction functor along the homomorphism $\phi\colon H\to G$ and on the right the pullback along $\sfE[\B \phi]\colon \sfE_H\to \sfE_G$.
		\item $\ELL_G$ is strong symmetric monoidal when $G$ is \emph{applicable}.
		\item\label{it:rep_sphere_invertible} Let $V$ be a finite dimensional real representation of $G$. Then the object 
		\[
		\ELL_G(S^V, \infty) := \fib(\Ell_G(S(V)) \to \ELL_G(\ast))
		\]
		is a tensor-invertible object of $\QCoh(\sfE_G)$.
        \item\label{it:naturality_of_cohomology_in_stacks} Let $G$ be a compact Lie group and $f\colon \M' \to \M$ a map of stacks. Then the triangle
        \[\begin{tikzcd}
                &      {(\Spc^{\fin}_G)^{\op}}\ar[dl, "{\ELL_G}", swap]\ar[rd, "{\ELL_G'}"]    &   \\
            {\QCoh(\sfE_G)}\ar[rr, "{f^\ast}"]   &&   {\QCoh(\sfE'_G)}
        \end{tikzcd}\]
        commutes, where $\sfE'$ is the base change of $\sfE$ to $\M'$, $f'\colon \sfE_G' \to \sfE_G$ the base change of $f$, and $\ELL_G'$ the equivariant elliptic cohomology theory associated with $\sfE'$.
	\end{enumerate}
\end{theorem}
	
\begin{proof}
	For part (1), we first observe that \cite[Pr.8.1]{gepner2024global2ringsgenuinerefinements} gives $\ELL_G$ its lax symmetric monoidal structure. By construction, $\ELL_G$ is Kan extended from abelian global quotients, and so it suffices to prove that it limit-preserving on abelian global quotients over $\B G$. This follows by the same argument as \cite[Lemma 6.5]{davidandlennart}. Part (2) essentially holds by definition. 
    Part (3) follows from \Cref{thm:Ell_geo_nice_prop}(\ref{it:base change_for_global_quotients}). 
    Since $\Ellst[-]$ preserves products, we immediately obtain an external Künneth isomorphism for equivariant elliptic cohomology. Part (4) follows from this by combining this with part (3) for the diagonal $\Delta\colon G\to G\times G$. Part (5) is contained in \cite{gm2}. 
    For part (6), since both $f^\ast \ELL_G$ and $\ELL_G'$ are exact functors between stable categories, to show they agree on $\Spc_G^\fin$, it suffices to check on orbits $G/H$ for each closed subgroup $H\subset G$. In this case, we can identify
    \[f^\ast \ELL_G(G/H) = f^\ast i_\ast \O_{\sfE_H} \simeq i'_\ast f'^\ast \O_{\sfE_H} \simeq i'_\ast \O_{\sfE_H'} = \ELL_G'(G/H),\]
    where $i\colon \sfE_H \to \sfE_G$ and $i'\colon \sfE_H' \to \sfE_G'$ are both affine by \Cref{thm:Ell_geo_nice_prop}(\ref{it:Ell_closed_embedding_affine}) which allows for the Beck--Chevalley equivalence used above by \cite[Pr.2.2.2.5]{reconstruction}.
\end{proof}
	
\begin{remark}
	By unwinding the definitions, we see that $\ELL_G(G/H) = i_*\O_{\sfE_H}$, where $i\colon \sfE_H\to \sfE_G$ is the map on elliptic cohomology induced by the closed embedding $i\colon H\subset G$. Since this is affine by \Cref{thm:Ell_geo_nice_prop}(\ref{it:Ell_closed_embedding_affine}), we obtain an isomorphism of affine stacks over $\sfE_G$
	\[
	\sfE_H \simeq \relSpec_{\sfE_G}(\ELL_G(G/H)).\qedhere
	\]
\end{remark}

\begin{remark}\label{rmk:reduced_cohomology}
    As $\QCoh(\sfE_G)$ is a pointed category, there exists a canonical factorisation of $\ELL_G$ from the category of based finite $G$-spaces,
    \[\widetilde{\ELL}_G\colon (\Spc_{G,\ast}^\fin)^\op \to \QCoh(\sfE_G),\]
    where $\widetilde{\ELL}_G(X)$ is the fibre of the induced map $\ELL_G(X) \to \ELL_G(\ast)$.
\end{remark}
	
A crucial observation is that the points (4) and (5) of the previous theorem combine with the universal property of genuine $G$-spectra to give an extension of equivariant elliptic cohomology to genuine $G$-spectra; universality also implies that the naturality squares in points (2) and (3) above also extend.
	
\begin{theorem}\label{thm:naturality_of_elliptic_cohomology}
    Let $\sfE$ be an oriented elliptic curve over $\M$. 
	\begin{enumerate}
		\item Let $G$ be applicable. Then there exists a unique limit-preserving symmetric monoidal functor 
		\[
		\ELL_G(-)\colon (\Sp_G^{\fin})^{\op}\to \QCoh(\sfE_G)
		\]
		extending $G$-equivariant elliptic cohomology along $\Sigma^\infty_+\colon \Spc_G^\fin\to \Sp_G^{\fin}$.
		\item\label{it:ell_cohomolgy_induction} Let $i\colon H\subset G$ be a closed embedding of applicable compact Lie groups. Then the square 
		\[\begin{tikzcd}
			{(\Sp^{\fin}_H)^{\op}} & {\QCoh(\sfE_H)} \\
			{(\Sp^{\fin}_G)^{\op}} & {\QCoh(\sfE_G)}
			\arrow["{\ELL_H}", from=1-1, to=1-2]
			\arrow["{-\times_H G}"', from=1-1, to=2-1]
			\arrow["{i_*}", from=1-2, to=2-2]
			\arrow["{\ELL_G}", from=2-1, to=2-2]
		\end{tikzcd}\]
		commutes.
		\item\label{it:cohomology_spectra_nice_pullback} Let $\phi\colon H\to G$ be a homomorphism of applicable compact Lie groups. Then the square 
		\[\begin{tikzcd}
			{(\Sp^{\fin}_G)^{\op}} & {\QCoh(\sfE_G)} \\
			{(\Sp^{\fin}_H)^{\op}} & {\QCoh(\sfE_H)}
			\arrow["{\ELL_G}", from=1-1, to=1-2]
			\arrow["{\phi^*}"', from=1-1, to=2-1]
			\arrow["{\phi^*}", from=1-2, to=2-2]
			\arrow["{\ELL_H}"', from=2-1, to=2-2]
		\end{tikzcd}\]
		commutes.
        \item Using the notation of \Cref{thm:functor/naturality_GM_theories}(\ref{it:naturality_of_cohomology_in_stacks}) where $G$ is also applicable, then the diagram
        \[\begin{tikzcd}
                &      {(\Sp^{\fin}_G)^{\op}}\ar[dl, "{\ELL_G}", swap]\ar[rd, "{\ELL_G'}"]    &   \\
            {\QCoh(\sfE)}\ar[rr, "{f^\ast}"]   &&   {\QCoh(\sfE')}
        \end{tikzcd}\]
        commutes.
	\end{enumerate}
\end{theorem}

When $G=U(1)$ is the torus of rank 1, part 1 of this theorem appears as \cite[Pr.9.2]{davidandlennart}.
	
\begin{proof}
Combine \Cref{thm:functor/naturality_GM_theories} and the discussion in \cite[\textsection 11]{gepner2024global2ringsgenuinerefinements}; we omit the details.
\end{proof}

\begin{remark}\label{rmk:general_cohomology_definition}
    One can also generalise the construction of $\ELL_G$ above to all compact Lie groups as follows: choose a closed embedding $\phi \colon G \to U(n)$, and consider the composite
    \[(\Sp_G^\fin)^\op \xrightarrow{(\Ind_G^{U(n)})^\op} (\Sp_{U(n)}^\fin)^\op \xrightarrow{\ELL_{U(n)}} \QCoh(\sfE_{U(n)}),\]
    courtesy of the fact that $U(n)$ is applicable. By \Cref{thm:functor/naturality_GM_theories}(\ref{it:induction_space_cohomology}), the $G$-sphere spectrum is sent to $\phi_\ast \O_{\sfE_G}$ under this composite, where $\phi\colon \sfE_G \to \sfE_{U(n)}$ is the induced map of stacks. As induction is op-lax monoidal, its opposite is lax monoidal, and the above composite admits a natural factorisation as
    \[(\Sp_G^\fin)^\op \to \Mod_{\phi_\ast \O_{\sfE_G}}(\QCoh(\sfE_{U(n)}).\]
    By \Cref{thm:Ell_geo_nice_prop}(\ref{it:Ell_closed_embedding_affine}), the map $\phi$ is affine, so by \cite[Pr.2.2.2.5(2)]{reconstruction}, we can identify the codomain of the factorisation above with $\QCoh(\sfE_G)$, giving us the desired functor
    \[\ELL_G \colon (\Sp_G^\fin)^\op \to \QCoh(\sfE_G).\]
    This construction seems to be well known when $G$ is a finite cyclic group using $G \leq U(1)$; see \cite[\textsection2.1]{chua_c2TMF}. We warn that this functor will \textbf{not} generally be symmetric monoidal, meaning one cannot continue to work with $\ELL_G$ through this section unless $G$ is applicable.
\end{remark}

The above construction of equivariant elliptic cohomology then formally dualises.
	
\begin{mydef}\label{def:elliptic_homology}
Let $G$ be an applicable compact Lie group. Write
\[\O_{\sfE_G} \otimes - \colon \Sp_G \to \QCoh(\sfE_G)\]
for the filtered colimit extension, using that $\Ind(\Sp_G^\fin) \simeq \Sp_G$, of the composite
\[\Sp_G^{\fin} \xrightarrow{ \mathbb D(-)} (\Sp_G^{\fin})^{\op} \xrightarrow{\ELL_G(-)} \QCoh(\sfE_G)\]
of $G$-equivariant elliptic cohomology and Spanier--Whitehead duality. As a symmetric monoidal functor in $\PrLst$, it admits a lax symmetric monoidal right adjoint
\[\gGa(\sfE_G, -)\colon \QCoh(\sfE_G)\to \Sp_G.\]
This notation is justified by the computation of its $G$-fixed points
\[\gGa(\sfE_G, \calF)^G \simeq \Gamma(\sfE_G, \calF).\]
In fact, for any closed subgroup $\phi \colon H\to G$, we have
\begin{equation}\label{eq:H-fixed_points}
    \gGa(\sfE_G, \calF)^H \simeq \Ga(\sfE_H, \phi^\ast \calF)
\end{equation}
by \Cref{thm:naturality_of_elliptic_cohomology}(\ref{it:cohomology_spectra_nice_pullback}). In particular, we obtain a lift $\Ga_G \sfE = \gGa(\sfE_G,\O_{\sfE_G}) \in \Sp_G$ of the global sections of $\sfE_G$. As usual, we may pass to an adjunction
\begin{equation}\label{eq:equivariant_adjunction}\O_{\sfE_G} \otimes_{\Gamma_G \sfE} - \colon \Mod_{\Gamma_G \sfE}(\Sp_G) \rightleftarrows \QCoh(\sfE_G) : \gGa(\sfE_G,-),\end{equation}
where again the left adjoint, which we call \emph{$G$-equivariant elliptic homology}, is strong symmetric monoidal.
\end{mydef}

\begin{notation}
Over an affine base $\M = \Spec A$ we write $\Ga_G \sfE = A_G$, suppressing the elliptic curve. At the other extreme, for $\calE$ the universal oriented elliptic curve over $\M_\Ell^\ori$ we also write $\Ga_G \calE = \TMF_G$ for \emph{equivariant topological modular forms}.
\end{notation}

\begin{remark}\label{rmk:TMFG_for_general_G}
As mentioned in the introduction, one can define $\Ga_G \sfE$, and therefore also $\TMF_G$, for all compact Lie groups $G$, and there are a few ways to do this. First, one could take the \emph{global} $\E_\infty$-ring $\gTMF$ of Gepner--Pol and the second-named author \cite[Exs.14.25-26]{gepner2024global2ringsgenuinerefinements}, which comes with an underlying $\E_\infty$-$G$-ring $\gTMF_G$ for each compact Lie group $G$. Alternatively, one could first choose a closed embedding $\phi \colon G \to U(n)$ for some $n\geq 1$, and then define $\TMF_G$ as the restriction of $\TMF_{U(n)}$ along $\phi^\ast \colon \Sp_{U(n)} \to \Sp_{G}$. In fact, both definitions agree. To see this, we note that $\QCoh(\calE_\bullet)$ is $\mathcal{A}$-genuine, where $\mathcal{A}$ is the family of applicable groups. Therefore, one can use \cite{gepner2024global2ringsgenuinerefinements} to obtain a global spectrum ${\gTMF}'$ whose value at $U(n)$ is $\TMF_{U(n)}$ and which manifestly agrees with $\gTMF$ after restricting to abelian compact Lie groups. As $\gTMF$ is right-induced from abelian groups, to identify it with $\gTMF'$, it suffices to prove that $\gTMF'$ is right-induced from abelian groups. This is precisely recorded in \cite
{gm2}. The same ideas also apply to $\Ga_G\sfE$ for more generally.
\end{remark}

The following is the analogy of \Cref{thm:naturality_of_elliptic_cohomology} for equivariant elliptic homology.
    
\begin{prop}\label{pr:homomorphism_and_linear_homology}
	Let $\sfE$ be an oriented elliptic curve over a stack $\M$.
    \begin{enumerate}
        \item For each homomorphism of applicable compact Lie groups $\phi\colon H\to G$, the square
		\[\begin{tikzcd}
			{\Mod_{\Ga_G \sfE}(\Sp_G)} && {\QCoh(\sfE_G)} \\
			{\Mod_{\Ga_H \sfE}(\Sp_H)} && {\QCoh(\sfE_H)}
			\arrow["{\O_{\sfE_G} \otimes_{\Ga_G \sfE} -}", from=1-1, to=1-3]
			\arrow["{\phi^*}"', from=1-1, to=2-1]
			\arrow["{\phi^*}", from=1-3, to=2-3]
			\arrow["{\O_{\sfE_H} \otimes_{\Ga_H\sfE} -}"', from=2-1, to=2-3]
		\end{tikzcd}\]	
     in $\CAlg(\PrLst)$ commutes.
        \item\label{it:homology_commutes_with_maps_of_stacks} Using the notation of \Cref{thm:functor/naturality_GM_theories}(\ref{it:naturality_of_cohomology_in_stacks}), where $G$ is also applicable, then the square
        \[\begin{tikzcd}
			{\Mod_{\Ga_G \sfE}(\Sp_G)} && {\QCoh(\sfE_G)} \\
			{\Mod_{\Ga_G \sfE'}(\Sp_G)} && {\QCoh(\sfE_G')}
			\arrow["{\O_{\sfE_G} \otimes_{\Ga_G \sfE} -}", from=1-1, to=1-3]
			\arrow["{f^*}"', from=1-1, to=2-1]
			\arrow["{f^*}", from=1-3, to=2-3]
			\arrow["{\O_{\sfE_G'} \otimes_{\Ga_G\sfE'} -}"', from=2-1, to=2-3]
		\end{tikzcd}\]
        in $\CAlg(\PrLst)$ commutes.
    \end{enumerate} 
\end{prop}

\begin{proof}
For part 1, passing to equivariant Spanier--Whitehead duals commutes with taking restrictions of compact Lie groups, so we immediately obtain a commutative square
\[\begin{tikzcd}
		{\Sp_G} && {\QCoh(\sfE_G)} \\
		{\Sp_H} && {\QCoh(\sfE_H)}
		\arrow["{\O_{\sfE_G} \otimes -}", from=1-1, to=1-3]
		\arrow["{\phi^*}"', from=1-1, to=2-1]
		\arrow["{\phi^*}", from=1-3, to=2-3]
		\arrow["{\O_{\sfE_H} \otimes -}"', from=2-1, to=2-3]
\end{tikzcd}\]
in $\CAlg(\PrLst)$ from part 3 of \Cref{thm:naturality_of_elliptic_cohomology}. By passing to categories of modules, we obtain the desired commutative square. For part 2, we naturally identify these two composites as both preserve colimits, as all functors are left adjoints, and both agree on the generators of $\Ga_G \sfE \otimes G/H_+$ of $\Mod_{\Ga_G \sfE}(\Sp_G)$ by \Cref{thm:functor/naturality_GM_theories}(\ref{it:naturality_of_cohomology_in_stacks}).
\end{proof}

\begin{remark}
    To be explicit, one can compute the values of the functor
    \[\phi^*\colon \Mod_{\Ga_G \sfE}(\Sp_G) \to \Mod_{\Ga_H \sfE}(\Sp_H); \qquad X \mapsto \Ga_H \sfE \otimes_{\phi^* \Ga_G \sfE} \phi^* X.\]
    When $\phi$ is a closed embedding, one may further identify $\phi^* \Ga_G \sfE \simeq \Ga_H \sfE$.
\end{remark}

\section{Equivariant affineness}\label{sec:affine_case}
The set-up from the previous section allows us to formulate the main concept of this article:

\begin{mydef}
    Let $G$ be an applicable compact Lie group. An oriented elliptic curve $\sfE$ is \emph{$G$-equivariantly affine} if the $G$-equivariant elliptic homology--$G$-equivariant global sections adjunction of (\ref{eq:equivariant_adjunction}) is an equivalence.
\end{mydef}

\subsection{The inductive argument}\label{ssec:inductive_step}
This subsection contains our main inductive argument, which allows us to bootstrap equivariant affineness of $\sfE$ from smaller to larger groups. We begin with a simple, yet very important, fully-faithfulness statement.

\begin{prop}\label{pr:left_adjoint_is_fullyfaithful}
    Let $\sfE$ be an oriented elliptic curve over a stack $\M$ and $G$ an applicable compact Lie group such that $\sfE_G$ is 0-semiaffine. Then the functor 
    \[    \O_{\sfE_G} \otimes_{\Ga_G \sfE} - \colon \Mod_{\Ga_G \sfE}(\Sp_G)\to \QCoh(\sfE_G)    \]
    is fully faithful.
\end{prop}

Keep in mind that the assumptions above hold in the situation of \Cref{thm:extra-applicable} by \Cref{pr:0semiaffine_spectral_schemes}. The proof of this statement uses the following formal fact.

\begin{lemma}\label{lem:source_rigid_fully_faithfulness}
    Let \(F\colon \calC\to \calD\) be a map in \(\CAlg(\PrLst)\), with right adjoint \(G\), such that
    \begin{itemize}
        \item \(\calC\) is compactly generated, and compact and dualisable objects agree,
        \item the tensor unit \(\1_\calD\) is compact, and
        \item the adjunction unit \(\1_\calC\to GF\1_\calC\) is an equivalence.
    \end{itemize}
    Then \(G\) commutes with filtered colimits and \(F\) is fully faithful.
\end{lemma}

\begin{proof}
    The first two assumptions guarantee that \(F\) is an internal left adjoint in \(\Mod_\calC(\PrL)\) by \cite[Obs.4.51]{ramzi_separability}, so \(G\) is itself a left adjoint, and the adjunction \(F\dashv G\) is \(\calC\)-linear.
    This implies that we have an equivalence $\eta_X\simeq X\otimes \eta_{\1_\calC}$ for all $X\in \calC$. Therefore $\eta_X$ is an equivalence since $\eta_{\1_\calC}$ is an equivalence.
\end{proof}

\begin{proof}[Proof of \Cref{pr:left_adjoint_is_fullyfaithful}]
To apply \Cref{lem:source_rigid_fully_faithfulness}, first note that $\Sp_G$ is compactly generated and compacts and dualisables agree, with compact dualisable generators given by suspension spectra of orbits. The second conditions holds by assumption, and the third is automatic from the definitions: $\underline{\Gamma}(\sfE_G, \O_{\sfE_G} \otimes_{\Ga_G \sfE} (\Ga_G \sfE)) = \underline{\Gamma}(\sfE_G, \O_{\sfE_G}) = \Ga_G \sfE$.
\end{proof}

\begin{remark}
    One can also show that the equivariant global sections functor
    \[\gGa(\sfE_G, -) \colon \QCoh(\sfE_G) \to \Mod_{\Ga_G \sfE}(\Sp_G)\]
    preserves colimits in the situation of \Cref{pr:left_adjoint_is_fullyfaithful}. Since we do not need this statement directly, we omit its proof.
\end{remark}

Now we can state our main inductive tool.

\begin{theorem}\label{lem:inductive_step_general}
    Let $\sfE$ be an oriented elliptic curve over an even-periodic \(\E_\infty\)-ring $A$, \(\phi\colon H\to G\) and \(f\colon G\to N\) be maps between applicable Lie groups, with \(\phi\) a closed embedding.
    Further assume the following:
    \begin{enumerate}
        \item The stacks $\sfE_G, \sfE_H$, and $\sfE_N$ are all nonconnective spectral schemes.
        \item $\sfE$ is both $H$- and $N$-equivariantly affine.
        \item The map of underlying classical schemes $\sfE_H^\heartsuit \to \sfE_G^\heartsuit$ is an effective Cartier divisor.
        \item The open complement $\sfE^\heartsuit_G \setminus \sfE_H^\heartsuit$ is affine over \(\sfE_N^\heartsuit\).
    \end{enumerate}
    Then $\sfE$ is $G$-equivariantly affine.
\end{theorem}

Our proof of the above fact uses the following categorical lemma.

\begin{lemma}\label{lem:comonadic_helper}
    Consider the following diagram in \(\Cat^L\), so categories and left adjoint functors, whose right adjoints are denoted to the right of the symbol $\dashv$:
\[\begin{tikzcd}
	{X'} && {Y'} \\
	X && Y \\
	& C
	\arrow["{F'\dashv G'}", from=1-1, to=1-3]
	\arrow["{A_X\dashv B_X}"', from=1-1, to=2-1]
	\arrow["{A_Y\dashv B_Y}", from=1-3, to=2-3]
	\arrow["{F\dashv G}", from=2-1, to=2-3]
	\arrow["{L_X\dashv R_X}"', from=2-1, to=3-2]
	\arrow["{L_Y\dashv R_Y}", from=2-3, to=3-2]
\end{tikzcd}\]
    Denote the composites \(L_X\circ A_X=L'_X\) and \(L_Y\circ A_Y=L'_Y\), with right adjoints \(R'_X\) and \(R'_Y\).
    Assume that \(B_X\) and \(B_Y\) are fully faithful, and that \(L_X\) and \(L_Y\) are comonadic.
    Then \(F\) is an equivalence if and only if the natural transformation
    \begin{equation}\label{eq:comonad_comparison}
    L'_Y\circ F'\circ G'\circ R'_Y\xrightarrow{L'_Y\circ\epsilon\circ R'_Y}L'_Y\circ R'_Y\end{equation}
    is an equivalence, where \(\epsilon\) is the adjunction counit.
\end{lemma}

\begin{proof}
    In what follows, we denote the adjunction unit and counit of any adjunction by \(\eta\) and \(\epsilon\), respectively.
    By (the dual of) \cite[Cor.4.7.3.16 \&  Rmk.4.7.3.17]{ha}, $F$ is an  equivalence if and only if the natural transformation
    \[L_X\circ R_X=L_Y\circ F\circ R_X\xrightarrow{L_Y\circ \eta \circ F\circ R_X}L_Y\circ R_Y\circ L_Y\circ F\circ R_X
    \xrightarrow{L_Y\circ R_Y\circ \epsilon}L_Y\circ R_Y\]
    is an equivalence.
    This expresses that, since \(L_X\) and \(L_Y\) are comonadic, \(F\) is an equivalence if and only if the comparison on their associated comonads induced by \(F\) is an equivalence.
    Since \(B_X\) and \(B_Y\) are fully faithful, their adjunction counits are equivalences, so we can check whether the above is an equivalence after precomposing with the counit of \(A_X\dashv B_X\) and postcomposing with the inverse of the counit of \(A_Y\dashv B_Y\).
    
    We will show that the resulting natural transformation is isomorphic to (\ref{eq:comonad_comparison}), hence proving the proposition. The required manipulations are most easily expressed with string diagrams; see \cite{Marsden_2014_CategoryTheoryUsing,hinze_marsden_string_diagrams}. To employ this language, we note that we may work in the homotopy 2-category. Our convention is that functors act from right to left and transformations from bottom to top.

    Since \(F\) admits a right adjoint, we can rewrite the original transformation as
\begin{equation*}
\begin{tikzpicture}[x=1cm,y=1cm,baseline=(current bounding box.center)]
  \path
    coordinate (bl) at (0,0) coordinate (br) at (4,0) coordinate (tr) at (4,4)
    coordinate (tl) at (0,4) coordinate (v) at (1.5,0.5) coordinate (eta) at (1.5,1.5)
    coordinate (v2) at (2,2) coordinate (epsilon) at (2.5,3) coordinate (p1) at (1.5,0)
    coordinate (p2) at (0.5,4) coordinate (p3) at (1,4) coordinate (p4) at (3.5,0);

  \fill[fill={rgb,255:red,255;green,204;blue,204}] (p3) -- (tr) -- (br) -- (p4) to[out=90,in=0,out looseness=1.25,in looseness=0.5] (epsilon) to[out=180,in=90] (v2) to[out=-135,in=0,looseness=1.25] (eta) to[out=180,in=-90,out looseness=0.5,in looseness=1.25] (p3) -- cycle;
  \fill[fill={rgb,255:red,238;green,238;blue,187}] (p4) -- (p1) -- (v) to[out=45,in=-45,in looseness=1.25] (v2) to[out=90,in=180] (epsilon) to[out=0,in=90,out looseness=0.5,in looseness=1.25] (p4) -- cycle;
  \fill[fill={rgb,255:red,255;green,204;blue,204}] (tl) -- (p2) to[out=-90,in=135,out looseness=1.25,in looseness=0.5] (v) -- (p1) -- (bl) -- cycle;
  \fill[fill={rgb,255:red,204;green,221;blue,170}] (p2) -- (p3) to[out=-90,in=180,out looseness=1.25,in looseness=0.5] (eta) to[out=0,in=-135,looseness=1.25] (v2) to[out=-45,in=45,out looseness=1.25] (v) to[out=135,in=-90,out looseness=0.5,in looseness=1.25] (p2) -- cycle;

  \draw[line width=0.8pt] (v) -- (p1);
  \draw[line width=0.8pt] (v) to[out=135,in=-90,out looseness=0.5,in looseness=1.25] (p2);
  \draw[line width=0.8pt] (v2) to[out=-135,in=0,looseness=1.25] (eta) to[out=180,in=-90,out looseness=0.5,in looseness=1.25] (p3);
  \draw[line width=0.8pt] (v) to[out=45,in=-45,in looseness=1.25] (v2);
  \draw[line width=0.8pt] (v2) to[out=90,in=180] (epsilon) to[out=0,in=90,out looseness=0.5,in looseness=1.25] (p4);

  \fill[black] (v) circle (1.56pt);
  \fill[black] (eta) circle (1.56pt);
  \fill[black] (v2) circle (1.56pt);
  \fill[black] (epsilon) circle (1.56pt);

  \node[inner sep=0pt] at (1.2,0.2) {$L_X$};
  \node[inner sep=0pt] at (0.3,2.5) {$L_Y$};
  \node[inner sep=0pt] at (1.3,3.5) {$R_Y$};
  \node[inner sep=0pt] at (2.3,1.0) {$F$};
  \node[inner sep=0pt] at (3.2,0.2) {$R_X$};
  \node[inner sep=0pt,shift={(-0.0,-0.2)}] at (eta) {$\eta$};
  \node[inner sep=0pt,shift={(0.0,0.2)}] at (epsilon) {$\epsilon$};
\end{tikzpicture}
=
\begin{tikzpicture}[x=1cm,y=1cm,baseline=(current bounding box.center)]
  \path
    coordinate (bl) at (0,0) coordinate (br) at (4,0) coordinate (tr) at (4,4)
    coordinate (tl) at (0,4) coordinate (v) at (1,1.5) coordinate (v2) at (3,1.5)
    coordinate (epsilon) at (2,2.5) coordinate (p1) at (1,0) coordinate (p2) at (3,0)
    coordinate (p3) at (0.5,4) coordinate (p4) at (3.5,4);

  \fill[fill={rgb,255:red,204;green,221;blue,170}] (p3) -- (p4) to[out=-90,in=45,out looseness=1.25,in looseness=0.75] (v2) to[out=135,in=0,out looseness=1.25,in looseness=0.75] (epsilon) to[out=180,in=45,out looseness=0.75,in looseness=1.25] (v) to[out=135,in=-90,out looseness=0.75,in looseness=1.25] (p3) -- cycle;
  \fill[fill={rgb,255:red,238;green,238;blue,187}] (p2) -- (p1) -- (v) to[out=45,in=180,out looseness=1.25,in looseness=0.75] (epsilon) to[out=0,in=135,out looseness=0.75,in looseness=1.25] (v2) -- cycle;
  \fill[fill={rgb,255:red,255;green,204;blue,204}] (tl) -- (p3) to[out=-90,in=135,out looseness=1.25,in looseness=0.75] (v) -- (p1) -- (bl) -- cycle;
  \fill[fill={rgb,255:red,255;green,204;blue,204}] (p4) -- (tr) -- (br) -- (p2) -- (v2) to[out=45,in=-90,out looseness=0.75,in looseness=1.25] (p4) -- cycle;

  \draw[line width=0.8pt] (v) -- (p1);
  \draw[line width=0.8pt] (v2) -- (p2);
  \draw[line width=0.8pt] (v) to[out=45,in=180,out looseness=1.25,in looseness=0.75] (epsilon) to[out=0,in=135,out looseness=0.75,in looseness=1.25] (v2);
  \draw[line width=0.8pt] (p3) to[out=-90,in=135,out looseness=1.25,in looseness=0.75] (v);
  \draw[line width=0.8pt] (p4) to[out=-90,in=45,out looseness=1.25,in looseness=0.75] (v2);

  \fill[black] (v) circle (1.56pt);
  \fill[black] (v2) circle (1.56pt);
  \fill[black] (epsilon) circle (1.56pt);

  \node[inner sep=0pt] at (0.7,0.7) {$L_X$};
  \node[inner sep=0pt] at (3.3,0.7) {$R_X$};
  \node[inner sep=0pt] at (1.3,2.2) {$F$};
  \node[inner sep=0pt] at (2.7,2.2) {$G$};
  \node[inner sep=0pt] at (0.8,3.5) {$L_Y$};
  \node[inner sep=0pt] at (3.2,3.5) {$R_Y$};
  \node[inner sep=0pt,shift={(0.011,0.246)}] at (epsilon) {$\epsilon$};
\end{tikzpicture}
\end{equation*}
Then, as described above, we conjugate this, composing from the bottom to the top, and slide some strings around:
\begin{equation*}
\begin{tikzpicture}[x=1cm,y=1cm,baseline=(current bounding box.center)]
  \path
    coordinate (bl) at (0,0) coordinate (br) at (4,0) coordinate (tr) at (4,4)
    coordinate (tl) at (0,4) coordinate (v) at (1,1.5) coordinate (v2) at (3,1.5)
    coordinate (epsilon) at (2,2.5) coordinate (epsilon2) at (2,3) coordinate (epsilon3) at (2,1.5)
    coordinate (v3) at (1.5,0.5) coordinate (v4) at (2.5,0.5) coordinate (v5) at (1,3.5)
    coordinate (v6) at (3,3.5) coordinate (p1) at (1.5,0) coordinate (p2) at (2.5,0)
    coordinate (p3) at (1,4) coordinate (p4) at (3,4);

  \fill[fill={rgb,255:red,255;green,204;blue,204}] (p4) -- (tr) -- (br) -- (p2) -- (v4) to[out=45,in=-90,out looseness=1.5,in looseness=1.25] (v2) to[out=45,in=-45] (v6) -- cycle;
  \fill[fill={rgb,255:red,255;green,204;blue,204}] (tl) -- (p3) -- (v5) to[out=-135,in=135] (v) to[out=-90,in=135,out looseness=1.25,in looseness=1.5] (v3) -- (p1) -- (bl) -- cycle;
  \fill[fill={rgb,255:red,204;green,221;blue,170}] (epsilon) to[out=180,in=45,out looseness=0.75,in looseness=1.25] (v) to[out=135,in=-135] (v5) to[out=-45,in=180,out looseness=0.5,in looseness=1.25] (epsilon2) to[out=0,in=-135,out looseness=1.25,in looseness=0.5] (v6) to[out=-45,in=45] (v2) to[out=135,in=0,out looseness=1.25,in looseness=0.75] (epsilon) -- cycle;
  \fill[fill={rgb,255:red,238;green,238;blue,187}] (v) to[out=45,in=180,out looseness=1.25,in looseness=0.75] (epsilon) to[out=0,in=135,out looseness=0.75,in looseness=1.25] (v2) to[out=-90,in=45,out looseness=1.25,in looseness=1.5] (v4) to[out=135,in=0,out looseness=1.25] (epsilon3) to[out=180,in=45,in looseness=1.25] (v3) to[out=135,in=-90,out looseness=1.5,in looseness=1.25] (v) -- cycle;
  \fill[fill={rgb,255:red,221;green,221;blue,221}] (p3) -- (p4) -- (v6) to[out=-135,in=0,out looseness=0.5,in looseness=1.25] (epsilon2) to[out=180,in=-45,out looseness=1.25,in looseness=0.5] (v5) -- cycle;
  \fill[fill={rgb,255:red,204;green,238;blue,255}] (p2) -- (p1) -- (v3) to[out=45,in=180,out looseness=1.25] (epsilon3) to[out=0,in=135,in looseness=1.25] (v4) -- cycle;

  \draw[line width=0.8pt] (v) to[out=45,in=180,out looseness=1.25,in looseness=0.75] (epsilon) to[out=0,in=135,out looseness=0.75,in looseness=1.25] (v2);
  \draw[line width=0.8pt] (v3) -- (p1);
  \draw[line width=0.8pt] (v4) -- (p2);
  \draw[line width=0.8pt] (v3) to[out=135,in=-90,out looseness=1.5,in looseness=1.25] (v);
  \draw[line width=0.8pt] (v4) to[out=45,in=-90,out looseness=1.5,in looseness=1.25] (v2);
  \draw[line width=0.8pt] (v4) to[out=135,in=0,out looseness=1.25] (epsilon3) to[out=180,in=45,in looseness=1.25] (v3);
  \draw[line width=0.8pt] (v6) to[out=-135,in=0,out looseness=0.5,in looseness=1.25] (epsilon2) to[out=180,in=-45,out looseness=1.25,in looseness=0.5] (v5);
  \draw[line width=0.8pt] (v5) -- (p3);
  \draw[line width=0.8pt] (v6) -- (p4);
  \draw[line width=0.8pt] (v) to[out=135,in=-135] (v5);
  \draw[line width=0.8pt] (v2) to[out=45,in=-45] (v6);

  \fill[black] (v) circle (1.56pt);
  \fill[black] (v2) circle (1.56pt);
  \fill[black] (epsilon) circle (1.56pt);
  \fill[black] (epsilon2) circle (1.56pt);
  \fill[black] (epsilon3) circle (1.56pt);
  \fill[black] (v3) circle (1.56pt);
  \fill[black] (v4) circle (1.56pt);
  \fill[black] (v5) circle (1.56pt);
  \fill[black] (v6) circle (1.56pt);

  \node[font=\fontsize{8}{10}\selectfont,inner sep=0pt] at (1.3,2.3) {$F$};
  \node[font=\fontsize{8}{10}\selectfont,inner sep=0pt] at (2.7,2.3) {$G$};
  \node[font=\fontsize{8}{10}\selectfont,inner sep=0pt] at (1.2,0.2) {$L'_X$};
  \node[font=\fontsize{8}{10}\selectfont,inner sep=0pt] at (2.8,0.2) {$R'_X$};
  \node[font=\fontsize{8}{10}\selectfont,inner sep=0pt] at (0.7,1.2) {$L_X$};
  \node[font=\fontsize{8}{10}\selectfont,inner sep=0pt] at (3.3,1.2) {$R_X$};
  \node[font=\fontsize{8}{10}\selectfont,inner sep=0pt] at (1.4,1.2) {$A_X$};
  \node[font=\fontsize{8}{10}\selectfont,inner sep=0pt] at (2.6,1.2) {$B_X$};
  \node[font=\fontsize{8}{10}\selectfont,inner sep=0pt] at (1.3,2.9) {$A_Y$};
  \node[font=\fontsize{8}{10}\selectfont,inner sep=0pt] at (2.7,2.9) {$B_Y$};
  \node[font=\fontsize{8}{10}\selectfont,inner sep=0pt] at (0.406,3.027) {$L_Y$};
  \node[font=\fontsize{8}{10}\selectfont,inner sep=0pt] at (3.659,2.943) {$R_Y$};
  \node[font=\fontsize{8}{10}\selectfont,inner sep=0pt,shift={(0.0,-0.2)}] at (epsilon) {$\epsilon$};
  \node[font=\fontsize{8}{10}\selectfont,inner sep=0pt,shift={(0.0,0.2)}] at (epsilon2) {$\epsilon^{-1}$};
  \node[font=\fontsize{8}{10}\selectfont,inner sep=0pt,shift={(0.0,0.2)}] at (epsilon3) {$\epsilon$};

\end{tikzpicture}
=
\begin{tikzpicture}[x=1cm,y=1cm,baseline=(current bounding box.center)]
  \path
    coordinate (bl) at (0,0) coordinate (br) at (4,0) coordinate (tr) at (4,4)
    coordinate (tl) at (0,4) coordinate (epsilon) at (2,2.5) coordinate (epsilon2) at (2,3)
    coordinate (epsilon3) at (2,1.5) coordinate (v) at (3.5,0.5) coordinate (v2) at (1,3.5)
    coordinate (v3) at (3,3.5) coordinate (v4) at (0.5,0.5) coordinate (v5) at (1,1)
    coordinate (v6) at (3,1) coordinate (p1) at (0.5,0) coordinate (p2) at (3.5,0)
    coordinate (p3) at (1,4) coordinate (p4) at (3,4);

  \fill[fill={rgb,255:red,204;green,221;blue,170}] (v2) to[out=-45,in=180,out looseness=0.5,in looseness=1.25] (epsilon2) to[out=0,in=-135,out looseness=1.25,in looseness=0.5] (v3) to[out=-45,in=45,looseness=0.75] (v) to[out=135,in=-90] (v6) to[out=45,in=0,in looseness=1.25] (epsilon) to[out=180,in=135,out looseness=1.25] (v5) to[out=-90,in=45] (v4) to[out=135,in=-135,looseness=0.75] (v2) -- cycle;
  \fill[fill={rgb,255:red,204;green,238;blue,255}] (p2) -- (p1) -- (v4) to[out=45,in=-90] (v5) to[out=45,in=180,looseness=1.25] (epsilon3) to[out=0,in=135,out looseness=1.25] (v6) to[out=-90,in=135] (v) -- cycle;
  \fill[fill={rgb,255:red,238;green,238;blue,187}] (epsilon3) to[out=180,in=45,looseness=1.25] (v5) to[out=135,in=180,in looseness=1.25] (epsilon) to[out=0,in=45,out looseness=1.25] (v6) to[out=135,in=0,in looseness=1.25] (epsilon3) -- cycle;
  \fill[fill={rgb,255:red,255;green,204;blue,204}] (p4) -- (tr) -- (br) -- (p2) -- (v) to[out=45,in=-45,looseness=0.75] (v3) -- cycle;
  \fill[fill={rgb,255:red,255;green,204;blue,204}] (tl) -- (p3) -- (v2) to[out=-135,in=135,looseness=0.75] (v4) -- (p1) -- (bl) -- cycle;
  \fill[fill={rgb,255:red,221;green,221;blue,221}] (p3) -- (p4) -- (v3) to[out=-135,in=0,out looseness=0.5,in looseness=1.25] (epsilon2) to[out=180,in=-45,out looseness=1.25,in looseness=0.5] (v2) -- cycle;

  \draw[line width=0.8pt] (v4) -- (p1);
  \draw[line width=0.8pt] (v) -- (p2);
  \draw[line width=0.8pt] (v3) to[out=-135,in=0,out looseness=0.5,in looseness=1.25] (epsilon2) to[out=180,in=-45,out looseness=1.25,in looseness=0.5] (v2);
  \draw[line width=0.8pt] (v2) -- (p3);
  \draw[line width=0.8pt] (v3) -- (p4);
  \draw[line width=0.8pt] (v2) to[out=-135,in=135,looseness=0.75] (v4);
  \draw[line width=0.8pt] (v3) to[out=-45,in=45,looseness=0.75] (v);
  \draw[line width=0.8pt] (v4) to[out=45,in=-90] (v5);
  \draw[line width=0.8pt] (v) to[out=135,in=-90] (v6);
  \draw[line width=0.8pt] (v5) to[out=45,in=180,looseness=1.25] (epsilon3) to[out=0,in=135,out looseness=1.25] (v6);
  \draw[line width=0.8pt] (v5) to[out=135,in=180,in looseness=1.25] (epsilon) to[out=0,in=45,out looseness=1.25] (v6);

  \fill[black] (epsilon) circle (1.56pt);
  \fill[black] (epsilon2) circle (1.56pt);
  \fill[black] (epsilon3) circle (1.56pt);
  \fill[black] (v) circle (1.56pt);
  \fill[black] (v2) circle (1.56pt);
  \fill[black] (v3) circle (1.56pt);
  \fill[black] (v4) circle (1.56pt);
  \fill[black] (v5) circle (1.56pt);
  \fill[black] (v6) circle (1.56pt);

  \node[font=\fontsize{8}{10}\selectfont,inner sep=0pt] at (0.8,0.2) {$L'_X$};
  \node[font=\fontsize{8}{10}\selectfont,inner sep=0pt] at (3.2,0.2) {$R'_X$};
  \node[font=\fontsize{8}{10}\selectfont,inner sep=0pt] at (1.4,3) {$A_Y$};
  \node[font=\fontsize{8}{10}\selectfont,inner sep=0pt] at (2.8,3) {$B_Y$};
  \node[font=\fontsize{8}{10}\selectfont,inner sep=0pt] at (0.3,2.9) {$L_Y$};
  \node[font=\fontsize{8}{10}\selectfont,inner sep=0pt] at (3.7,2.9) {$R_Y$};
  \node[font=\fontsize{8}{10}\selectfont,inner sep=0pt] at (1.4,1.6) {$A_X$};
  \node[font=\fontsize{8}{10}\selectfont,inner sep=0pt] at (2.6,1.6) {$B_X$};
  \node[font=\fontsize{8}{10}\selectfont,inner sep=0pt] at (0.9,2.1) {$F$};
  \node[font=\fontsize{8}{10}\selectfont,inner sep=0pt] at (3.1,2.1) {$G$};
  \node[font=\fontsize{8}{10}\selectfont,inner sep=0pt,shift={(0.0,-0.2)}] at (epsilon) {$\epsilon$};
  \node[font=\fontsize{8}{10}\selectfont,inner sep=0pt,shift={(0.0,0.2)}] at (epsilon2) {$\epsilon^{-1}$};
  \node[font=\fontsize{8}{10}\selectfont,inner sep=0pt,shift={(0.0,-0.2)}] at (epsilon3) {$\epsilon$};
\end{tikzpicture}
=
\begin{tikzpicture}[x=1cm,y=1cm,baseline=(current bounding box.center)]
  \path
    coordinate (bl) at (0,0) coordinate (br) at (4,0) coordinate (tr) at (4,4)
    coordinate (tl) at (0,4) coordinate (epsilon) at (2,2.5) coordinate (epsilon2) at (2,3)
    coordinate (epsilon3) at (2,1.5) coordinate (v) at (3.5,0.5) coordinate (v2) at (1,3.5)
    coordinate (v3) at (3,3.5) coordinate (v4) at (0.5,0.5) coordinate (v5) at (1,1)
    coordinate (v6) at (3,1) coordinate (p1) at (0.5,0) coordinate (p2) at (3.5,0)
    coordinate (p3) at (1,4) coordinate (p4) at (3,4);

  \fill[fill={rgb,255:red,204;green,238;blue,255}] (p2) -- (p1) -- (v4) to[out=45,in=-90] (v5) to[out=45,in=180,looseness=1.25] (epsilon3) to[out=0,in=135,out looseness=1.25] (v6) to[out=-90,in=135] (v) -- cycle;
  \fill[fill={rgb,255:red,204;green,221;blue,170}] (v2) to[out=-45,in=180,out looseness=0.5,in looseness=1.25] (epsilon2) -- (epsilon) to[out=180,in=135,out looseness=1.25] (v5) to[out=-90,in=45] (v4) to[out=135,in=-135,looseness=0.75] (v2) -- cycle;
  \fill[fill={rgb,255:red,204;green,221;blue,170}] (epsilon2) to[out=0,in=-135,out looseness=1.25,in looseness=0.5] (v3) to[out=-45,in=45,looseness=0.75] (v) to[out=135,in=-90] (v6) to[out=45,in=0,in looseness=1.25] (epsilon) -- cycle;
  \fill[fill={rgb,255:red,221;green,221;blue,221}] (epsilon3) to[out=180,in=45,looseness=1.25] (v5) to[out=135,in=180,in looseness=1.25] (epsilon) to[out=0,in=45,out looseness=1.25] (v6) to[out=135,in=0,in looseness=1.25] (epsilon3) -- cycle;
  \fill[fill={rgb,255:red,255;green,204;blue,204}] (p4) -- (tr) -- (br) -- (p2) -- (v) to[out=45,in=-45,looseness=0.75] (v3) -- cycle;
  \fill[fill={rgb,255:red,255;green,204;blue,204}] (tl) -- (p3) -- (v2) to[out=-135,in=135,looseness=0.75] (v4) -- (p1) -- (bl) -- cycle;
  \fill[fill={rgb,255:red,221;green,221;blue,221}] (p3) -- (p4) -- (v3) to[out=-135,in=0,out looseness=0.5,in looseness=1.25] (epsilon2) to[out=180,in=-45,out looseness=1.25,in looseness=0.5] (v2) -- cycle;

  \draw[line width=0.8pt] (v4) -- (p1);
  \draw[line width=0.8pt] (v) -- (p2);
  \draw[line width=0.8pt] (epsilon2) to[out=180,in=-45,out looseness=1.25,in looseness=0.5] (v2);
  \draw[line width=0.8pt] (epsilon2) to[out=0,in=-135,out looseness=1.25,in looseness=0.5] (v3);
  \draw[line width=0.8pt] (v2) -- (p3);
  \draw[line width=0.8pt] (v3) -- (p4);
  \draw[line width=0.8pt] (v2) to[out=-135,in=135,looseness=0.75] (v4);
  \draw[line width=0.8pt] (v3) to[out=-45,in=45,looseness=0.75] (v);
  \draw[line width=0.8pt] (v4) to[out=45,in=-90] (v5);
  \draw[line width=0.8pt] (v) to[out=135,in=-90] (v6);
  \draw[line width=0.8pt] (v5) to[out=45,in=180,looseness=1.25] (epsilon3) to[out=0,in=135,out looseness=1.25] (v6);
  \draw[line width=0.8pt] (v5) to[out=135,in=180,in looseness=1.25] (epsilon);
  \draw[line width=0.8pt] (epsilon) to[out=0,in=45,out looseness=1.25] (v6);
  \draw[line width=0.8pt,dotted] (epsilon2) -- (epsilon);

  \fill[black] (epsilon) circle (1.56pt);
  \fill[black] (epsilon2) circle (1.56pt);
  \fill[black] (epsilon3) circle (1.56pt);
  \fill[black] (v) circle (1.56pt);
  \fill[black] (v2) circle (1.56pt);
  \fill[black] (v3) circle (1.56pt);
  \fill[black] (v4) circle (1.56pt);
  \fill[black] (v5) circle (1.56pt);
  \fill[black] (v6) circle (1.56pt);

  \node[font=\fontsize{8}{10}\selectfont,inner sep=0pt] at (0.8,0.2) {$L'_X$};
  \node[font=\fontsize{8}{10}\selectfont,inner sep=0pt] at (3.2,0.2) {$R'_X$};
  \node[font=\fontsize{8}{10}\selectfont,inner sep=0pt] at (1.4,3.0) {$A_Y$};
  \node[font=\fontsize{8}{10}\selectfont,inner sep=0pt] at (2.8,3.0) {$B_Y$};
  \node[font=\fontsize{8}{10}\selectfont,inner sep=0pt] at (0.3,2.9) {$L_Y$};
  \node[font=\fontsize{8}{10}\selectfont,inner sep=0pt] at (3.7,2.9) {$R_Y$};
  \node[font=\fontsize{8}{10}\selectfont,inner sep=0pt] at (1.4,1.6) {$F'$};
  \node[font=\fontsize{8}{10}\selectfont,inner sep=0pt] at (2.6,1.6) {$G'$};
  \node[font=\fontsize{8}{10}\selectfont,inner sep=0pt] at (0.8,2.2) {$A_Y$};
  \node[font=\fontsize{8}{10}\selectfont,inner sep=0pt] at (3.2,2.2) {$B_Y$};
  \node[font=\fontsize{8}{10}\selectfont,inner sep=0pt,shift={(0.0,-0.2)}] at (epsilon) {$\epsilon$};
  \node[font=\fontsize{8}{10}\selectfont,inner sep=0pt,shift={(0.0,0.2)}] at (epsilon2) {$\epsilon^{-1}$};
  \node[font=\fontsize{8}{10}\selectfont,inner sep=0pt,shift={(0.0,-0.2)}] at (epsilon3) {$\epsilon$};
\end{tikzpicture}
\end{equation*}
In the second equality above we used that \(F\circ A_X=A_Y\circ F'\), and that the counit of a composite is the composite of counits. We can now cancel the \(A_Y\dashv B_Y\) counit against its inverse to arrive at the desired conclusion:
\begin{equation*}
\begin{tikzpicture}[x=1cm,y=1cm,baseline=(current bounding box.center)]
  \path
    coordinate (bl) at (0,0) coordinate (br) at (4,0) coordinate (tr) at (4,4)
    coordinate (tl) at (0,4) coordinate (epsilon) at (2,1.5) coordinate (v) at (3.5,0.5)
    coordinate (v2) at (1,3.5) coordinate (v3) at (3,3.5) coordinate (v4) at (0.5,0.5)
    coordinate (v5) at (1,1) coordinate (v6) at (3,1) coordinate (p1) at (0.5,0)
    coordinate (p2) at (3.5,0) coordinate (p3) at (1,4) coordinate (p4) at (3,4);

  \fill[fill={rgb,255:red,221;green,221;blue,221}] (p3) -- (p4) -- (v3) -- (v6) to[out=135,in=0,in looseness=1.25] (epsilon) to[out=180,in=45,looseness=1.25] (v5) -- (v2) -- cycle;
  \fill[fill={rgb,255:red,204;green,238;blue,255}] (p2) -- (p1) -- (v4) to[out=45,in=-90] (v5) to[out=45,in=180,looseness=1.25] (epsilon) to[out=0,in=135,out looseness=1.25] (v6) to[out=-90,in=135] (v) -- cycle;
  \fill[fill={rgb,255:red,255;green,204;blue,204}] (tl) -- (p3) -- (v2) to[out=-135,in=135,looseness=0.75] (v4) -- (p1) -- (bl) -- cycle;
  \fill[fill={rgb,255:red,255;green,204;blue,204}] (p4) -- (tr) -- (br) -- (p2) -- (v) to[out=45,in=-45,looseness=0.75] (v3) -- cycle;
  \fill[fill={rgb,255:red,204;green,221;blue,170}] (v3) to[out=-45,in=45,looseness=0.75] (v) to[out=135,in=-90] (v6) -- cycle;
  \fill[fill={rgb,255:red,204;green,221;blue,170}] (v4) to[out=135,in=-135,looseness=0.75] (v2) -- (v5) to[out=-90,in=45] (v4) -- cycle;

  \draw[line width=0.8pt] (v4) -- (p1);
  \draw[line width=0.8pt] (v) -- (p2);
  \draw[line width=0.8pt] (v2) -- (p3);
  \draw[line width=0.8pt] (v3) -- (p4);
  \draw[line width=0.8pt] (v2) to[out=-135,in=135,looseness=0.75] (v4);
  \draw[line width=0.8pt] (v3) to[out=-45,in=45,looseness=0.75] (v);
  \draw[line width=0.8pt] (v4) to[out=45,in=-90] (v5);
  \draw[line width=0.8pt] (v) to[out=135,in=-90] (v6);
  \draw[line width=0.8pt] (v5) to[out=45,in=180,looseness=1.25] (epsilon) to[out=0,in=135,out looseness=1.25] (v6);
  \draw[line width=0.8pt] (v5) -- (v2);
  \draw[line width=0.8pt] (v6) -- (v3);

  \fill[black] (epsilon) circle (1.56pt);
  \fill[black] (v) circle (1.56pt);
  \fill[black] (v2) circle (1.56pt);
  \fill[black] (v3) circle (1.56pt);
  \fill[black] (v4) circle (1.56pt);
  \fill[black] (v5) circle (1.56pt);
  \fill[black] (v6) circle (1.56pt);
  \fill[black] (v5) circle (1.56pt);
  \fill[black] (v2) circle (1.56pt);

  \node[font=\fontsize{8}{10}\selectfont,inner sep=0pt] at (0.8,0.2) {$L'_X$};
  \node[font=\fontsize{8}{10}\selectfont,inner sep=0pt] at (3.2,0.2) {$R'_X$};
  \node[font=\fontsize{8}{10}\selectfont,inner sep=0pt] at (0.3,2.9) {$L_Y$};
  \node[font=\fontsize{8}{10}\selectfont,inner sep=0pt] at (3.7,2.9) {$R_Y$};
  \node[font=\fontsize{8}{10}\selectfont,inner sep=0pt] at (1.5,1.6) {$F'$};
  \node[font=\fontsize{8}{10}\selectfont,inner sep=0pt] at (2.5,1.6) {$G'$};
  \node[font=\fontsize{8}{10}\selectfont,inner sep=0pt] at (1.3,2.6) {$A_Y$};
  \node[font=\fontsize{8}{10}\selectfont,inner sep=0pt] at (2.7,2.6) {$B_Y$};
  \node[font=\fontsize{8}{10}\selectfont,inner sep=0pt,shift={(0.0,-0.2)}] at (epsilon) {$\epsilon$};

\end{tikzpicture}
=
\begin{tikzpicture}[x=1cm,y=1cm,baseline=(current bounding box.center)]
  \path
    coordinate (bl) at (0,0) coordinate (br) at (4,0) coordinate (tr) at (4,4)
    coordinate (tl) at (0,4) coordinate (v) at (1,1.5) coordinate (v2) at (3,1.5)
    coordinate (epsilon) at (2,2.5) coordinate (p1) at (1,0) coordinate (p2) at (3,0)
    coordinate (p3) at (0.5,4) coordinate (p4) at (3.5,4);

  \fill[fill={rgb,255:red,221;green,221;blue,221}] (p3) -- (p4) to[out=-90,in=45,out looseness=1.25,in looseness=0.75] (v2) to[out=135,in=0,out looseness=1.25,in looseness=0.75] (epsilon) to[out=180,in=45,out looseness=0.75,in looseness=1.25] (v) to[out=135,in=-90,out looseness=0.75,in looseness=1.25] (p3) -- cycle;
  \fill[fill={rgb,255:red,204;green,238;blue,255}] (p2) -- (p1) -- (v) to[out=45,in=180,out looseness=1.25,in looseness=0.75] (epsilon) to[out=0,in=135,out looseness=0.75,in looseness=1.25] (v2) -- cycle;
  \fill[fill={rgb,255:red,255;green,204;blue,204}] (tl) -- (p3) to[out=-90,in=135,out looseness=1.25,in looseness=0.75] (v) -- (p1) -- (bl) -- cycle;
  \fill[fill={rgb,255:red,255;green,204;blue,204}] (p4) -- (tr) -- (br) -- (p2) -- (v2) to[out=45,in=-90,out looseness=0.75,in looseness=1.25] (p4) -- cycle;

  \draw[line width=0.8pt] (v) -- (p1);
  \draw[line width=0.8pt] (v2) -- (p2);
  \draw[line width=0.8pt] (v) to[out=45,in=180,out looseness=1.25,in looseness=0.75] (epsilon) to[out=0,in=135,out looseness=0.75,in looseness=1.25] (v2);
  \draw[line width=0.8pt] (p3) to[out=-90,in=135,out looseness=1.25,in looseness=0.75] (v);
  \draw[line width=0.8pt] (p4) to[out=-90,in=45,out looseness=1.25,in looseness=0.75] (v2);

  \fill[black] (v) circle (1.56pt);
  \fill[black] (v2) circle (1.56pt);
  \fill[black] (epsilon) circle (1.56pt);

  \node[inner sep=0pt] at (0.7,0.4) {$L_X'$};
  \node[inner sep=0pt] at (3.3,0.4) {$R_X'$};
  \node[inner sep=0pt] at (1.2,2.3) {$F'$};
  \node[inner sep=0pt] at (2.8,2.3) {$G'$};
  \node[inner sep=0pt] at (0.8,3.6) {$L_Y'$};
  \node[inner sep=0pt] at (3.2,3.6) {$R_Y'$};
  \node[inner sep=0pt,shift={(0.0,0.2)}] at (epsilon) {$\epsilon$};
\end{tikzpicture}
\end{equation*}
\end{proof}

We will now prove \Cref{lem:inductive_step_general}. Although the proof is lengthy, at its core, it follows the same pattern as our overview of the $U(1)$-case from the introduction: we want to match isotropy separation sequences for $H \leq G$ with an effective Cartier divisor $\sfE_H\to \sfE_G$ as recollemonts.

\begin{proof}[Proof of \Cref{lem:inductive_step_general}]
    Let us make the abbreviations $\sfE_G=\sfX$, $\sfE_H=\sfZ$, $\sfE_N = \sfS$, and \(\Spec A=\sfM\) together with the map of schemes $f\colon \sfX \to \sfS$. The map $\phi$ induces an affine map of stacks $\phi \colon \sfZ \to \sfX$ by \Cref{thm:Ell_geo_nice_prop}(\ref{it:Ell_closed_embedding_affine}). Moreover, by \Cref{lem:even_periodic_underlying_ECD}, whose hypotheses hold by assumption, the induced map of nonconnective spectral Deligne--Mumford stacks $\sfZ \to \X$ is an effective Cartier divisor. Let us denote the associated open complement as \(i \colon \sfU = \sfX\setminus \sfZ \to \sfX\). As all of the stacks \(\sfX\), \(\sfZ\), \(\sfS\), and \(\sfM\) are, by assumption, in the essential image of the fully faithful functor
    \[h_{(-)}:\SpSch^\nc\to\Stk,\]
    we will not distinguish between these two categories.
    
    Let $\widetilde{E}\langle H\rangle$ be the unique $G$-spectrum with geometric fixed points 
    \[
    \Phi^K (\widetilde{E}\langle H\rangle) \simeq \begin{cases} 0 & \text{if }K\text{ is subconjugate to }H,\\ \Sph & \text{otherwise.}\end{cases}
    \] 
    As a first step, we want to show that
    \[\O_{\sfX} \otimes \widetilde{E}\langle H\rangle\simeq \O_{\sfX}(\infty \sfZ) \in \QCoh(\sfX)\]
    as \(\E_0\)-algebras, using the notation of \Cref{prop:nc_ECD_flat_over_base}. For this, it helps to have an explicit model of $\widetilde{E}\langle H\rangle$. Consider the \(G\)-space \(G/H\). For \(K\subset G\) a closed subgroup, we have that \((G/H)^K\) is nonempty exactly when \(K\) is subconjugate to \(H\).
    Therefore, the cofibre of pointed $G$-spaces
    \[(G/H)_+ \to *_+ = S^0 \to C\]
    has \(C^K=S^0\) if \(K\) is not subconjugate to \(H\), and it is connected otherwise.
    We find that the geometric fixed points of the \(G\)-spectrum
    \[D=\colim \left(\Sph_G^0 = \Sigma_G^\infty \ast_+ \to\Sigma^\infty_GC\to (\Sigma^\infty_GC)^{\otimes 2}\to \cdots\right)\]
    agree with the geometric fixed points of \(\widetilde{E}\langle H\rangle\), hence \(D \simeq \widetilde{E}\langle H\rangle\).
    By construction, applying \(\O_{\sfE_G} \otimes -\) to the suspension spectra of \(S^0\to C\) is given by first dualising in $G$-spectra and then applying \(\widetilde{{\ELL}}_G\); recall this reduced variant from \Cref{rmk:reduced_cohomology}. Equivalently, this is the dual of the fibre of
    \[\ELL_G(S^0)\to \ELL_G(G/H_+).\]
    By construction of equivariant elliptic cohomology in \Cref{thm:functor/naturality_GM_theories}, this is the map $\O_{\sfX} \to \phi_* \O_{\sfZ}$ in $\QCoh(\sfX)$ that identifies $\O_{\sfX} \otimes \Sigma^\infty_G C\simeq \O_{\sfX}(\sfZ)$, as well as the $\E_0$-algebras
    \[\O_\X \otimes \widetilde{E}\langle H\rangle\simeq \O_\sfX \otimes D \simeq \O_\sfX(\infty \sfZ) \in \QCoh(\sfX),\]
    giving our first step.

    Next, one combines \Cref{lem:even_periodic_underlying_ECD}, whose hypotheses hold by assumption, with \Cref{thm:Ell_geo_nice_prop}(\ref{it:Ell_is_flat_and_good}) allows us to apply \Cref{prop:nc_ECD_flat_over_base}, which then shows that \(i:\sfU\to\sfX\) is a quasi-compact open immersion, \(i_*\O_\sfU\simeq \O_\sfX(\infty \sfZ)\) as \(\E_0\)-algebras, and the kernel of \(\phi^*\) agrees with the essential image of \(i_*\). We now consider the commutative diagram
\begin{equation}\label{eq:general_induction_diagram1}
\begin{tikzcd}
	{\Mod_{A_G}(\Sp_G)} && {\QCoh(\sfX)} \\
	{\Mod_{A_H}(\Sp_H)} && {\QCoh(\sfZ)}
	\arrow["{\O_{\X} \otimes_{\Ga_G \sfE} - }", from=1-1, to=1-3]
	\arrow["{\mathrm{Res}^G_H}"', from=1-1, to=2-1]
	\arrow["{\phi^*}", from=1-3, to=2-3]
	\arrow["{\O_{\sfZ} \otimes_{\Ga_H \sfE} - }"', from=2-1, to=2-3]
\end{tikzcd}
\end{equation}
    in \(\CAlg(\PrLst)\) from \Cref{pr:homomorphism_and_linear_homology}. As a module \(M\in\Mod_{A_G}(\Sp_G)\) is \(\langle H\rangle\)-local, meaning local with respect to the $G$-spectrum $\widetilde{E}\langle H \rangle$, if and only if its restriction to \(\Mod_{A_H}(\Sp_H)\) vanishes, and since we have the equivalence \(\O_{\X} \otimes \widetilde{E}\langle H\rangle\simeq i_*\O_\sfU\) as established above, we can extend this to a larger commutative diagram
\begin{equation}\label{eq:general_induction_diagram2}
\begin{tikzcd}
	{\Mod_{A_G}(\Sp_G)^{\langle H\rangle-\loc}} & {\QCoh(\sfU)} \\
	{\Mod_{A_G}(\Sp_G)} & {\QCoh(\sfX)} \\
	{\Mod_{A_G}(\Sp_G)^{\langle H\rangle-\cpl}} & {\QCoh(\sfX)/\QCoh(\sfU)} \\
	{\Mod_{A_H}(\Sp_H)} & {\QCoh(\sfZ)}
	\arrow["{F^\loc}", from=1-1, to=1-2]
	\arrow[hook, from=1-1, to=2-1]
	\arrow["{i_*}", hook, from=1-2, to=2-2]
	\arrow["{F=\O_{\sfX} \otimes_{\Ga_G \sfE} - }", from=2-1, to=2-2]
	\arrow[two heads, from=2-1, to=3-1]
	\arrow[two heads, from=2-2, to=3-2]
	\arrow["{F^\cpl}", from=3-1, to=3-2]
	\arrow[from=3-1, to=4-1]
	\arrow[from=3-2, to=4-2]
	\arrow["{F_H=\O_{\sfZ} \otimes_{\Ga_H \sfE} - }", from=4-1, to=4-2]
\end{tikzcd}
\end{equation}
    in \(\PrLst\) which has the following features (compare \cite[Pr.2.10]{Tokic_2026_FamilyCompletionTheorem}):
    \begin{enumerate}
        \item the upper and middle vertical functors in each column form a symmetric monoidal recollement,
        \item the lower vertical functors in each column are conservative,
        \item the upper square is vertically left-adjointable, and the associated left adjoints to the upper vertical functors are symmetric monoidal localisations,
        \item the middle square is vertically left-adjointable, and
        \item the pasting of the lower two squares agrees with (\ref{eq:general_induction_diagram1}).
    \end{enumerate}

    We claim that both \(F^\loc\) and \(F^\cpl\) are essentially surjective.
    Assuming this, we see that the essential image of \(F\) contains every local and every nilpotent object of \(\QCoh(\sfX)\) with respect to the recollement associated with \(i_*\O_\sfU\).
    Since \(F\) is fully faithful by \Cref{pr:left_adjoint_is_fullyfaithful}, and every object of \(\QCoh(X)\) is the fibre of a map from a local to a nilpotent object, \(F\) must be essentially surjective, thus an equivalence.

    To prove this claim, we start with $F^\cpl$.
    The idea is to apply \Cref{lem:comonadic_helper}.
    First, we show that the lower vertical functors in (\ref{eq:general_induction_diagram2}) are comonadic: We have already seen that they are conservative, so by (the dual of) \cite[Th.4.7.0.3]{ha}, it suffices to show that they commute with limits. Since the middle vertical functors in (\ref{eq:general_induction_diagram2}) admit fully faithful right adjoints, this is equivalent to the vertical functors in (\ref{eq:general_induction_diagram1}) commuting with limits.
    It is well-known that \(\mathrm{Res}^G_H\) admits a left adjoint, while $\phi^\ast$, by \Cref{pr:pullback_has_left_adjoint_perfect_stacks}, admits a left adjoint if and only if the right adjoint to $\phi_\ast$, written as $\phi^{(1)}$, commutes with coproducts. As $\phi$ is affine, \(\phi_*\) is conservative, and the endofunctor \(\phi_*\phi^{(1)}\) is equivalent to \(\underline{\map}_X(\phi_*\O_Z,\bullet)\).
    By assumption, \(\phi_*\O_Z\) is the cofibre of a map between invertible, and thus compact, objects, so it is compact itself. Hence \(\phi_*\phi^{(1)}\) commutes with coproducts, so also \(\phi^{(1)}\) commutes with coproducts. In particular, the lower vertical functors of (\ref{eq:general_induction_diagram2}) are comonadic.

    By assumption \(F_H\) is an equivalence, so  \Cref{lem:comonadic_helper} tells us that \(F^\cpl\) is an equivalence if
    \[\mu\colon \phi^*\circ F\circ \underline{\Ga}(\sfE_G, -)\circ \phi_*\overset{\phi^*\circ\epsilon\circ\phi_*}{\Longrightarrow} \phi^*\circ \phi_*\]
    is an equivalence of enodfunctors of $\QCoh(\sfZ)$, where \(\epsilon\) is the adjunction counit of \(F\dashv \underline{\Ga}(\sfE_G, -)\). By \Cref{pr:pullback_has_left_adjoint_perfect_stacks}, \(\mu\) is a natural transformation between left adjoint functors between presentable stable categories.
    Thus, it suffices to check that \(\mu\) is an equivalence on a set of generators of \(\QCoh(\sfZ)\).
    Since \(F_{H}\) is an equivalence and \(\Mod_{A_{H}}(\Sp_{H})\) is generated by \(A_{H}\otimes \mathbb{D}(H/K_+)\), for \(K\subset H\), it suffices to show that \(\mu\) is an equivalence on objects of the form \(F_{H}(A_{H}\otimes \mathbb{D}(H/K_+))\). As \(F\) is fully faithful, \(\epsilon\) is an equivalence on all objects in the essential image of \(F\).
    Thus, it suffices to show that \(\phi_* F_{H}(A_{H}\otimes \mathbb{D}(H/K_+))\) is in that essential image for all closed subgroups \(K\subset G\). By the construction of elliptic homology (\Cref{def:elliptic_homology}) and the naturality of elliptic cohomology (\Cref{thm:naturality_of_elliptic_cohomology}(\ref{it:ell_cohomolgy_induction})), we obtain natural equivalences
\[\phi_*F_{H}(A_{H}\otimes \mathbb{D}(H/K_+)) \simeq \phi_*\ELL_{H}(H/K_+) \simeq \ELL_{G}(G\times_H(H/K_+))\]
\[\simeq \ELL_{G}(G/K_+) \simeq F(A_{G}\otimes \mathbb{D}(G/K_+)),\]
    hence \(F^\cpl\) is an equivalence.

    Now let us consider the functor \(F^\loc\), for which we finally involve $\sfE_N$. By \Cref{pr:homomorphism_and_linear_homology} and the adjointability of the top square in (\ref{eq:general_induction_diagram2}), we obtain a commutative diagram
\begin{equation}\label{eq:general_induction_diagram3}
\begin{tikzcd}
	{\Mod_{A_G}(\Sp_G)^{\langle H\rangle-\loc}} && {\QCoh(\sfU)} \\
	{\Mod_{A_G}(\Sp_G)} && {\QCoh(\X)} \\
	{\Mod_{A_N}(\Sp_N)} && {\QCoh(\sfS)}
	\arrow["F^\loc", from=1-1, to=1-3]
	\arrow[from=2-1, to=1-1]
	\arrow["{F}", from=2-1, to=2-3]
	\arrow["{i^*}"', from=2-3, to=1-3]
	\arrow[from=3-1, to=2-1]
	\arrow["{F_N=\O_{\X} \otimes_{\Ga_N \sfE} - }", from=3-1, to=3-3]
	\arrow["{f^*}"', from=3-3, to=2-3]
\end{tikzcd}
\end{equation}
    in \(\CAlg(\PrLst)\).
    By assumption, \(F_N\) is an equivalence, and \((f\circ i)^\heartsuit\) is an affine map of schemes. By \cite[Rmk.2.4.4.2]{sag}, we see that \((f\circ i)_{\geq 0}\) is an affine map between connective spectral schemes, and as $\sfU$ and $\sfS$ are flat over $\sfM$, then $f\circ i$ is affine by \Cref{pr:equiv_of_closed_immersions}. Since \(F\) is fully faithful, \(F^\loc\) must be fully faithful, so the right adjoint of \(F^\loc\) maps the tensor unit to the tensor unit.
    By \cite[Th.4.8.5.11 \& Rmk.4.8.5.12]{ha}, we can extract the commutative diagram
\[\begin{tikzcd}
	{\Mod_{A_G}(\Sp_G)^{\langle H\rangle-\loc}} & {\QCoh(\sfU)} \\
	{\Mod_{\underline{\Ga}(\sfE_N,f_*i_*\O_U)}(\Mod_{A_N}(\Sp_N))} & {\Mod_{f_*i_*\O_U}(\QCoh(S))}
	\arrow["F^\loc", from=1-1, to=1-2]
	\arrow[from=2-1, to=1-1]
	\arrow[from=2-1, to=2-2]
	\arrow[from=2-2, to=1-2]
\end{tikzcd}\]
 in $\PrLst$ from (\ref{eq:general_induction_diagram3}). The lower-horizontal functor is an equivalence by assumption, and the right-vertical functor is an equivalence by \cite[Pr.2.2.2.5(2)]{reconstruction}. Thus, \(F^\loc\) is essentially surjective, finishing the proof.
\end{proof}

\subsection{Application to sequences of compact Lie groups}\label{ssec:explicit_sequences_of_compact_liegroups}

Let us now apply \Cref{lem:inductive_step_general} to $U(1)$, $U(n)$, $SU(n)$, and lastly, $Sp(n)$. An important step in many of the proofs below is the use of \Cref{thm:extra-applicable}; see \Cref{ex:u(n)_extra_applicable} for a prelude.

For the unitary case, our induction will start with $U(1)$, which we prove explicitly. Again, we suggest the reader refers to the more explicit approach to this statement outlined in the introduction for some initial motivation.

\begin{prop}\label{thm:ellipticU(1)0affineness}
  Let $p\colon \sfE \to \Spec A$ be an oriented elliptic curve over a strictly even-periodic Noetherian $\E_\infty$-ring $A$. Then $\sfE$ is $U(1)$-equivariantly affine.
\end{prop}
\begin{proof}
    We want to apply \Cref{lem:inductive_step_general} with \(H=e\), \(G=U(1)\), \(N=e\), and both \(\phi\) and \(f\) the unique maps.
    Let us check that the assumptions are fulfilled:
    \begin{enumerate}
        \item $\sfE = \sfE_{U(1)}$ is a nonconnective spectral scheme by assumption. 
        \item Since \(\sfE_e\simeq\Spec A\) is affine, \(\underline{\Ga}(\sfE_e,-)\) is an equivalence.
        \item The map in question is given by the zero section \(e\colon \sfM^\heartsuit\to \sfE^\heartsuit\), which is an effective Cartier divisor by \cite[Lm.1.2.2]{km}.
        \item That \(\sfE^\heartsuit\setminus e\to \M^\heartsuit\) is affine can be checked Zariski locally on \(\M^\heartsuit\), where it follows from choosing Weierstrass coordinates; see \cite[\textsection 13.1.6]{olsson}.\qedhere
    \end{enumerate}
\end{proof}

\begin{prop}\label{thm:U_applicable_affine}
    Let $p\colon \sfE \to \Spec A$ be an oriented elliptic curve over a strictly even-periodic Noetherian $\E_\infty$-ring $A$ and $n\geq 1$. Then $\sfE$ is $U(n)$-equivariantly affine.
\end{prop}
\begin{proof}
    The proof proceeds by induction in \(n\); the base case holds by \Cref{thm:ellipticU(1)0affineness}. For the inductive step, we want to apply \Cref{lem:inductive_step_general} with \(H=U(n)\), \(G=U(n+1)\), \(N=U(1)\), \(\phi\) the standard inclusion as block matrices in the upper-left corner, and  \(f\) the determinant map.
    Let us check that the assumptions are fulfilled:
    \begin{enumerate}
        \item Apply \Cref{thm:extra-applicable}.
        \item $U(n)$-equivariant affineness holds by induction or \Cref{thm:ellipticU(1)0affineness}.
        \item By \Cref{thm:extra-applicable}, also see \Cref{ex:u(n)_extra_applicable}, the composite $\sfE_{U(n)}^\heartsuit \xrightarrow{\phi} \sfE_{U(n+1)}^\heartsuit \xrightarrow{f} \sfE_{U(1)}^\heartsuit$ can be identified with
        \[S^n(\sfE^\heartsuit/\sfM^\heartsuit) \xrightarrow{i_n} S^{n+1}(\sfE^\heartsuit / \sfM^\heartsuit) \xrightarrow{\sigma_{n+1}} \sfE^\heartsuit.\]
        By \Cref{prop:Hilb_Schemes_ECD_affine_complement}, the map \(i_n\) is an effective Cartier divisor.
        \item Similarly, by \Cref{prop:Hilb_Schemes_ECD_affine_complement}, the open complement of \(i_n\) is affine over \(\sfE^\heartsuit\).\qedhere
    \end{enumerate}
\end{proof}

This brings us to the special unitary case.

\begin{prop}\label{thm:SU_applicable_affine}
    Let $p\colon \sfE\to \Spec A$ be an oriented elliptic curve over a strictly even-periodic Noetherian $\E_\infty$-ring $A$ and $n\geq 1$. Then $\sfE$ is $SU(n)$-equivariantly affine.
\end{prop}
\begin{proof}
    Let us induct on the variable $n$; the case of $n=1$ being clear as $SU(1)$ is the trivial group, so \(\Sp_{SU(1)}\simeq \Sp\) and \(\sfE_{SU(1)}\simeq \Spec A\) is affine by assumption. For the inductive step $n\geq 1$, we want to apply \Cref{lem:inductive_step_general} with \(H=SU(n)\), \(G=SU(n+1)\), \(N=e\), \(\phi\) the standard inclusion as block matrices in the upper-left corner, and  \(f\) the unique map.
    Let us check that the assumptions are fulfilled:
    \begin{enumerate}
        \item Apply \Cref{thm:extra-applicable}.
        \item Holds by induction and that we are working over an affine base.
        \item We first note that $SU(n)$ has $(n-1)$-dimensional maximal torus with Weyl group the permutation group $\Sigma_n$. By \Cref{thm:extra-applicable}, we can identify
        \[\sfE_{SU(n)}^\heartsuit \simeq K_n \simeq \Spec \pi_0 A \times_{\sfE^\heartsuit} S^n(\sfE^\heartsuit/\Spec \pi_0 A)\] 
        as the pullback of the symmetric product along the zero section $e \colon \Spec \pi_0 A \to \sfE^\heartsuit$; see the discussion preceding \Cref{lem:SpecialUnitary_Schemes_ECD_affine_complement}. This also allows us to identify the map \(\sfE_{SU(n)}^\heartsuit \to \sfE_{SU(n+1)}^\heartsuit\) as the map \(j_n \colon K_n/\Sigma_n\to K_{n+1}/\Sigma_{n+1}\) induced by the base change of the usual inclusion of symmetric products \(i_n\colon S^n(\sfE^\heartsuit/\M^\heartsuit)\to S^{n+1}(\sfE^\heartsuit/\M^\heartsuit)\), along the zero section. By \Cref{prop:Hilb_Schemes_ECD_affine_complement}, the map \(i_n\) is an effective Cartier divisor relative to \(\sfE^\heartsuit\), hence $j_n$ is an effective Cartier divisor relative to $\M$ by base change (\cite[Tag \href{https://stacks.math.columbia.edu/tag/056Q}{056Q}]{stacks}).
        \item Similarly, we know that the open complement to $i_n$ is affine by \Cref{prop:Hilb_Schemes_ECD_affine_complement}, and open complements and affineness is closed under base change, hence the open complement to $j_n$ is affine over $\M$ as desired.\qedhere
    \end{enumerate}
\end{proof}

Finally, the symplectic case.

\begin{prop}\label{thm:Sp_applicable_affine}
    Let $p\colon \sfE \to \Spec A$ be an oriented elliptic curve over a strictly even-periodic Noetherian $\E_\infty$-ring $A$ and $n\geq 1$. Then $\sfE$ is $Sp(n)$-equivariantly affine.
\end{prop}
\begin{proof}
    The proof proceeds by induction in \(n\); the base case \(n=1\) follows from \Cref{thm:SU_applicable_affine}, by the isomorphism \(Sp(1)\simeq SU(2)\). For the inductive step, we want to apply \Cref{lem:inductive_step_general} with \(H=Sp(n)\), \(G=Sp(n+1)\), \(N=e\), \(\phi\) the standard inclusion as block matrices in the lower right corner, and  \(f\) the unique map.
    Let us check that the assumptions are fulfilled:
    \begin{enumerate}
        \item Apply \Cref{thm:extra-applicable}.
        \item Holds by induction and that we are working over an affine base.
        \item By \Cref{thm:extra-applicable} we have that
        \[\sfE_{Sp(n)}^\heartsuit\xrightarrow{\phi} \sfE_{Sp(n+1)}^\heartsuit\xrightarrow{f}\sfE_e^\heartsuit\]
        agrees with
        \[T^n(\sfE^\heartsuit/\sfM^\heartsuit)\xrightarrow{k_n} T^{n+1}(\sfE^\heartsuit/\sfM^\heartsuit)\to \sfM^\heartsuit\]
        using the notation preceding \Cref{prop:Symplectic_Schemes_ECD_affine_complement}. In particular, the map $k_n$ above is an effective Cartier divisor over $\sfM^\heartsuit$ by \Cref{prop:Symplectic_Schemes_ECD_affine_complement}.
        \item Similarly, by \Cref{prop:Symplectic_Schemes_ECD_affine_complement}, the open complement of \(k_n\) is affine over \(\M^\heartsuit\). \qedhere
    \end{enumerate}
\end{proof}

\begin{remark}\label{rmk:identification-of-divisors}
    Let $\sfE$ be an oriented elliptic curve over stack $\M$. In \cite[Lm.9.1]{davidandlennart}, it is shown that the line bundle $\O_{\sfE}(-e)$, corresponding to the effective Cartier divisor given by the unit section $e\colon \Spec A \to \sfE$, is given by the $U(1)$-equivariant elliptic cohomology of the $U(1)$-representation sphere $S^\mu$, where $\mu$ is the tautological $1$-dimensional unitary $U(1)$-representation. Gepner--Meier prove this for $\M\in \SpDM^\nc$, but the same argument holds for $\M$ in $\Stk$. In terms of the $U(1)$-equivariant affineness of $\sfE$ of \Cref{thm:ellipticU(1)0affineness}, so if $\M = \Spec A$ is a strictly even-periodic Noetherian affine, this identifies the line bundle $\O_{\sfE}(e)$ in $\QCoh(\sfE)$ with the $A_{U(1)}$-module $A_{U(1)}\otimes S^{\mu}$.

    In fact, this is a generic phenomenon for the other matrix groups and the inclusions between them appearing in \Cref{thm:U_applicable_affine,thm:SU_applicable_affine,thm:Sp_applicable_affine}, and does not appeal to equivariant affineness; in all cases, this essentially holds by definition as in the proof of \cite[Lm.9.1]{davidandlennart}. To summarise:
    \begin{itemize}
        \item For $n\geq 1$, the $U(n+1)$-global sections of the line bundle $\O_{\sfE_{U(n+1)}}(\sfE_{U(n)})$ on $\sfE_{U(n+1)}$ are given by the $A_{U(n+1)}$-module $A_{U(n+1)} \otimes S^V$, where $V$ is the tautological $(n+1)$-dimensional complex $U(n+1)$-representation.
        \item For $n\geq 1$, the $SU(n+1)$-global sections of the line bundle $\O_{\sfE_{SU(n+1)}}(\sfE_{SU(n)})$ on $\sfE_{SU(n+1)}$ are given by the $A_{SU(n+1)}$-module $A_{SU(n+1)} \otimes S^V$, where $V$ is the tautological $(n+1)$-dimensional complex $SU(n+1)$-representation.
        \item For $n\geq 1$, the $Sp(n+1)$-global sections of the line bundle $\O_{\sfE_{Sp(n+1)}}(\sfE_{Sp(n)})$ on $\sfE_{Sp(n+1)}$ are given by the $A_{Sp(n+1)}$-module $A_{Sp(n+1)} \otimes S^V$, where $V$ is the tautological $(n+1)$-dimensional quaternionic $Sp(n+1)$-representation.
    \end{itemize}
    Using \Cref{rmk:general_cohomology_definition}, one can also consider similar computations for any compact Lie groups. It seems like an interesting question to try to find a general dictionary between divisors of $\sfE_G$ and $G$-representations, if one exists.
\end{remark}

\subsection{Extension to finite products}\label{ssec:extension_to_finite_products}
Next we deal with products of groups, starting with quasi-coherent sheaves.

\begin{prop}\label{pr:kuenneth_for_qcoh_of_E_G}
    Let $\sfE$ be an oriented elliptic curve over a perfect stack $\M$ with universally 0-affine diagonal and $G,K$ be compact Lie groups such that either $\sfE_G$ or $\sfE_K$ is a perfect stack. Then the natural map induced by the projections $G \gets G\times K \to K$
    \[\QCoh(\sfE_G) \otimes_{\QCoh(\M)} \QCoh(\sfE_K) \to \QCoh(\sfE_{G\times K})\]
    in $\CAlg(\PrLst)$ is an equivalence.
\end{prop}

\begin{proof}
    First, recall that \Cref{thm:Ell_geo_nice_prop}(\ref{it:Ell_product_preserving}) identifies $\sfE_{G\times K} \simeq \sfE_G \times_\M \sfE_K$. We are then exactly in the situation of \Cref{pr:general_kuenneth}, which finishes the proof.
\end{proof}

This immediately yields an equivariant affineness statement.

\begin{prop}\label{pr:propegation_to_products}
    Let $\sfE \to \M$ be an oriented elliptic curve over a 0-affine stack with universally 0-affine diagonal and $G,K$ be two applicable Lie groups. If $\sfE$ is both $G$- and $K$-equivariantly affine, then $\sfE$ is $G\times K$-equivariantly affine.
\end{prop}

\begin{proof}
    We have a commutative diagram
\begin{equation}\label{eq:extend_to_products}
\begin{tikzcd}
	{\Mod_{\Ga_G \sfE}(\Sp_G)\otimes_{\Mod_A(\Sp)}\Mod_{\Ga_K \sfE}(\Sp_K)} & {\QCoh(\sfE_G)\otimes_{\QCoh(\sfM)}\QCoh(\sfE_K)} \\
	{\Mod_{\Ga_{G\times K} \sfE}(\Sp_{G\times K})} & {\QCoh(\sfE_{G\times K})}
	\arrow[from=1-1, to=1-2]
	\arrow["{pr_1^*\otimes pr_2^*}"', from=1-1, to=2-1]
	\arrow["{pr_1^*\otimes pr_2^*}", from=1-2, to=2-2]
	\arrow["{\O_{\sfE_{G\times K}}\otimes_{\Ga_{G\times K} \sfE} - }", from=2-1, to=2-2]
\end{tikzcd}
\end{equation}
    in \(\CAlg(\PrLst)\) courtesy of \Cref{pr:homomorphism_and_linear_homology}. The assumptions that $\sfE$ is both $G$- and $K$-equivariantly affine and that $\M$ is 0-affine show the upper-horizontal functor is an equivalence. Moreover, $G$- and $K$-equivariant affineness show that both $\sfE_G$ and $\sfE_K$ are perfect, respectively. The previous result (\Cref{pr:kuenneth_for_qcoh_of_E_G}) shows that the right-vertical functor is also an equivalence, hence the lower-horizontal functor of (\ref{eq:extend_to_products}) is essentially surjective. By \Cref{pr:left_adjoint_is_fullyfaithful}, this functor is also fully faithful, finishing the proof.
\end{proof}

This gives us the affine version of our main equivariant affineness result.

\begin{cor}\label{cor:affine_version_of_main_statement}
    Let $G= \prod_i U(n_i) \times \prod_j SU(n_j) \times \prod_k Sp(n_k)$ be a finite product of unitary, special unitary, and symplectic groups. Then any oriented elliptic curve $\sfE$ over a strictly even-periodic Noetherian $\E_\infty$-ring $A$ is $G$-equivariantly affine.
\end{cor}

\begin{proof}
    Combine \Cref{thm:SU_applicable_affine,thm:U_applicable_affine,thm:Sp_applicable_affine} together with \Cref{pr:propegation_to_products}.
\end{proof}

\subsection{Equivariant quasi-coherent sheaves and base change}\label{ssec:equivariant_qcoh_and_base_change}
Let us recall the definition of $\QCoh_G(\M)$ from the introduction.

\begin{mydef}
    For a compact Lie group $G$ and a stack $\M$, let
    \[\QCoh_G(\M) = \QCoh(\M) \otimes \Sp_G \in \CAlg(\PrLst)\]
    denote the category of \emph{$G$-equivariant quasi-coherent sheaves}.
\end{mydef}

To emphasise, this definition is not supposed to be a genuinely equivariant replacement for quasi-coherent sheaves on $\M$, that would presumably require $\M$ itself to be a kind of genuinely equivariant stack $\M$. This is rather a definition designed to provide a relativised version of \Cref{maintheorem_universal}, of course given as \Cref{maintheorem_arbitrary_base}.

In the affine case $\M=\Spec A$, one has
\begin{equation}\label{eq:qcoh_G_on_affines}\QCoh_G(\Spec A) \simeq \Mod_A(\Sp) \otimes \Sp_G \simeq \Mod_{L_G A}(\Sp_G),\end{equation}
where $L_G \colon \Sp \to \Sp_G$ is the symmetric monoidal left adjoint to the underlying spectrum functor $U_G \colon \Sp_G \to \Sp$. As $\Sp_G$ is dualisable in $\CAlg(\PrLst)$, being compactly generated, one can rewrite $\QCoh_G(\X)$ using a presentation of $\M = \colim \Spec A$ as
\[\QCoh_G(\M) = \QCoh(\M) \otimes \Sp_G {\simeq} \lim (\Mod_A(\Sp) \otimes \Sp_G) \simeq \lim \Mod_{L_G A}(\Sp_G).\]
If $\M$ is 0-affine, such as the universal case of $\M=\M_\Ell^\ori$, we can also identify
\begin{equation}\label{eq:0affine_qcoh_g_comptuation}\QCoh_G(\M) = \QCoh(\M) \otimes \Sp_G \simeq \Mod_{\Ga(\M)}(\Sp) \otimes \Sp_G \simeq \Mod_{L_G \Ga(\M)}(\Sp_G).\end{equation}
If one further assumes that $\Ga(\M)$ comes equipped with a $G$-equivariant refinement, so some $A \in \CAlg(\Sp_G)$ such that $U_G(A) = \Ga(\M)$, then the counit for $L_G \dashv U_G$ also identifies
\[\Mod_A(\Sp_G) \simeq \Mod_A(\Mod_{L_G (\Ga(\M))}(\Sp_G)) \simeq \Mod_{A}(\QCoh_G(\M)).\]
One can then formally globalise the notion of an equivariant global sections functor.

\begin{mydef}\label{df:g-pushforward}
    Let $G$ be an applicable compact Lie group and $p \colon \sfE \to \M$ an oriented elliptic curve. The right adjoint to the $\QCoh(\M)$-linearisation of the functor
    \[\O_{\sfE_G} \otimes - \colon \Sp_G \to \QCoh(\sfE_G) \in \CAlg(\PrLst)\]
    of \Cref{def:elliptic_homology} takes the form $p_{\ast,G} \colon \QCoh(\sfE_G) \to \QCoh_G(\M)$. This right adjoint is lax monoidal, so it naturally factors as the \emph{$G$-equivariant pushforward}
    \[p_{\ast,G} \colon \QCoh(\sfE_G) \to \Mod_{p_{\ast,G} \O_{\sfE_G}}(\QCoh_G(\M)).\]
    By construction, for each closed embedding $\phi \colon H \to G$, the $H$-fixed points of $p_{\ast,G}$ are computed as
    \[p_{\ast,G}(\calF)^H \simeq (p_H)_\ast \O_{\sfE_H} \in \QCoh(\M),\]
    where $p_H \colon \sfE_H \to \M$ is the structure map.
\end{mydef}

\begin{remark}
    Alternatively, one can define $p_{\ast,G}$ using the right adjoint to the $\Sp_G$-linearisation of the pullback $(p_G)^\ast \colon \QCoh(\M) \to \QCoh(\sfE_G)$.
\end{remark}

Using the concept of equivariant quasi-coherent sheaves, one can state a general base change property for equivariant elliptic cohomology theories.

\begin{theorem}\label{pr:base_change_for_equivariant_theories}
	Let $p \colon \sfE \to \M$ be an oriented elliptic curve over a stack $\M$, $f\colon \M'\to \M$ be a map of stacks, and $G$ an applicable compact Lie group. Write $p'\colon \sfE' \to \M'$ for the oriented elliptic curve over $\M'$ defined by pulling back $\sfE$. Then natural Beck--Chevalley map
    \[f_G^\ast p_{\ast,G} \O_{\sfE_G} \to p'_{\ast,G} \O_{\sfE'_G} \in \CAlg(\QCoh_G(\M'))\]
    is an equivalence, where $f_G^\ast\colon \QCoh_G(\M) \to \QCoh_G(\sfM')$ is the tensor product of the pullback functor $f^\ast\colon \QCoh(\M) \to \QCoh(\M')$ with $\Sp_G$.
\end{theorem}
This generalised base change formula has many potential applications; in \Cref{sec:appendix2}, we will see how it implies an equivariant form of Wood's theorem for $\TMF$.

\begin{example}\label{ex:relative_base_change}
    If both $\M$ and $\M'$ are 0-affine, then \Cref{pr:base_change_for_equivariant_theories} takes the form
    \[\Ga_G \sfE \otimes_{\Ga(\M)} \Ga(\M') \simeq \Ga_G \sfE' \in \CAlg(\Sp_G).\]
    For instance, consider the \emph{oriented Tate curve} $\sfT$ over either $\Spec \KU\llpar q\rrpar$ or $\Spec \KU\llpar q\rrpar / C_2$ of \cite[Th.A]{globaltate}; the associated $G$-equivariant $\E_\infty$-rings are written as $\KU\llpar q\rrpar_G$ and $\KO\llpar q\rrpar$, respectively. The natural quotient map of stacks then gives the identification
    \[\KO\llpar q\rrpar_G \otimes_{\KO\llpar q \rrpar} \KU\llpar q \rrpar \simeq \KU\llpar q \rrpar_G.\]
    When mixed with the universal case, this also gives the identifications
    \[\TMF_G \otimes_\TMF \KU\llpar q \rrpar \simeq \KU\llpar q \rrpar_G \qquad \TMF_G \otimes_\TMF \KO\llpar q \rrpar \simeq \KO\llpar q \rrpar_G.\qedhere\]
\end{example}

\begin{remark}\label{rmk:base_change_for_general_guys}
    Staying within the situation of \Cref{ex:relative_base_change}, if one defines $\Ga_G \sfE$ and $\Ga_G \sfE'$ for a general compact Lie group $G$ as in \Cref{rmk:TMFG_for_general_G}, then one immediately obtains the equivalences in $\CAlg(\Sp_G)$
    \[\Ga_G \sfE \otimes_{\Ga(\M)} \Ga(\M') = \res_{G}^{U(n)} (\Ga_{U(n)} \sfE) \otimes_{\Ga(\M)}  \Ga(\M') \simeq \res_{G}^{U(n)} (\Ga_{U(n)} \sfE \otimes_{\Ga(\M)} \Ga(\M')) \simeq \Ga_G \sfE'\]
    from the $U(n)$-case and the definition of $\Ga_G \sfE' = \res^{U(n)}_G \Ga_{U(n)}\sfE'$, as restriction is symmetric monoidal and commutes with colimits. In particular, this applies in the universal case:
    \[\TMF_G \otimes_{\Ga(\M)} \Ga(\M) \simeq \Ga_G \sfE\qedhere\]
\end{remark}

To prove the above base change result, we first show a related general adjointability fact.

\begin{lemma}\label{lem:elliptic_stacks_uni_adjointable}
    Let \(\sfE\) be an oriented elliptic curve over a base stack \(\sfM\), and let \(G\) be a compact Lie group.
    Then the map \(\sfE_G\to\sf M\) is universally adjointable.
\end{lemma}

\begin{proof}
    By \Cref{thm:Ell_geo_nice_prop}(\ref{it:base_change_in_stacks}), and the fact that universally adjointable morphisms satisfy base change, it is enough to prove this in the universal case \(\calE_G\to \sfM_\Ell^\ori\).
    The stack \(\sfM_\Ell^\ori\) admits an effective epimorphism from a strictly even-periodic Noetherian affine \(\Spec R\), so by \Cref{lem:uni_adjointable_local_finitary} it suffices to work over such a base. Choose a closed embedding \(G\subset U(n)\).
    By \Cref{thm:Ell_geo_nice_prop}(\ref{it:Ell_closed_embedding_affine}) the resulting map \(\sfE_G\to \sfE_{U(n)}\) is affine, thus universally adjointable.
    Thus, it will suffice to prove this for \(G=U(n)\).
    As a consequence of \Cref{ex:schemes_are_passable} and \Cref{thm:extra-applicable}, we find that the stack \(\sfE_{U(n)}\) is perfect.
    By \Cref{pr:permancence_per_proper}, the map \(\sfE_{U(n)}\to \sfM=\Spec R\) is perfect, thus universally adjointable by \Cref{pr:main_perfect_stacks_statement}, and we are done.
\end{proof}

Before proving \Cref{pr:base_change_for_equivariant_theories}, we separate out one final categorical lemma.

\begin{lemma}\label{lem:conversative_after_tensor}
Suppose $F\colon \calC\to \calC'$ is a functor of presentable categories which is both a left and a right adjoint. Then, given any $\calD\in \PrL$ the following square commutes 
\[\begin{tikzcd}
	{\calD\otimes \calC} & {\calD\otimes \calC'} \\
	{\Fun^R(\calD^{\op},\calC)} & {\Fun^R(\calD^{\op},\calC')}
	\arrow["{\calD\otimes F}", from=1-1, to=1-2]
	\arrow["\simeq", from=1-1, to=2-1]
	\arrow["\simeq", from=1-2, to=2-2]
	\arrow["{F_*}", from=2-1, to=2-2]
\end{tikzcd},\]
where the vertical equivalences are \cite[Pr.4.8.1.17]{ha}. In particular, if $F$ is conservative then $\calD \otimes F$ is again conservative.
\end{lemma}

\begin{proof}
It suffices to prove that the square
\[\begin{tikzcd}
	{\calD\times \calC} & {\calD\times \calC'} \\
	{\Fun^R(\calD^{\op},\calC)} & {\Fun^R(\calD^{\op},\calC')}
	\arrow["{\calD\times F}", from=1-1, to=1-2]
	\arrow["{- \boxtimes - }"', from=1-1, to=2-1]
	\arrow["{- \boxtimes -}", from=1-2, to=2-2]
	\arrow["{F_*}", from=2-1, to=2-2]
\end{tikzcd}\]
commutes, where now the vertical arrows are the universal functors. Unravelling the proof of \cite[Pr.4.8.1.17]{ha}, we see that $d\boxtimes c \simeq L_d(c)$, where $L_d$ is the left adjoint of $\mathrm{ev}_{d}\colon \Fun^R(\calD^{\op},\calC)\to \calC$. But then $d\boxtimes F(c)\simeq L_d(F(c)) \simeq F_* L_d(c)$, where the final equivalence can be checked after passing to right adjoints. The second statement follows immediately.
\end{proof}
	
\begin{proof}[Proof of \Cref{pr:base_change_for_equivariant_theories}]
    Note that it suffices to prove the universal case, so where $\M=\M_\Ell^\ori$ and $\sfE=\calE$, as the relative case follows by pasting together two Beck--Chevalley equivalences. In the universal case then, consider the map of stacks $f\colon \M\to \M_\Ell^\ori$ defining $\sfE$, write $\pi\colon \calE \to \M_\Ell^\ori$ for the universal projection, and $g\colon \sfE_G \to \calE_G$ for the base change of $f$ along $\calE_G \to \M_\Ell^\ori$. Let us also abbreviate $f_G^\ast$ as $f^\ast$; context should abate any potential confusion.
    
    The functors $\O_{\calE_G}\otimes -$ and $\O_{\sfE_G}\otimes -$ each admit $\QCoh(\M_\Ell^\ori)$- and $\QCoh(\M)$-linearisations
    \[\calF \colon \QCoh_G(\M_\Ell^\ori) \to \QCoh(\calE_G), \qquad F \colon \QCoh_G(\M) \to \QCoh(\sfE_G),\]
    respectively, whose right adjoints are given by the $G$-equivariant pushforwards of \Cref{df:g-pushforward}. Consider the square
\begin{equation}\label{eq:base_change_BC_square}\begin{tikzcd}
	{\QCoh_G(\sfM_{\Ell}^\ori)} && {\QCoh({\cal E}_G)} \\
	{\QCoh_G(\sfM)} && {\QCoh(\sfE_G).}
	\arrow["{\calF}", from=1-1, to=1-3]
	\arrow["{f^*}"', from=1-1, to=2-1]
	\arrow["{g^*}", from=1-3, to=2-3]
	\arrow["{F}"', from=2-1, to=2-3]
\end{tikzcd}\end{equation}
    To see that this square commutes, it suffices to check this on the compact generators $G/H_* \otimes \O_{\M_\Ell^\ori}$ of $\QCoh_G(\M_\Ell^\ori)$ as all functors commute with colimits. We then identify
    \[g^\ast \calF(G/H_* \otimes \O_{\M_\Ell^\ori}) \simeq g^\ast i_\ast \O_{\calE_H}
    \simeq j_\ast \O_{\sfE_H} \simeq F(G/H_+ \otimes \O_{\M}) \simeq F(f^\ast(G/H_+ \otimes \O_{\M_\Ell^\ori})),\]
    courtesy of the adjointable pullback of stacks
    \[\begin{tikzcd}
        {\sfE_H}\ar[r, "j"]\ar[d, "h"]  &   {\sfE_G}\ar[d, "g"] \\
        {\calE_H}\ar[r, "i"]            &   {\calE_G,}
    \end{tikzcd}\]
    which is a pullback by the naturality of the equivalence in \Cref{thm:Ell_geo_nice_prop}(\ref{it:base_change_in_stacks}), and adjointable as $i$ and $j$ are affine by \Cref{thm:Ell_geo_nice_prop}(\ref{it:Ell_closed_embedding_affine}) and the fact that affine maps are adjointable by \cite[Pr.2.2.2.5]{reconstruction}. The comparison map in question is by definition the Beck--Chevalley transformation associated with the square (\ref{eq:base_change_BC_square}) evaluated at \(\O_{{\cal E}_G}\). Let \(H\subset G\) be a closed subgroup.
    Since \(G/H_+\in \Sp_G\) is compact, the square
\[\begin{tikzcd}
	{\QCoh(\sfM_{\Ell}^\ori)} && {\QCoh_G(\sfM_{\Ell}^\ori)} \\
	{\QCoh(\sfM)} && {\QCoh_G(\sfM)}
	\arrow["{(G/H_+)\otimes - }", from=1-1, to=1-3]
	\arrow["{f^*}"', from=1-1, to=2-1]
	\arrow["{f^*}", from=1-3, to=2-3]
	\arrow["{(G/H_+)\otimes - }"', from=2-1, to=2-3]
\end{tikzcd}\]
is horizontally right adjointable. Now we note that the horizontal right adjoints above \((-)^H\) are jointly conservative as $H$ varies over conjugacy classes of closed subgroups of $G$. Indeed, this follows from \Cref{lem:conversative_after_tensor} applied to the functor $\Sp_G \to \prod_{(H)} \Sp$ given by taking $H$-fixed points for each conjugacy class of closed subgroups $H\leq G$, which is well-known to be conservative. It therefore suffices to prove that the comparison map on $H$-fixed points
\[(f^\ast \pi_{\ast,G} \O_{\calE_G})^H \simeq f^\ast (\pi_H)_\ast \O_{\calE_H} \to (p_H)_\ast \O_{\sfE_H} \simeq (p_{\ast,G} \O_{\sfE_G})^H\]
is an equivalence, where $\pi_H \colon \calE_H \to \M_\Ell^\ori$ is the structure map and $p_H$ is its pullback along $f$. However, the above map is the Beck--Chevalley transformation associated with the pullback square
\[\begin{tikzcd}
	{\sfE_H}\ar[r, "{p_H}"]\ar[d] & \sfM\ar[d, "{f}"] \\
	{\calE_H}\ar[r, "{\pi_H}"] & {\sfM_\Ell^\ori}
\end{tikzcd}\]
of \Cref{thm:Ell_geo_nice_prop}(\ref{it:base_change_in_stacks}), evaluated at \(\O_{\calE_H}\), which is an equivalence by \Cref{lem:elliptic_stacks_uni_adjointable}.
\end{proof}

\subsection{Over a general base stack}\label{ssec:general_case}
From the case over a strictly even-periodic Noetherian affine, we can formally deduce equivariant affineness for the universal oriented elliptic curve $\calE$.

\begin{theorem}\label{thm:step_to_universal_case}
    Let $\Spec R^\bullet$ be a diagram of affine stacks with colimit $\M_\Ell^\ori$ and write $\sfE^\bullet$ for the pullback of the universal oriented elliptic curve along $\Spec R^\bullet \to \M_\Ell^\ori$. Suppose $G$ is an applicable Lie group such that each $\sfE^\bullet$ is $G$-equivariantly affine. Then $\calE$ is $G$-equivariantly affine.
\end{theorem}

An alternative proof using an explicit descendability argument is given in \Cref{rmk:alternative_proof}.

\begin{proof}
    In this proof, let us occasionally abbreviate $\Mod_A(\calC)$ as $\Mod_A$ for brevity; the appropriate $\calC$ will be clear from context. We then compute quasi-coherent sheaves on $\calE_G$ as follows:
     \begin{align*}
    \QCoh(\calE_G) \simeq  & \lim \QCoh(E_G^\bullet) \\
                   \simeq  & \lim \Mod_{R^\bullet_{G}} \\
                   \simeq  & \lim \Mod_{R^\bullet_{G}}(\Mod_{L_G R^\bullet})\\
                   \simeq  & \lim \Mod_{R^\bullet \otimes_\TMF \TMF_G}\left(\Mod_{L_G R^\bullet}\right) \\
                   \simeq  & \lim \Mod_{R^\bullet \otimes_\TMF \TMF_G}\left(\Mod_{R^\bullet} \otimes_{\Mod_\TMF} \Mod_{L_G \TMF} \right) \\
                   \simeq  & \Mod_{\TMF \otimes_\TMF \TMF_G}\left(\Mod_\TMF \otimes_{\Mod_\TMF} \Mod_{L_G \TMF} \right) \\
                   \simeq  & \Mod_{\TMF_G}\left(\Mod_{L_G \TMF} \right) \\
                   \simeq  & \Mod_{\TMF_G}(\Sp_G)
    \end{align*}
    The first equivalence follows from the equivalence of stacks
    \[\calE_G \simeq \calE_G \times_{\M_\Ell^\ori} \colim \Spec R^\bullet \simeq \colim E_G^\bullet;\]
    the second holds by the $G$-equivariantly affine assumption; the third from canonical counit of $\E_\infty$-$G$-rings $L_G R^\bullet \to R^\bullet_{G}$; the forth by \Cref{pr:base_change_for_equivariant_theories} (see \Cref{ex:relative_base_change}); the fifth is an easy consequence of (\ref{eq:qcoh_G_on_affines}); the sixth follows as the algebra in the fifth line corresponding to $\TMF \otimes_\TMF \TMF_G$ under the equivalence of categories
    \[\Mod_{\TMF} \otimes_{\Mod_\TMF} \Mod_{L_G \TMF} \simeq \lim \left( \Mod_{R^\bullet} \otimes_{\Mod_\TMF} \Mod_{L_G \TMF}\right),\]
    which uses that $\Sp_G$ is dualisable as a stable category,
    is precisely $R^\bullet \otimes_\TMF \TMF_G$ by \Cref{ex:relative_base_change}; the seventh is just cancellation; and finally, the eighth follows from the counit $L_G \TMF \to \TMF_G$.
\end{proof}

Note that $\M_\Ell^\ori$ has a cover by strictly even-periodic Noetherian affines $\Spec R$. For example, take the affine étale cover
    \[\Spec R = \Spec \TMF(2) \times \TMF(3) \to \M_\Ell^\ori.\]
    The terms in its associated Čech nerve remain affine strictly even-periodic Noetherian, as each term is étale over $R$ and $\M_\Ell^\ori$ has affine diagonal:
    
\begin{lemma}\label{lm:mellori_affine_diagonal}
    The stack $\M_\Ell^\ori$ has affine, hence universally 0-affine, diagonal.
\end{lemma}

\begin{proof}
    Consider the unique map $\M_\Ell^\ori \to \M_\FG^\ori$ into Lurie's \emph{moduli stack of oriented formal groups} from \cite[Df.4.1.1.2]{reconstruction}, which is affine by \cite[Ex.4.2.3.4]{reconstruction}. Consider the following commutative diagram of stacks
    \[\begin{tikzcd}
        {\M_\Ell^\ori}\ar[r, "\Delta"]\ar[d]    &   {\M_\Ell^\ori \times \M_\Ell^\ori}\ar[d]    \\
        {\M_\FG^\ori}\ar[r, "\Delta"]   &   {\M_\FG^\ori \times \M_\FG^\ori.}
    \end{tikzcd}\]
    We just stated that the vertical maps are affine, and the lower horizontal map is an equivalence by \cite[Cor.4.1.4.1]{reconstruction}. As affine morphisms satisfy cancellation, see \cite[Pr.A.3]{Davies_galois_geometric_FP}, we conclude that $\M_\Ell^\ori$ has affine diagonal.
\end{proof}
    
Combining \Cref{cor:affine_version_of_main_statement} with \Cref{thm:step_to_universal_case} yields \Cref{maintheorem_universal}:

\begin{theorem}[\Cref{maintheorem_universal}]
    Let $G= \prod_i U(n_i) \times \prod_j SU(n_j) \times \prod_k \Sp(n_k)$ be a finite product of unitary, special unitary, and symplectic groups. Then $\calE$ is $G$-equivariantly affine.
\end{theorem}

The proof of \Cref{maintheorem_arbitrary_base} then follows using the next short statement.

\begin{prop}\label{pr:kuenneth_for_general_qcoh_orientedEll}
    Let $\sfE$ be an oriented elliptic curve over a stack $\M$ defining a map $\M \to \M_\Ell^\ori$ and $G$ be a compact Lie group such that the universal $\calE_G$ is a perfect stack. Then the natural map in $\CAlg(\PrLst)$
    \[\QCoh(\EE_G) \otimes_{\QCoh(\M_\Ell^\ori)} \QCoh(\M) \xrightarrow{\simeq} \QCoh(E_G)\]
    is an equivalence.
\end{prop}

\begin{proof}
    From our hypotheses, \Cref{ex:MM_mellor_affineness_intro}, and \Cref{lm:mellori_affine_diagonal}, the conditions of \Cref{pr:general_kuenneth} are satisfied, which immediately finishes the proof.
\end{proof}

\begin{theorem}[{\Cref{maintheorem_arbitrary_base}}]
    Let $G= \prod_i U(n_i) \times \prod_j SU(n_j) \times \prod_k \Sp(n_k)$ be a finite product of unitary, special unitary, and symplectic groups. Then for any oriented elliptic curve $p\colon \sfE \to \M$ over a stack $\M$, the $G$-pushforward functor
    \[p_{\ast,G} \colon \QCoh(\sfE_G) \xrightarrow{\simeq} \Mod_{p_{\ast,G} \O_{\sfE_G}}(\QCoh_G(\M))\]
    is an equivalence of symmetric monoidal categories.
\end{theorem}

It should be highlighted that even in the affine case, this generalises \Cref{cor:affine_version_of_main_statement}, as there we assume our affine base is strictly even-periodic and Noetherian. If $\M$ is 0-affine, this also shows that $\sfE_G$ is equivariantly affine.

\begin{proof}
    Consider the following equivalences in $\CAlg(\PrLst)$:
    \begin{align*}
    \QCoh(E_G) & \xleftarrow{\simeq} \QCoh(\EE_G) \otimes_{\QCoh(\M_\Ell^\ori)} \QCoh(\M)\\
    &\xrightarrow{\simeq} \Mod_{\TMF_G}(\Sp_G) \otimes_{\QCoh(\M_\Ell^\ori)} \QCoh(\M) \\
    &\xrightarrow{\simeq} \Mod_{\TMF_G}(\QCoh_G(\M_\Ell^\ori)) \otimes_{\QCoh(\M_\Ell^\ori)} \QCoh(\M) \\
    & \simeq \Mod_{\TMF_G \otimes_\TMF \O_\M}(\QCoh_G(\M_\Ell^\ori) \otimes_{\QCoh(\M_\Ell^\ori)} \QCoh(\M))\\
    & \simeq \Mod_{p_{\ast,G} \O_{\sfE_G}}(\QCoh_G(\M))
    \end{align*}
The first equivalence is \Cref{pr:kuenneth_for_general_qcoh_orientedEll}, which holds as $\calE_G$ is perfect by \Cref{maintheorem_universal} and $\M_\Ell^\ori$ is perfect with affine diagonal by \Cref{ex:MM_mellor_affineness_intro} and \Cref{lm:mellori_affine_diagonal}. The second equivalence is \Cref{maintheorem_universal}, the third is just (\ref{eq:0affine_qcoh_g_comptuation}), and the rest are clear.
\end{proof}

\appendix
\section{The spectral functor-of-points}\label{ssec:properties_preserved_by_h}
By \cite[Th.3.2.1.3]{reconstruction}, there is a unique fully faithful embedding
\[h^\et_{(-)} \colon \SpDM^\nc \to \Stk^\et\]
extending the assignment on affines $\Spet R \mapsto \Spec R$ and commuting with étale colimits, and \cite[Pr.2.1.2.6]{reconstruction} provides us with a colimit-preserving sheafification functor
\[L^\fpqc \colon \Stk^\et \to \Stk.\]
The composite is written as
\[h_{(-)} \colon \SpDM^\cn \xrightarrow{h^\et_{(-)}} \Stk^\et \xrightarrow{L^\fpqc} \Stk.\]
The following complements some of the basic facts above $h_{(-)}$ from \cite[Pr.3.2.2.2]{reconstruction}; the subtle point is that this functor is not generally fully faithful.

\begin{lemma}\label{lem:h_finite_limits}
    The functors \(h_{(-)}^\et\) and \(h_{(-)}\) preserve finite limits.
\end{lemma}

\begin{proof}
    As sheafification \(L^\fpqc\) preserves finite limits (\cite[Pr.2.1.2.6]{reconstruction}), it suffices to show that \(h^\et_{(-)}\) preserves small limits. Let $\X_i$ be a diagram in $\SpDM^\nc$ with limit $\X$ and $\sfT$ be an object of $\Stk^\et$ with presentation $\sfT \simeq \colim \Spec A$. Observe the desired equivalences
    \[\Stk^\et(\sfT, h^\et_\X) \simeq \lim_A \Stk^\et(\Spec A, h^\et_\X) \simeq \lim_A \Stk^\et(h^\et_{\Spet A}, h^\et_\X) \simeq \lim_A \SpDM^\nc(\Spet A, \X)\]
    \[\simeq \lim_A \lim_i \SpDM^\nc(\Spet A, \X_i) \simeq \lim_i \lim_A \Stk^\et(\Spec A, h^\et_{\X_i}) \simeq \lim_i \Stk^\et(\sfT, h^\et_{\X_i}),\]
    the third and fifth equivalences coming from the fully faithfulness of $h_{(-)}^\et$.
\end{proof}

\begin{lemma}\label{lem:h_map_from_affine_local_lift}
    Let \(\X\in \SpDM^\nc\) and let \(p\colon \Spec A\to h_\X\) be a map in $\Stk$. Then there exists a faithfully flat map of $\E_\infty$-rings \(q\colon A\to B\) and a map \(f\colon \Spet B\to \X\) in $\SpDM^\nc$ together with a homotopy \(h_f\simeq p\circ \Spec q\).
\end{lemma}

For the proof, we will use the notation $\Shv_\kappa^{\fpqc}$ for the category of fpqc-sheaves on $\kappa$-compact $\E_\infty$-rings for a regular cardinal $\kappa$. By \cite[Cor.3.1.1.4]{reconstruction}, we have an identification
\begin{equation}\label{eq:presentation_of_stack}
    \Stk \simeq \colim_\kappa \Shv_\kappa^{\fpqc},
\end{equation}
where all of the transition maps are left adjoint functors.

\begin{proof}
We can write our $\X$ in $\SpDM^\nc$ as a colimit of affines $\Spet R$ with étale transition maps. As $h_{(-)}$ preserves étale colimits (\cite[Pr.3.2.2.2(2)]{reconstruction}), we also have $h_\X \simeq \colim \Spec R$ in $\Stk$. Fix a regular cardinal \(\kappa\) such that \(A\) and all of the \(R\) are \(\kappa\)-compact.
Let \(h^\kappa_\X\) denote the colimit of \(\Spec_\kappa R\) in \(\Shv^\fpqc_\kappa\). As the transition maps in the diagram defining the colimit (\ref{eq:presentation_of_stack}) all preserve colimits, we find that the image of \(h^\kappa_\X\) in \(\Stk\) is \(h_\X\), and that, for \(\lambda \geq \kappa\), the image of \(h^\kappa_\X\) is \(h^\lambda_\X\). Since the presentation of \(\Stk\) from (\ref{eq:presentation_of_stack}) is filtered, \(p\) lifts to a map
    \[p_\lambda\colon \Spec_\lambda A\to h^\lambda_\X \in \Shv^\fpqc_\lambda\]
for some \(\lambda\geq \kappa\), where $\Shv^\fpqc_\lambda$ is an actual topos. In particular, we can apply \cite[Lm.6.2.3.13]{htt} to see that the morphism
    \[r\colon \coprod_{i\in I}\Spec_\lambda R_i\to h^\lambda_\X\]
is an effective epimorphism in \(\Shv^\fpqc_\lambda\). By \cite[Cor.B.7]{HesselholtPstragowski_2024_DiracGeometryII}, there is a faithfully flat map \(q\colon A\to B\) with \(B\) a \(\lambda\)-compact $\E_\infty$-ring, and a factorisation of the composite $\Spec_\lambda B\to \Spec_\lambda A \to h^\lambda_\X$ through \(r\). Passing this diagram from $\Shv^\fpqc_\lambda$ to \(\Stk\) yields the commutative diagram
\[\begin{tikzcd}
	{\Spec B} & {\coprod_{i\in I}\Spec R_i} \\
	{\Spec A} & {h_\X.}
	\arrow["g", from=1-1, to=1-2]
	\arrow["{\Spec q}"', from=1-1, to=2-1]
	\arrow["r", from=1-2, to=2-2]
	\arrow["p"', from=2-1, to=2-2]
\end{tikzcd}\]
    Since coproducts exist in \(\SpDM^\nc\), \(h_{(-)}\) preserves coproducts, and spectral Deligne--Mumford stacks of the form \(\coprod_{i\in I}\Spet R_i\) are schematic, the map \(g\) lifts uniquely to a map \(g'\colon \Spet B \to \coprod_{i\in I}\Spet R_i\) courtesy of \cite[Pr.3.2.2.2(3) \& Pr.3.2.2.4]{reconstruction}. By definition, the map \(r\) also lifts to \(\SpDM^\nc\), giving the desired result.
\end{proof}

Altogether this allows us to compare the notion of affine morphisms in $\SpDM^\nc$ and $\Stk$.

\begin{cor}\label{cor:h_affine_map}
    If \(f\colon\Y\to \X\) is affine in \(\SpDM^\nc\), then \(h_f\colon h_\Y\to h_\X\) is an affine map in \(\Stk\). If $\Y$ and $\X$ both satisfy \emph{flat descent}, as defined in \cite[Df.3.2.2.1]{reconstruction}, then the converse also holds.
\end{cor}

By \cite[Pr.3.2.2.4]{reconstruction}, all nonconnective spectral Deligne--Mumford stacks with quasi-affine diagonal satisfy flat descent. By \cite[Th.1.6.2.1]{sag}, nonconnective spectral schemes also satisfy fpqc flat descent.

\begin{proof}
    First, assume that $f$ is affine. Let \(p\colon \Spec A\to h_\X\) be any map from an affine, and denote the base change of \(h_f\) along \(p\) by \(\Y_A \to\Spec A\). We want to show that \(\Y_A\) is affine. By \Cref{lem:h_map_from_affine_local_lift}, there is a faithfully flat map \(A\to B\) such that the composite $h_{\Spet B}\simeq \Spec B\to \Spec A \to h_\X$ lifts to a map $\Spet B \to \X$ in \(\SpDM^\nc\). As $f$ is affine in $\SpDM^\nc$, the pullback of $\Spet B$ along $f$ is affine $\Spet C$. By \Cref{lem:h_finite_limits}, we see that $\Spec C$ is the pullback of $\Spec B \to h_\X$ along $h_f$. As \(\Spec B\to \Spec A\) is an effective epimorphism and affine maps are local \cite[Pr.A.16]{Tokic_2026_FamilyCompletionTheorem}, otherwise phrased as satisfying descent in $\Stk$ by \cite[Pr.A.3]{Davies_galois_geometric_FP}, we conclude.

    For the converse when both $\Y$ and $\X$ satisfy flat descent, first note that such objects of $\SpDM^\nc$ are closed under limits and contain affines. As \(h_{(-)}\) preserves finite limits by \Cref{lem:h_finite_limits}, and that \(h_{(-)}\) is fully faithful when restricted to objects satisfying flat descent by \cite[Pr.3.2.2.2(3)]{reconstruction}, we are done.
\end{proof}

Finally, we would like to compare the definitions of various immersions given in $\Stk$ and $\SpDM^\nc$. This is mainly to tie in \Cref{ssec:immersions} with \cite{temperedglobal,Tokic_2026_FamilyCompletionTheorem}.

To each stack \(\X\), we can associate a topological space \(|\X|\), see \cite[Con.5.5.1]{temperedglobal}, which on an affine \(\Spec A\) agrees with the classical Zariski spectrum of \(\pi_0 A\), resulting in a colimit-preserving functor
\[|-|\colon \Stk\to \Top,\]
where $\Top$ is the 1-category of topological spaces.

\begin{mydef}\label{def:open_immersion}
    A map \(f\colon \Y\to\X\) of stacks is an {\it open immersion} if \(|f|\colon |\Y|\to |\X|\) is an open embedding of topological spaces, and \(f\) is an \(|-|\)-Cartesian edge. An open immersion $f$ such that \(|f|^{-1}\) preserves quasi-compact open subsets is a {\it quasi-compact (qc-) open immersion}.
\end{mydef}

This is clearly a generalisation of \cite[Df.3.3.1.1]{reconstruction} where the target $\X=\Spec A$ was affine. The corresponding definition in $\SpDM^\nc$ is \Cref{df:open_immersion_SPDM}.

\begin{prop}\label{prop:h_open_immersion}
    Let \(i\colon \U \to \X\) be a (qc-) open immersion of (nonconnective) spectral Deligne--Mumford stacks.
    Then \(h_i\colon h_\U\to h_\X\) is a (qc-) open immersion of stacks.
\end{prop}
\begin{proof}
    By \cite[Pr.A.29]{Tokic_2026_FamilyCompletionTheorem}, it suffices to show that, for every map \(f\colon \Spec A\to h_\X\), the base change \(\U_A\to \Spec A\) of \(h_\U\) is an open immersion. By \Cref{lem:h_map_from_affine_local_lift}, there is a faithfully flat map \(q\colon \Spec B\to \Spec A\) such that \(f\circ q\) lifts to a map \(\Spet B\to \X\).
    Thus, by \Cref{lem:h_finite_limits}, the further base change of \(\U_A\) to \(\Spec B\) is an open immersion.
    Since the class of open immersions satisfies descent by \cite[Lm.A.30]{Tokic_2026_FamilyCompletionTheorem}, then \(\U_A\to \Spec A\) is also an open immersion. Finally, suppose that \(i\colon \U\to \X\) is quasi-compact.
    Since being a quasi-compact map of topological spaces is detected on the associated frames of opens, by \cite[Rmk.A.7]{Tokic_2026_FamilyCompletionTheorem}, it will suffice to show that \(|\U|\to |\X|\) is quasi-compact.
    Let \(C\subset |\X|\) be a quasi-compact open subset.
    Denote the underlying \(\infty\)-topos of \(\X\) by \(\mathcal{X}\).
    Since \(\mathcal{X}\) is spatial by \cite[Rmk.1.6.6.1]{sag}, \(C\subset |X|=|\mathcal{X}|\) corresponds to a \((-1)\)-truncated and quasi-compact object \(c\in\mathcal{X}\).
    As \(i\) is an open immersion, there is a \((-1)\)-truncated object \(u\in\mathcal{X}\) such that the underlying geometric morphism of \(\infty\)-topoi of \(i\) can be identified with the forgetful functor \(\mathcal{X}_{/u}\to \mathcal{X}\).
    That \(i\) is quasi-compact implies, by definition, that \(\mathrm{pr}_2\colon c\times u\to u\) is a quasi-compact object of \(\mathcal{X}_{/u}\), that is that
    \[(\mathcal{X}_{/u})_{/(\mathrm{pr}_2\colon c\times u\to u)}\simeq \mathcal{X}_{/c\times u}\]
    is quasi-compact.
    In particular, the space \(|\mathcal{X}_{/c\times u}|\simeq C\cap |U|\) is quasi-compact.
\end{proof}

Recall from \cite[Df.5.5.3]{temperedglobal}, that closed immersion in $\Stk$ is an affine map $\Y \to \X$ such that when pulled back against any $\Spec A \to \X$, the induced map $\Spec B \to \Spec A$ induces a surjection on $\pi_0$.

\begin{prop}\label{lm:h_closed_immersion}
    If $j \colon \sfZ \to \sfX$ is a closed immersion in $\SpDM^\nc$, then $h_j \colon h_{\sfZ} \to h_\sfX$ is a closed immersion in $\Stk$.
\end{prop}

\begin{proof}
    First, note that $h_j$ is affine by \Cref{cor:h_affine_map}. Let $\X \simeq \colim \Spet A$ be a presentation of $\X$ as a colimit of étale affine. As $h_{(-)}$ preserves such colimits by \cite[Pr.3.2.2.2(2)]{reconstruction}, we also have $h_\X \simeq \colim \Spec A$ in $\Stk$. By \cite[Lm.A.12]{Tokic_2026_FamilyCompletionTheorem}, the canonical map $\coprod \Spec A \to h_\X$ is an effective epimorphism in $\Stk$. By \cite[Lm.A.18]{Tokic_2026_FamilyCompletionTheorem}, to show that $h_j$ is a closed immersion, it suffices to show that its base change to each $\Spec A$ is a closed immersion. As $h_j$ is affine, we see that the pullback of $j$ to $\Spec A$ takes the form $\Spec B \to \Spec A$. By \Cref{lem:h_map_from_affine_local_lift}, there is a faithfully flat $A$-algebra $A'$ such that the composite $\Spec A' \to h_\X$ lifts to $\Spet A' \to \X$ in $\SpDM^\nc$. Applying \Cref{lem:h_finite_limits}, we also see that the pullback of $j$ along $\Spet A'$ is given by $\Spet B'$ for some faithfully flat $B$-algebra $B'$. As $j$ is a closed immersion, $\pi_0 A' \to \pi_0 B'$ is surjective, and since closed immersions of classical schemes satisfy fpqc descent, \cite[Tag \href{https://stacks.math.columbia.edu/tag/03I0}{03I0}]{stacks}, we are done.
\end{proof}

\section{Symmetric powers of classical curves}\label{sec:appendix_classical_fun_stuff}
For this appendix, let \(M\) be a classical Noetherian scheme, and \(p\colon E\to M\) be a classical elliptic curve with zero section \(e\colon M\to E\) as in \cite[\textsection13.1.1]{olsson} or \cite[\textsection2.1.1]{km}, for example. In particular, \(E\) is an abelian group scheme over $M$ with unit \(e\) and $p$ is flat, proper, smooth, and projective. We will use that actions of finite groups on quasi-projective \(M\)-schemes admit geometric quotients \cite[Rmk.4.9(5)]{Rydh_2013_ExistencePropertiesGeometric}, which are themselves quasi-projective $M$-schemes, and that these geometric quotients are categorical quotients \cite[Rmk.2.8]{Rydh_2013_ExistencePropertiesGeometric}, in other words, quotients in the category of schemes.

By \cite[Tag \href{https://stacks.math.columbia.edu/tag/0B9A}{0B9A}]{stacks}, the Hilbert functor of $n$-points on $E$ is representable by a scheme
\[p_n \colon \mathrm{Hilb}^n_{E/M}\to M\]
over \(M\). These $p_n$ are separated (\cite[Tag \href{https://stacks.math.columbia.edu/tag/0B98}{0B98}]{stacks}) and smooth of relative dimension \(n\) (\cite[Tag \href{https://stacks.math.columbia.edu/tag/0B9H}{0B9H}]{stacks}).
Moreover, by \cite[Tag \href{https://stacks.math.columbia.edu/tag/0B9G}{0B9G}]{stacks}, there are morphisms
\[f_n\colon E^{\times_M\, n}\to\mathrm{Hilb}^n_{E/M}\]
over \(M\) that are finite locally free of degree \(n!\); this will turn out (\ref{eq:symprod_Hilbert_schemes}) to be the quotient map by the obvious $\Sigma_n$-action on the domain. By \cite[Tags \href{https://stacks.math.columbia.edu/tag/03GN}{03GN}]{stacks}, we conclude that \(p_n\) is proper. For \(n\geq 0\), denote the relative \(n\)th symmetric power of \(E\) by
\[S^n(E/M)=(E\times_M\cdots\times_M E)/\Sigma_n.\]
Since the group structure of \(E\) is commutative, we obtain summation maps
\[S^n(E/M)\xrightarrow{\sigma_n}E.\]
The maps \(f_n\) factor over the quotient maps as
\begin{equation}\label{eq:symprod_Hilbert_schemes}q_n\colon S^n(E/M)\to \mathrm{Hilb}^n_{E/M},\end{equation}
which are equivalences by \cite[XVII.6.3.9]{sgafour}; this relies on $E$ being a curve over $M$. It follows that the summation maps \(\sigma_n\) are proper (\cite[Tag \href{https://stacks.math.columbia.edu/tag/01W6}{01W6}]{stacks}), of finite presentation (\cite[Tag \href{https://stacks.math.columbia.edu/tag/02FV}{02FV}]{stacks}), and even flat (\cite[Tag \href{https://stacks.math.columbia.edu/tag/036K}{036K}]{stacks}). Consider the transition maps
\begin{equation}\label{eq:transition_maps} i_n\colon \mathrm{Hilb}^n_{E/M} \xrightarrow{\id\times e} \mathrm{Hilb}^n_{E/M} \times_ME \xrightarrow{t_n} \mathrm{Hilb}^{n+1}_{E/M},\end{equation}
where \(t_n\) is map constructed in \cite[Tag \href{https://stacks.math.columbia.edu/tag/0B9G}{0B9G}]{stacks}; informally, the map $t_n$ is given by sending a pair $((x_1, \ldots, x_n), y)$ to the $(n+1)$-tuple $(x_1,\ldots, x_n, y)$. As in the proof of \cite[Tag \href{https://stacks.math.columbia.edu/tag/0B9G}{0B9G}]{stacks} we see that
\[\mathrm{Hilb}^n_{E/M} \times_M E \xrightarrow{(t_n,\mathrm{pr_2})} \mathrm{Hilb}^{n+1}_{E/M} \times_M E\]
is the universal closed subscheme parametrised by \(\mathrm{Hilb}^{n+1}_{E/M}\).
In particular, it is an effective Cartier divisor relative to \(\mathrm{Hilb}^{n+1}_{E/M}\) courtesy of \cite[Tag \href{https://stacks.math.columbia.edu/tag/0B9D}{0B9D}]{stacks}.
It is also an effective Cartier divisor relative to \(E\), via the projections to the second factor.
The base change of this along \(e\colon M\to E\)
is given by $i_n$ from (\ref{eq:transition_maps}), which is thus an effective Cartier divisor relative to \(M\) by \cite[Tag \href{https://stacks.math.columbia.edu/tag/056Q}{056Q}]{stacks}. This map also commutes with the summation maps \(\sigma_n\), so it is also an effective Cartier divisor relative to \(E\). This proves the first half of the following statement:

\begin{prop}\label{prop:Hilb_Schemes_ECD_affine_complement}
    For $n\geq 1$, the transition map (\ref{eq:transition_maps})
    \[i_n \colon \mathrm{Hilb}^n_{E/M}\to \mathrm{Hilb}^{n+1}_{E/M}\]
    is an effective Cartier divisor, and its open complement is affine over \(E\).
\end{prop}

\begin{proof}
Denote the dual of the ideal sheaf of \(i_n\) by \(L_n\) and consider the situation after base change to a geometric point of \(M\). Note that, by construction, the formation of Hilbert schemes of points commutes with base changes. By \cite[Th.4 \& Pr.6]{Schwarzenberger_1963_JacobiansSymmetricProducts}, the maps \(\sigma_n\) are projective bundles with fibre the projective space \(\mathbf{P}^{n-1}\) for each $n\geq 2$.
Moreover, \cite[Pr.11]{Schwarzenberger_1963_JacobiansSymmetricProducts} identifies \(L_n\) with the canonical line bundle \(\O(1)\) on these projective bundles.
Thus, if \(M=\Spec \bar{\kappa}\) is the spectrum of an algebraically closed field $\bar{\kappa}$, then \(L_n\) is an ample line bundle relative to \(\sigma_{n+1}\), for \(n\geq 1\).

Returning to our original setting, we see that \(L_n\) is ample on all geometric fibres of \(\sigma_{n+1}\), for \(n\geq 1\).
It is then also ample on all fibres of \(\sigma_{n+1}\): indeed, let \(x\in|E|\) be a point, \(\bar{x}\colon \Spec \bar{\kappa(x)}\to E\) be the associated geometric point, and \(F\) be a coherent sheaf on \(\mathrm{Hilb}^{n+1}_{E/M}\vert_x\). Then \(F\vert_{\bar{x}}\) is coherent on \(\mathrm{Hilb}^{n+1}_{E/M}\vert_{\bar{x}}\). Since \(L_n\vert_{\bar{x}}\) is ample, we have that
\[H^q(\mathrm{Hilb}^{n+1}_{E/M}\vert_{\bar{x}},F\vert_{\bar{x}} \otimes (L_n\vert_{\bar{x}})^{\otimes m}) =0\]
for \(q>0\) and all large enough \(m\).
As \(\bar{x}\to x\) is flat, flat base change (\cite[Tag \href{https://stacks.math.columbia.edu/tag/02KH}{02KH}]{stacks}) gives
\[H^q(\mathrm{Hilb}^{n+1}_{E/M} \vert_{\bar{x}},F\vert_{\bar{x}} \otimes(L_n\vert_{\bar{x}})^{\otimes m}) = H^q(\mathrm{Hilb}^{n+1}_{E/M} \vert_{x},F \otimes (L_n\vert_{x})^{\otimes m}) \otimes_{k(x)} k(\bar{x}).\]
By \cite[Pr.2.6.1]{Grothendieck_1961_ElementsGeometrieAlgebrique}, we conclude that \(L_n\vert_x\) is ample, and by \cite[Cor.9.6.5]{Grothendieck_1966_ElementsGeometrieAlgebrique}, that \(L_n\) is ample relative to \(\sigma_{n+1}\).

Let \(U_{n+1}=\mathrm{Hilb}^{n+1}_{E/M}\setminus\mathrm{Hilb}^{n}_{E/M}\) be the open complement. The line bundle \(L_n\) comes with a canonical section which becomes an equivalence when restricted to \(U_{n+1}\), meaning that the line bundle \(L_n\vert_{U_{n+1}}\) is ample relative to \(\sigma_{n+1}\) (\cite[Tag \href{https://stacks.math.columbia.edu/tag/0B3E}{0B3E}]{stacks}), so
\[U_{n+1}\to \mathrm{Hilb}^{n+1}_{E/M}\xrightarrow{\sigma_{n+1}}E\]
is quasi-affine (\cite[Tag \href{https://stacks.math.columbia.edu/tag/0891}{0891}]{stacks}). Since \(M\) is Noetherian, \(\mathrm{Hilb}^{n+1}_{E/M}\) is also Noetherian (\cite[Cor.6.3.8]{Grothendieck_1960_ElementsGeometrieAlgebrique}), and consequently so is \(U_{n+1}\). In particular, every quasi-coherent sheaf of ideals \(I\subset \O_{U_{n+1}}\) is coherent. We have already seen that \(\O_{U_{n+1}}\) is ample relative to the composite
\[\iota_{n+1}\colon U_{n+1}\to \mathrm{Hilb}^{n+1}_{E/M}\xrightarrow{\sigma_{n+1}}E,\]
so \(R^1(\iota_{n+1})_*I=0\).
Since \(\iota_{n+1}\) is quasi-compact and separated, we see that \(\iota_{n+1}\) is affine by Serre's affineness criterion (\cite[Tag \href{https://stacks.math.columbia.edu/tag/01XF}{01XF}]{stacks}).
\end{proof}

We are also interested in a subobject of the symmetric product above. Let \(K_n\) be the pullback of \(s_n \colon E^{\times_M n}\to E\) along \(e \colon M\to E\), and denote the inclusion by \(\iota \colon K_n\to E^{\times_M n}\). The $M$-scheme $K_n$ inherits a $\Sigma_n$-action such that $\iota$ is $\Sigma_n$-equivariant. Moreover, as a closed subscheme of $E^{\times_M n}$, we see that $K_n$ is quasi-projective, so this $\Sigma_n$-action on $K_n$ admits a geometric quotient $K_n/ \Sigma_n$.

\begin{prop}\label{lem:SpecialUnitary_Schemes_ECD_affine_complement}
    For any $n\geq 2$, the commutative diagram of schemes
    \[\begin{tikzcd}
	{K_n/\Sigma_n} & {S^n(E/M)} \\
	M & E
	\arrow[from=1-1, to=1-2]
	\arrow[from=1-1, to=2-1]
	\arrow["{\sigma_n}", from=1-2, to=2-2]
	\arrow["e", from=2-1, to=2-2]
\end{tikzcd}\]
is a pullback.
\end{prop}

\begin{proof}
    Let us write $Y\to X$ for the $\Sigma_n$-quotient of the map $K_n \to E^{\times_M n}$ in fppf-sheaves of sets. By universality of colimits, we see that the diagram
    \[\begin{tikzcd}
        {Y}\ar[r]\ar[d]    &   {X}\ar[d]    \\
        {M}\ar[r]    &   {E}
    \end{tikzcd}\]
    is a pullback in fppf-sheaves of sets. It suffices then, to show that these quotients $Y$ and $X$ are actually schemes. Indeed, the fully faithfulness of the functor of points from the 1-category of classical schemes to fppf sheaves of sets would show that the schemes representing $Y$ and $X$ are the quotients in schemes $K_n/\Sigma_n$ and $S^n(E/M)$, respectively. This also would show the desired diagram is a pullback in schemes. To see that $Y$ and $X$ are schemes, we now apply Grothendieck's theorem for quotients of groupoids in schemes \cite[Th.V.4.1]{sga3_tomeI}, which applies as the group $\Sigma_n$ is finite and both $K_n$ and $E^{\times_M n}$ are quasi-projective.
\end{proof}

The elliptic curve comes with a natural action of \(C_2\), given by the inversion in the abelian group structure.
In the body of this article, we will also see the quotients
\[T^n(E/M)=(E\times_M\cdots\times_ME)/((C_2)^n\rtimes \Sigma_n)\]
arise. There is a natural identification
\[T^1(E/M)\simeq\mathbb{P}_M((p_*\O_E(2e))^\vee),\]
and \((p_*\O_E(2e))^\vee\) is a vector bundle of rank \(2\) on \(M\).
Thus, \(T^1(E/M)\) is a proper smooth curve over \(M\) of genus \(0\).

As \(p \colon E\to M\) is flat, \cite[Pr.2.10(ii)]{Rydh_2013_ExistencePropertiesGeometric} identifies
\[(E\times_M E)/C_2\simeq (E/C_2)\times_M E.\]
Applying this iteratively, we find that
\[(E\times_M\cdots \times_M E)/(C_2)^n\simeq T^1(E/M)\times_M\cdots \times_M T^1(E/M),\]
and consequently that
\[T^n(E/M)\simeq S^n(T^1(E/M)/M).\]
Again, by \cite[XVII.6.3.9]{sgafour}, we conclude that
\[T^n(E/M)\simeq \mathrm{Hilb}^n_{T^1(E/M)/M}.\]
Arguing as in \Cref{prop:Hilb_Schemes_ECD_affine_complement}, taking into account that our curve has relative genus 0, immediately yields:

\begin{prop}\label{prop:Symplectic_Schemes_ECD_affine_complement}
    For the $n\geq 1$, the map
    \[T^n(E/M)\to T^{n+1}(E/M)\]
    is an effective Cartier divisor over $M$ and its open complement is affine over \(M\).
\end{prop}

\section{Descendability for equivariant topological modular forms}\label{sec:appendix2}
The goal of this appendix is to prove the following statement:

\begin{theorem}\label{thm:descenable_maps_wls}
    Let $G$ be a compact Lie group. Then the morphisms of $\E_\infty$-$G$-rings
    \[\TMF_G[\tfrac{1}{3}] \to \TMF_1(3)_G \to \TMF(3)_G, \qquad \TMF_G[\tfrac{1}{2}] \to \TMF_1(2)_G \to \TMF(2)_G\]
    all admit descent. In particular, the map of $\E_\infty$-$G$-rings
    \[\TMF_G \to \TMF(2)_G \times \TMF(3)_G\]
    admits descent.
\end{theorem}

\begin{remark}\label{rmk:alternative_proof}
This immediately leads to a another proof of \Cref{thm:step_to_universal_case}, as $\TMF(2) \times \TMF(3)$ is a strictly even-periodic Noetherian $\E_\infty$-ring. Indeed, by \cite[Pr.3.22]{mathew_Galois} and writing $R_G^{\bullet}$ for the Čech nerve of the map $\TMF_G \to \TMF(2)_G \times \TMF(3)_G$, the result above already identifies the categories
\[\Mod_{\TMF_G}(\Sp_G) \simeq \lim \Mod_{R_G^\bullet}(\Sp_G),\]
meaning that if we have equivariant affineness for each strictly even-periodic Noetherian affine, such as the underlying $\E_\infty$-ring $R^\bullet$ of $R_G^\bullet$, then we have the desired universal equivalence
\[\QCoh(\calE_G) \simeq \lim \QCoh(\sfE_G^\bullet) \simeq \lim \Mod_{R_G^\bullet}(\Sp_G) \simeq \Mod_{\TMF_G}(\Sp_G),\]
where $\sfE_G^\bullet$ is the base change of $\calE_G$ along $\Spec R^\bullet \to \M_\Ell^\ori$.
\end{remark}

We include \Cref{thm:descenable_maps_wls} here as an alternative approach to equivariant affineness that might be helpful in future research.

To prove this result, we will use the following finite spectra. Let $DA(1)$ be the \emph{double of $A(1)$}, the $8$-cell $2$-local finite spectrum of \cite[Df.4.2]{tmfhomology}, and $X_3$ be the $3$-cell $3$-local finite spectrum of \cite[Pr.4.12]{tmfhomology}.

\begin{prop}\label{pr:equivariant_woods_theorem}
    Let $G$ be a compact Lie group. Then there are equivalences of $\TMF_G$-modules
    \[\TMF_G[\tfrac{1}{3}] \otimes DA(1) \simeq \TMF_1(3)_G, \qquad \TMF_G[\tfrac{1}{2}] \otimes X_3 \simeq \TMF_1(2)_G.\]
\end{prop}

This result appears in unpublished work of Sanath Devalapurkar without a direct proof. Both our proof below and his suggested proof build upon the nonequivariant case, a famous result of Hopkins--Mahowald, proven by Mathew in \cite[Th.1.2]{tmfhomology}.

\begin{proof}
    By \cite[Th.1.2]{tmfhomology}, we have the nonequivariant $2$-local Hopkins--Mahowald equivalence of $\TMF$-modules $\TMF \otimes DA(1) \simeq \TMF_1(3)$. Note that one only need invert $3$ for this equivalence, ie.~it lifts to a $\TMF[\tfrac{1}{3}]$-module equivalence $\TMF[\tfrac{1}{3}] \otimes DA(1) \simeq \TMF_1(3)$. Write $\TMF_1(3)_G$ for $\Ga_G \sfE$ where $\sfE$ is the base change of the universal oriented elliptic curve along $p \colon \M_1^\ori(3) \to \M_\Ell^\ori$. One observes the equivalences
    \[\TMF_1(3)_G \simeq \TMF_G \otimes_\TMF \TMF_1(3) \simeq \TMF_G \otimes_\TMF \TMF[\tfrac{1}{3}] \otimes DA(1) \simeq \TMF_G[\tfrac{1}{3}] \otimes DA(1),\]
    the first equivalence coming from \Cref{rmk:base_change_for_general_guys}, courtesy of the fact that $p$ is affine hence $\M_1^\ori(3)$ is 0-affine by \Cref{ex:MM_mellor_affineness_intro}, and the second by the above Hopkins--Mahowald equivalence. The analogous argument holds for $\TMF_1(2)$ and $X_3$.
\end{proof}

\begin{remark}
    It immediately follows that similar equivalences hold for \emph{global topological modular forms} $\gTMF$, and its variants with level structures, defined using either \cite[Ex.5.1.7]{temperedglobal} or \cite[\textsection14]{gepner2024global2ringsgenuinerefinements}; this point holds in both cases. Indeed, a map of global $\gTMF[\tfrac{1}{3}]$-modules $\gTMF[\tfrac{1}{3}]\otimes DA(1) \to \gTMF_1(3)$ is precisely the data of a map of $\TMF$-modules $\TMF[\tfrac{1}{3}] \otimes DA(1) \to \TMF_1(3)$, so we choose the nonequivariant equivalence above. To check that this map is an equivalence, it suffices to check on underlying $G$-spectra for each compact Lie group $G$ (or finite group in the context of \cite{temperedglobal}), where this map is an equivalence from the proof of \Cref{pr:equivariant_woods_theorem}. A similar argument also recovers the equivalence of global $\gTMF[\tfrac{1}{2}]$-modules
    \[\gTMF[\tfrac{1}{2}] \otimes X_3 \simeq \gTMF_1(2).\]
    Many other statements in this article admit global analogues; we will leave that for another time.
\end{remark}

\begin{proof}[Proof of \Cref{thm:descenable_maps_wls}]
    Let us now focus on the first composite of morphisms; the argument for the second composite goes through with the obvious changes. First, we note that $f\colon \TMF_G[\tfrac{1}{3}] \to \TMF_1(3)_G$ is faithful. Indeed, for an object $M$ of $\Mod_{\TMF_G[\tfrac{1}{3}]}(\Sp_G)$ such that $\TMF_1(3)_G \otimes_{\TMF_G} M$ vanishes. Using the base change equivalence of \Cref{rmk:base_change_for_general_guys}, we identify this with
    \[0 = \TMF_1(3)_G \otimes_{\TMF_G[\tfrac{1}{3}]} M \simeq \TMF_1(3) \otimes_{\TMF[\tfrac{1}{3}]} M,\]
    from which it follows that $M=0$ as ${\TMF[\tfrac{1}{3}]} \to \TMF_1(3)$ is faithful in $\CAlg$. A similar argument shows that $\TMF_1(3)_G \to \TMF(3)_G$ is faithful using \Cref{ex:relative_base_change} this time.

    By \cite[Th.3.38]{mathew_Galois}, which states that faithful algebra maps with target compact over the domain admit descent, it suffices to show that $\TMF_1(3)_G$ is compact as a module over $\TMF_G[\tfrac{1}{3}]$ and that $\TMF(3)_G$ is compact over $\TMF_1(3)_G$. The first statement follows from \Cref{pr:equivariant_woods_theorem}, as $DA(1)$ is a finite $8$-cell complex. For the second statement, we use that $p\colon \M^\ori(3) \to \M_1^\ori(3)$ is finite flat (of rank $6 = |GL_2(\F_3)|/8$), so $p_\ast \O_{\M^\ori(3)}$ is dualisable, as this is true locally. As $\M_1^\ori(3)$ is 0-affine, we see that $p_\ast \O_{\M^\ori(3)} \simeq \TMF(3)$ is compact in $\QCoh(\M_1^\ori(3)) \simeq \Mod_{\TMF_1(3)}(\Sp)$, as desired. By base change, so \Cref{ex:relative_base_change}, it follows that $\TMF_1(3)_G \to \TMF(3)_G$ views the target as compact over the base.

    For the ``in particular'' statement, the map $A \to A[\tfrac{1}{2}] \times A[\tfrac{1}{3}]$ admits descent for any $\E_\infty$-$G$-ring by the usual pullback $A \simeq A[\tfrac{1}{2}] \times_{A[\tfrac{1}{6}]} A[\tfrac{1}{3}]$. The two composites admitting descent above then combine with this to show that $\TMF_G \to \TMF(2)_G \times \TMF(3)_G$ admits descent.
\end{proof}

\begin{remark}\label{rmk-g-stacks}
    We also hope that these kinds of descendability results might help in future work studying \emph{genuinely equivariant stacks}, perhaps based on $\CAlg(\Sp_G)$ or equipped with the finitary \emph{universal descent topology}, where covers are morphisms admitting descent; this appears in \cite[Rmk.2.20]{Davies_galois_geometric_FP} when $G$ is the trivial group.
\end{remark}

\addcontentsline{toc}{section}{References}
\scriptsize
\bibliography{references} 
\bibliographystyle{alpha}
\end{document}